\documentclass[11pt]{amsart}
\usepackage{amssymb,latexsym}
\usepackage{amsthm}
\usepackage{amsmath}
\usepackage{amsfonts}
\usepackage{graphicx}
\usepackage[all]{xy}
\usepackage{fancyhdr}
\usepackage[top=1in,bottom=1in,left=1.25in,right=1.25in]{geometry}

\newtheorem{theorem}{Theorem}[section]
\newtheorem{definition}[theorem]{Definition}
\newtheorem{remark}[theorem]{Remark}
\newtheorem{proposition}[theorem]{Proposition}
\newtheorem{corollary}[theorem]{Corollary}
\newtheorem{lemma}[theorem]{Lemma}
\newtheorem{claim}[theorem]{Claim}
\newtheorem{example}[theorem]{Example}

\def\Q{\mathbb{Q}}
\def\F{\mathbb{F}}
\def\R{\mathbb{R}}
\def\Z{\mathbb{Z}}
\def\A{\mathbb{A}}

\def\Cm{\mathbb{C}}
\def\B{\mathcal{B}}
\def\C{\mathcal{C}}
\def\D{\mathcal{D}}
\def\I{\mathcal{I}}

\def\M{\mathcal{M}}
\def\N{\mathcal{N}}
\def\P{\mathcal{P}}
\def\Fm{\mathcal{F}}
\def\Lam{\Lambda}
\def\lan{\langle}
\def\ran{\rangle}

\def\lra{\longrightarrow}
\def\ra{\rightarrow}
\def\ov{\overline}
\def\wh{\widehat}
\def\st{\stackrel}

\usepackage[pdftex]{hyperref}
\usepackage{fancyhdr}

\begin{document}

\title{Cell decomposition of some unitary group Rapoport-Zink spaces}
\author{Xu SHEN}
\date{}
\address{Mathematisches Institut\\
Universit\"at Bonn\\
Endenicher Allee 60\\
53115, Bonn, Germany} \email{shen@math.uni-bonn.de}

\begin{abstract}
In this paper we study the $p$-adic analytic geometry of the basic unitary group Rapoport-Zink spaces $\M_K$ with signature $(1,n-1)$. Using the theory of Harder-Narasimhan filtration of finite flat groups developed in \cite{F2},\cite{F3}, and the Bruhat-Tits stratification of the reduced special fiber $\M_{red}$ defined in \cite{VW}, we find some relatively compact fundamental domain $\D_K$ in $\M_K$ for the action of $G(\Q_p)\times J_b(\Q_p)$, the product of the associated $p$-adic reductive groups, and prove that $\M_K$ admits a locally finite cell decomposition. By considering the action of regular elliptic elements on these cells, we establish a Lefschetz trace formula for these spaces by applying Mieda's main theorem in \cite{Mi2}.
\end{abstract}
\maketitle

\tableofcontents

\section{Introduction}

The general motivation of this article is the realization of local Langlands correspondences in the cohomology of Rapoport-Zink spaces, see \cite{R} and \cite{Ha}. These spaces are local analogues of PEL type Shimura varieties, and they uniformize some parts of these type Shimura varieties when passing to formal completion and $p$-adic analytification.

 The most well known Rapoport-Zink spaces are the Lubin-Tate spaces and Drinfeld spaces. In \cite{C}, it is conjectured the cohomology of Lubin-Tate spaces realizes the local Langlands and Jaquet-Langlands correspondences for $GL_n$. This has been essentially proved by Harris-Taylor in \cite{HT}, and completed by many other authors. In \cite{Fa2} and \cite{F}, Faltings and Fargues has established an isomorphism between the Lubin-Tate tower and the Drinfeld tower, and deduced also an isomorphism of the cohomology of the two towers. Thus the cohomology of the tower of Drinfeld spaces also realizes the local Langlands and Jaquet-Langlands correspondences for $GL_n$, as predicted originally by Drinfeld and partly realized by Harris \cite{Ha1}.

The description of the cohomology of some other Rapoport-Zink spaces in terms of irreducible smooth representations of the underlying $p$-adic reductive groups, has been done successfully by Fargues in \cite{F1} and Shin in \cite{Shin}. Both of them use global methods as that of Harris-Taylor, although their approaches are quite different: the former uses heavily rigid analytic geometry while the later is based on the stable trace formula. Their results are both about Rapoport-Zink spaces of EL type, and Fargues has also obtained results of the Rapoport-Zink spaces for $GU(3)$, based on the complete classfication of automorphic representations for unitary groups in three variables in \cite{Ro}. It would be nice if one can prove these local results by local methods. This will require a careful study of the geometry of Rapoport-Zink spaces, and then pass to cohomological applications. Some works in this direction are as \cite{Fa1}, \cite{St}, \cite{Y}.
 \\

In \cite{F}, the first step of the construction of an isomorphism between the towers of Lubin-Tate and Drinfeld, is by ``$p$-adicfy'' the Lubin-Tate tower. This $p$-adicfy procedure is to glue some formal models of Gross-Hopkins's fundamental domain.

To be precise, let $\widehat{\M}_{LT}$ be the formal Lubin-Tate space over $SpfW(\ov{\F}_p)$ for $GL_n/\Q_p$ for simple in this introduction. Recall that for a scheme $S\in$Nilp$W$, a $S$-valued point of $\widehat{\M}_{LT}$ is given by a pair $(H,\rho)$, with $H$ a one dimensional formal $p$-divisible group over $S$, and $\rho:\mathbf{H}_{\ov{S}}\ra H_{\ov{S}}$ is a quasi-isogeny. Here $W=W(\ov{\F}_p)$, Nilp$W$ is the category of schemes $S$ over $SpecW$ such that $p$ is locally nilpotent over $S$, $\ov{S}$ is the closed subscheme of $S$ defined by $p$, and $\mathbf{H}$ is the unique one dimensional formal $p$-divisible group of height $n$ over $\ov{\F}_p$. This space decomposes as a disjoint union of open and closed formal subschemes according to the height of quasi-isogeny. The associated $p$-adic Lubin-Tate space (in the sense of Berkovich) $\M_{LT}=\coprod_{i\in\Z}\M_{LT}^i$ admits an action by $GL_n(\Q_p)\times D^\times$, here $D$ is the division algebra of invariant $\frac{1}{n}$ over $\Q_p$. The action of $D^\times$ is just changing the quasi-isogeny, while the action of $GL_n(\Q_p)$ is a little complicated: it is defined by the Hecke correspondences, see \cite{RZ} or section 2 of this article for details.

There is a $p$-adic period mapping
\[\pi: \M_{LT}\ra \mathbf{P}^{n-1,an}\] of $p$-adic analytic spaces over $L:=W(\ov{\F}_p)_\Q$, and the fibers of this mapping is exactly the Hecke orbits on $\M_{LT}$. This reveals the very difference of the theories of uniformization of Shimura varieties between the complex and $p$-adic situation. Moreover, since by de Jong \cite{dJ2} $\pi$ is an \'{e}tale covering of $p$-adic analytic spaces, its fibers, i.e. the $p$-adic Hecke orbits, are thus discrete. This is also quite different to the situation over positive characteristic for non basic Newton polygon strata and the prime to $p$ Hecke orbits on Shimura varieties, see \cite{Ch} and \cite{Ch1} for example.

The fundamental domain of Gross-Hopkins is then given by \[\D:=\{x=(x_1,\dots,x_{n-1})\in\M_{LT}^0|v_x(x_i)\geq 1-\frac{i}{n},\,\forall\, 1\leq i\leq n-1\},\]where we use the coordinates $(x_1,\dots,x_{n-1})$ on $\M_{LT}^0$ (which is isomorphic to the open unit ball of dimension $n-1$) such that the Newton polygon of $H_x[p]$ for a point $x=(x_1,\dots,x_{n-1})\in\M_{LT}^0$ is given by the convex envelope of $(p^i,v(x_i))_{0\leq i\leq n}$ with $x_0=0, x_n=1$. Here $H_x$ is the $p$-divisible group associated to $x$,
$v_x$ is the valuation on the complete residue field of $x$.
Let $\Pi\in D^\times$ be an uniformizer such that it induces an isomorphism between the components
$\Pi^{-1}:\M_{LT}^i\st{\sim}{\ra}\M_{LT}^{i+1}$, then the domain $\D$ is such that we have a locally finite covering of the Lubin-Tate space
\[\M_{LT}=\bigcup_{\begin{subarray}{c}T\in GL_n(\Z_p)\setminus GL_n(\Q_p)/GL_n(\Z_p)\\i=0,\dots,n-1\end{subarray}}T.\Pi^{-i}\D.\]
Note $\D$ is closed, and more importantly (and non trivially) its underlying topological space is compact. The locally finiteness means that we can start from $\D$ and its translations $T.\Pi^{-i}\D$ to glue a Berkovich space, which is isomorphic to our Lubin-Tate space $\M_{LT}$. We may call such $\M_{LT}$ admits a cell decomposition, as an analogue of the classical situation. In \cite{Sh1} we have used this locally finite cell decomposition and the compactness of $\D$, to deduce a Lefschetz trace formula for Lubin-Tate spaces, by applying Mieda's theorem 3.13 in \cite{Mi2}.
\\

In \cite{F2} Fargues has developed a theory of Harder-Narasimhan filtration for finite flat group schemes, and applied to the study of $p$-divisible groups in \cite{F3}. For the details of Harder-Narasimhan filtrations see \cite{F2} or section 2 in the following for a review. In particular we have notions of semi-stable finite flat group schemes and $p$-divisible groups over a complete rank one valuation ring $O_K|\Z_p$. The basic observation is that, the Gross-Hopkins's fundamental domain is exactly the semi-stable locus in the connected component $\M_{LT}^0$, that is the locus where the associated $p$-divisible groups are semi-stable. Motivated by this fact, Fargues has studied fundamental domains in the Rapoport-Zink spaces for $GL_h/\Q_p$ with signature $(d,h-d)$, in particular there is no additional structures for the $p$-divisible groups considered.

There are two main ingredients in the article \cite{F3}. The first is an algorithm based the theory of Harder-Narasimhan filtrations of finite flat group schemes, which starts from any $p$-divisible groups over an $O_K$ as above and produce new ones which are more and more tend to be of HN-type, that is semi-stable for formal $p$-divisible groups whose special fiber is supersingular, see loc. cit. for the precise definition of $p$-divisible groups of HN-type. When the valuation $K$ is discrete, the algorithm stops after finite times. Passing to the Shimura varieties which give locally the Rapoport-Zink spaces for $GL_h/\Q_p$ with signature $(d,h-d)$, one can define a Hecke-equivarient stratification of the underlying topological space of these $p$-adic analytic Shimura varieties by Harder-Narasimhan polygons. The algorithm stops after finite times over complete discrete valuation rings means that, the Hecke orbits of the rigid points in the HN-type locus in each Harder-Narasimhan polygon strata, cover all the rigid points in the strata. For the basic polygon $\P_{ss}$ that is the line between the point $(0,0)$ and $(h,d)$ ($d$ is the dimension of $p$-divisible groups in the Rapoport-Zink spaces), the HN-type locus is exactly the semi-stable locus, and one has the statement as above.

The second main ingredient of \cite{F3} is the inequality \[HN(H)\leq Newt(H_{k})\] between the concave Harder-Narasimhan and Newton polygons, here $k$ is the residue field of $K$. The proof of this inequality for the case the valuation of $K$ is discrete is easy, while for the general case it is quite involved: Fargues has used the notions of Hodge-Tate modules and Banch-Colmez spaces in $p$-adic Hodge theory, and in fact one has also to pose a mild condition on $H$ in this case, which is naturally satisfied when $H$ coming from a point in Rapoport-Zink spaces. The moduli consequences of this inequality are that, the basic Newton polygon strata of the $p$-adic Shimura varieties is contained in their basic Harder-Narasimhan polygon strata, and the Hecke orbit of the semi-stable locus in the basic Rapoport-Zink space cover at least all the rigid points.

For the case $h$ and $d$ are co-prime to each other, Fargues can prove that the Hecke orbit of the semi-stable locus in the basic Rapoport-Zink covers all the space. More precisely, the main theorem of \cite{F3} is the following.
\begin{theorem}[Fargues, \cite{F3}]
Let $\M^{ss}\subset\M$ be the semi-stable locus in the basic $p$-adic Rapoport-Zink space $\M=\coprod_{i\in\Z}\M^i$ for $GL_h/\Q_p$ with signature $(d,h-d)$, and $\D:=\M^{ss}\bigcap \M^0$. Assume $(h,d)=1$. Let $\Pi\in D^\times$ be an uniformizer in $J_b(\Q_p)=D^\times$, where $D$ is the division algebra of invariant $\frac{d}{h}$ over $\Q_p$, such that $\Pi$ induces isomorphisms $\Pi^{-1}:\M^i\ra\M^{i+1}$. Then there is a locally finite covering of $\M$
\[\M=\bigcup_{\begin{subarray}{c}T\in GL_h(\Z_p)\setminus GL_h(\Q_p)/GL_h(\Z_p)\\i=0,\dots,h-1\end{subarray}}T.\Pi^{-i}\D.\]
\end{theorem}
For the case $d=1$ we recover the cell decomposition of Lubin-Tate space.
\\

The purpose of this article is to prove a similar result of cell decomposition for some unitary group Rapoport-Zink spaces.

More precisely, let $p>2$ be a fixed prime number, $\Q_{p^2}$ be the unramified extension of $\Q_p$ of degree $2$, and $G$ be the quasi-split unitary similitude group which is defined by a $n$-dimensional $\Q_{p^2}$ hermitian space. The basic formal Rapoport-Zink space $\widehat{\M}$ for $G$ with signature $(1,n-1)$ is the formal scheme formally locally of finite type over $SpfW$. A $S$-valued point of $\widehat{\M}$ for a $S\in$NilpW is given by $(H,\iota,\lambda,\rho)$, where $H$ is a $p$-divisible group over $S$, $\iota: \Z_{p^2}\ra End(H)$ is an action of $\Z_{p^2}$ on $H$ satisfying
certain Kottwitz type determinant conditions, $\lambda: H\ra H^D$ is a polarization compatible with the action $\iota$, and $\rho: \mathbf{H}_{\ov{S}}\ra H_{\ov{S}}$ is a quasi-isogeny. For more details see the following section 2. One has a decomposition $\widehat{\M}=\coprod_{i\in\Z,in \,even}\widehat{\M}^i$, where $\widehat{\M}^i$ is the locus where the height of the quasi-isogenies is $in$.

The geometry of the reduced special fiber $\M_{red}$ has been completely described by Vollaard-Wedhorn in \cite{VW}. It turns out each connected component $\M_{red}^i$ admits a Bruhat-Tits stratification
\[\M_{red}^i=\coprod_{\Lam\in \mathcal{B}(J_b^{der},\Q_p)^0}\M_\Lam^0,\]where $\mathcal{B}(J_b^{der},\Q_p)^0$ is the set of vertices in the Bruhat-Tits building of the derived subgroup $J_b^{der}$ of $J_b$ over $\Q_p$, and $\M_\Lam^0$ is a locally closed subscheme. Recall $J_b$ is the inner form of $G$ associated to the local data to define the Rapoport-Zink space, and when $n$ is odd we have in fact an isomorphism $G\simeq J_b$. There is a type function $t: \mathcal{B}(J_b^{der},\Q_p)^0\ra [1,n]$, which takes values on all the odd integers between 1 and $n$, and the fibers of $t$ are exactly the $J^{der}(\Q_p)$-orbits in $\mathcal{B}(J_b^{der},\Q_p)^0$. Let $t_{max}=n$ if $n$ is odd and $t_{max}=n-1$ if $n$ is even. Then the irreducible components of $\M_{red}^i$ are exactly these $\M_\Lam$, the schematic closure of $\M_\Lam^0$, with $t(\Lam)=t_{max}$.

Let $g_1\in J_b(\Q_p)$ be an element such that it induces isomorphisms $g_1:\widehat{\M}^i\st{\sim}{\ra}\widehat{\M}^{i+1}$ for $n$ even and $g_1:\widehat{\M}^i\st{\sim}{\ra}\widehat{\M}^{i+2}$ for $n$ odd. The element $p^{-1}\in J_b(\Q_p)$ induces always isomorphisms $p^{-1}: \widehat{\M}^i\ra\widehat{\M}^{i+2}$. We fix a choice of $g_1$ compatible with $p^{-1}$. Consider the connected component for $i=0$ and fix a choice of $\Lam$ such that $t(\Lam)=t_{max}$, and let $Stab(\Lam)$ be the stabilizer group of $\Lam$ in $J^{der}(\Q_p)$.
Let $\M=\coprod_{i\in\Z,in\, even}\M^i$ be the associated Berkovich analytic space of $\widehat{\M}$, and $sp: \M\ra\M_{red}$ be the specialization map.
\begin{theorem}
There is a closed analytic domain $\C\subset\M$, which contains the semi-stable locus $\M^{ss}$, such that if we set
\[\D:=\C\bigcap sp^{-1}(\M_\Lam),\]then $\D$ is relatively compact. Moreover, we have a locally finite covering of $\M$
\[\M=\bigcup_{\begin{subarray}{c}T\in G(\Z_p)\setminus G(\Q_p)/G(\Z_p)\\g\in J^{der}_b(\Q_p)/Stab(\Lam)\end{subarray}
}T.g\D\] if $n$ is odd, and
\[\M=\bigcup_{\begin{subarray}{c}T\in G(\Z_p)\setminus G(\Q_p)/G(\Z_p)\\j=0,1\\g\in J^{der}_b(\Q_p)/Stab(\Lam)\end{subarray}
}T.g_1^jg\D
\]if $n$ is even.
\end{theorem}

Let $\pi:\M\ra \Fm^a$ be the $p$-adic period mapping (see section 2), $K\subset G(\Z_p)$ be an open compact subgroup, $\M_K$ be the Rapoport-Zink space with level $K$ and $\pi_K:\M_K\ra \M$ be the natural projection. Denote by $\D_K$ the inverse image of $\D$ under $\pi_K$. Then we have the following corollaries.
\begin{corollary}
 We have a locally finite covering of $\Fm^a$
\[\Fm^a=\bigcup_{g\in J^{der}_b(\Q_p)/Stab(\Lam)}g\pi(\D)\] if $n$ is odd, and
\[\Fm^a=\bigcup_{\begin{subarray}{c}j=0,1\\g\in J^{der}_b(\Q_p)/Stab(\Lam)\end{subarray}}g_1^jg\pi(\D)\]if $n$ is even.
\end{corollary}
\begin{corollary}
We have a locally finite covering of the analytic space $\M_K$
\[\M_K=\bigcup_{\begin{subarray}{c}T\in G(\Z_p)\setminus G(\Q_p)/K\\g\in J^{der}_b(\Q_p)/Stab(\Lam)\end{subarray}
}T.g\D_K\] if $n$ is odd, and
\[\M_K=\bigcup_{\begin{subarray}{c}T\in G(\Z_p)\setminus G(\Q_p)/K\\j=0,1\\g\in J^{der}_b(\Q_p)/Stab(\Lam)\end{subarray}
}T.g_1^jg\D_K
\]if $n$ is even.
\end{corollary}

Finally we have a corollary for Shimura varieties.
\begin{corollary}
Let $Sh_{K^p}$ be as the Shimura variety introduced in section 9, $\widehat{Sh}^{an}_{K^p}$ be the generic analytic fiber of its $p$-adic completion $\widehat{Sh}_{K^p}$, and $\widehat{Sh}^{an,b_0}_{K^p}$ be the tube in $\widehat{Sh}^{an}_{K^p}$ over the basic strata $\ov{Sh}_{K^p}^{b_0}$, which is an open subspace. Let $\widehat{Sh}^{an}_{K_p\times K^p}\ra \widehat{Sh}^{an}_{K^p}$ be the covering in level $K_p\subset G(\Z_p)$ (an open compact subgroup), and $\widehat{Sh}^{an,b_0}_{K_p\times K^p}$ be the inverse image of $\widehat{Sh}^{an,b_0}_{K^p}$.  Denote $\C^i=\C\bigcap\M^i$ for each $i\in\Z$ such that $in$ is even, $\C'=\C^0$ if $n$ is odd and $\C'=\C^0\coprod \C^1$ if $n$ is even, $\C'_{K_p}$ the inverse image of $\C'$ in $\M_{K_p}$, $\mathcal{E}'_{K_p}$ the image of $\C'_{K_p}$ under the $p$-adic uniformization \[ I(\Q)\setminus\M_{K_p}\times G(\A^p_f)/K^p\simeq \coprod_{i\in I(\Q)\setminus G(\A^p_f)/K^p}\M_{K_p}/\Gamma_i\simeq \widehat{Sh}^{an,b_0}_{K_p\times K^p}.\]
\begin{enumerate}

\item Let $\Gamma=\Gamma_i$ be one of the above discrete, torsion free, cocompact modulo center subgroups of $J_b(\Q_p)$, and $\Gamma^{der}=\Gamma\cap J^{der}_b(\Q_p)$, $D_{K_p}=D_{iK_p}$ be the image of $\D_{K_p}$ under the morphism $\M_{K_p}\ra\M_{K_p}/\Gamma$, then we have a covering
    \[\M_{K_p}/\Gamma=\bigcup_{\begin{subarray}{c}T\in G(\Z_p)\setminus G(\Q_p)/K_p\\ g\in\Gamma^{der}\setminus J^{der}_b(\Q_p)/Stab(\Lam)\end{subarray}}T.gD_{K_p}\] if $n$ is odd, and
    \[\M_{K_p}/\Gamma=\bigcup_{\begin{subarray}{c}T\in G(\Z_p)\setminus G(\Q_p)/K_p\\j=0,1\\ g\in\Gamma^{der}\setminus J^{der}_b(\Q_p)/Stab(\Lam)\end{subarray}}T.g_1^jgD_{K_p}\] if $n$ is even.
\item Under the above notation, we have a covering \[\mathcal{E}'_{K_p}=\coprod_{i\in I(\Q)\setminus G(\A^p_f)/K^p}\bigcup_{g\in\Gamma^{der}\setminus J^{der}_b(\Q_p)/Stab(\Lam)}gD_{iK_p}\] if $n$ is odd, and \[\mathcal{E}'_{K_p}=\coprod_{i\in I(\Q)\setminus G(\A^p_f)/K^p}\bigcup_{\begin{subarray}{c}j=0,1\\g\in\Gamma^{der}\setminus J^{der}_b(\Q_p)/Stab(\Lam)\end{subarray}}g_1^jgD_{iK_p}\] if $n$ is even.  We have a covering
\[\widehat{Sh}^{an,b_0}_{K_p\times K^p}=\bigcup_{T\in G(\Z_p)\setminus G(\Q_p)/K_p}T.\mathcal{E}'_{K_p}.\]
\end{enumerate}

\end{corollary}

As one has seen, the first difficulty in our unitary group case is that, the geometry of the reduced special fiber of Rapoport-Zink space is more complicated than that for the case of $GL_h/\Q_p$ with signature $(d,h-d)$, since for the case $(h,d)=1$ considered above each connected component of the special fiber is already irreducible, see \cite{Ve1}. This is why we have to take the intersection of $\C$ with the tube over a fixed irreducible component to have a locally finite cell decomposition. The second difficulty is that, the algorithm above when applied to the PEL type Rapoport-Zink spaces, for example the unitary group case considered here, is not well compatible with the action of Hecke correspondences. One has to modify it. This is why the semi-stable locus may be not enough and we find a closed domain $\C\supset\M^{ss}$.

On the other hand, the inequality \[HN(H,\iota,\lambda)\leq Newt(H_k,\iota,\lambda)\]between the Harder-Narasimhan and Newton polygons for $p$-divisible groups with additional structures still holds. In fact this can be easily deduced from Fargues's inequality $HN(H)\leq Newt(H_k)$, since the former polygons are just defined respectively by normalization of the later polygons.

The general strategy to prove the above theorem, is that using the modified algorithm and the above inequality to deduce first the equalities in the theorem hold for rigid points. For the rest points, by the equivalence of suitable categories between Berkovich spaces and rigid analytic spaces, it suffices to prove these coverings are locally finite, thus admissible. This last argument is different from that in \cite{F3} section 16.

In our unitary case, we have in fact that the underlying topological space of $\D$ is locally compact, like the case of Lubin-Tate space. These two facts both come from the special phenomenon that, all the non basic Newton polygon has contacted points with the Hodge polygon, and thus one can deduce the Harder-Narasimhan polygon stratification and the Newton polygon stratification of the associated $p$-adic Shimura varieties coincide.

At this point we should note that, the Rapoport-Zink spaces for $GSp4$ is quite similar with our unitary case. See \cite{KR2} section 4 for a similar geometric description of the reduced special fiber. And the two non basic Newton polygons have contacted points with the Hodge polygon.  In particular our method here will enable us to prove an analogue result of cell decomposition for the basic $p$-adic $GSp4$ Rapoport-Zink spaces.
\\

In fact we can introduce some more natural parameter set for the cells, which is more convenient when considering the group actions. Consider the action of $\Q_p^\times$ on $G(\Z_p)\setminus G(\Q_p)/K\times J_b(\Q_p)/Stab(\Lam)$ through the imbedding $\Q_p^\times\ra G(\Q_p)\times J_b(\Q_p), z\mapsto (z,z)$, and set
\[\I_K:=(G(\Z_p)\setminus G(\Q_p)/K\times J_b(\Q_p)/Stab(\Lam))/\Q_p^\times,\]
which can be viewed as a subset of the quotient space $\B/K$, where $\B$ is the Bruhat-Tits building quotient by $\Q_p^\times$ (for the embedding $z\mapsto (z,z^{-1})$) $\B=\B(G\times J_b,\Q_p)/\Q_p^\times$.
For each $[T,g]\in\I_K$, the analytic domain $\D_{[T,g],K}:=T.g\D_K$ is well defined. We can rewrite the cell decomposition of $\M_K$ elegantly as
\[\M_K=\bigcup_{[T,g]\in\I_K}\D_{[T,g],K}.\]If $\gamma=(h,f)\in G(\Q_p)\times J_b(\Q_p)$ such that $hKh^{-1}=K$, $\gamma$ induces $\gamma:\M_K\ra\M_K, \I_K\ra\I_K$, which are compatible: $\gamma(\D_{[T,g],K})=\D_{\gamma([T,g]),K}=\D_{[Th,fg],K}$. Moreover, there is a homomorphism $
\varphi: \I_K\ra\Z$ with image $\Z$ if $n$ is even and $2\Z$ is $n$ is odd. For each $i\in\Z$ such that $in$ is even, let $\I_K^i=\varphi^{-1}(i)$, then we have
\[\M_K^i=\bigcup_{[T,g]\in\I_K^i}\D_{[T,g],K}.\] There is a metric $\ov{d}$ on $\I_K$, induced by the metric of $\B$. The locally finiteness can be read as, there exists some constant $c>0$ such that, for any $[T,g]\in\I_K$, we have
\[\{[T',g']\in\I_K|\D_{[T',g'],K}\bigcap\D_{[T,g],K}\neq \emptyset\}\subset\{[T',g']\in\I_K|\ov{d}([T,g],[T',g'])\leq c\},\]with the latter a finite set.

The induced locally finite cell decomposition of $\M_K^0$ and the locally compactness of $\D$, will enable us to establish a Lefschetz trace formula for these spaces, by applying Mieda's theorem 3.13 in \cite{Mi2}. The idea is similar with our work for Lubin-Tate case in \cite{Sh1}, by studying the action of $\gamma$ on the cells. For the definition of the subspaces $U_\rho\subset \M_K$ see section 11.
\begin{theorem}
For the fixed $\gamma=(h,g)\in G(\Q_p)\times J_b(\Q_p)$ with $h,g$ both regular elliptic semi-simple, and $v_p(deth)+v_p(detg)=0$, there exist a sufficient small open compact subgroup $K'\subset G(\Z_p)$ and a sufficient large number $\rho_0>>0$, such that for all open compact subgroups $K\subset G(\Z_p)$ contained in $K'$ and normalized by $h$, and all $\rho\geq \rho_0$,
we have the Lefschetz trace formula
\[Tr(\gamma|H_c^\ast(U_\rho\times\Cm_p,\ov{\Q}_l))=\#\textrm{Fix}(\gamma|\M_K^0\times\Cm_p),\]which is well defined and finite. Since the right hand side is independent of $\rho$, we can define
\[Tr(\gamma|H_c^\ast(\M_K^0\times\Cm_p,\ov{\Q}_l)):=Tr(\gamma|H_c^\ast(U_\rho\times\Cm_p,\ov{\Q}_l))\]for $\rho>>0$, and thus \[Tr(\gamma|H_c^\ast(\M_K^0\times\Cm_p,\ov{\Q}_l))=\#\textrm{Fix}(\gamma|\M_K^0\times\Cm_p).\]
\end{theorem}

We have a nice formula for the number of fixed points for the quotient space $\M_K/p^\Z$. Note if $g\in J_b(\Q_p)$ is a regular elliptic semi-simple element, for any $x\in \textrm{Fix}(g|\Fm^a(\Cm_p))$, there is a element $h_{g,x}\in G(\Q_p)$ which is conjugate to $g$ over $\ov{\Q}_p$ defined by the comparison isomorphism
\[V_p(H_y)\otimes_{\Q_p}B_{dR}\st{\sim}{\longrightarrow}V_L\otimes_LB_{dR},\]
where $y\in\pi^{-1}(x)$ is any point in the fiber of the $p$-adic period mapping $\pi:\M\ra\Fm^a$.
 \begin{corollary}
Let the notations be as in the above theorem. If $n$ is even we assume that $\frac{2}{n}(v_p(deth)+v_p(detg))$ is even. Fix compatible Haar measures on $G(\Q_p)$ and the centralizer of $h_{g,x}$, $G_{h_{g,x}}:=\{h'\in G(\Q_p)|h'h_{g,x}h'^{-1}=h_{g,x}\}$. Denote the characteristic function of $h^{-1}K$ by $1_{h^{-1}K}$ and the volume of $K$ under the fixed Haar measure by $Vol(K)$. Then we have the following formula
\[Tr(\gamma|H_c^\ast((\M_K/p^{\Z})\times\Cm_p,\ov{\Q}_l))=\sum_{x\in \textrm{Fix}(g|\Fm^a(\Cm_p))}Vol(G_{h_{g,x}}/p^{\Z})O_{h_{g,x}}(\frac{1_{h^{-1}K}}{Vol(K)}),\]
where $Vol(G_{h_{g,x}}/p^\Z)$ is the volume of $G_{h_{g,x}}/p^\Z$ by the induced Haar measure on $G(\Q_p)/p^\Z$,
\[O_{h_{g,x}}(\frac{1_{h^{-1}K}}{Vol(K)})=\int_{G(\Q_p)/G_{h_{g,x}}}\frac{1_{h^{-1}K}}{Vol(K)}(z^{-1}h_{g,x}z)dz\] is the orbit integral of $\frac{1_{h^{-1}K}}{Vol(K)}$ over the conjugate class of $h_{g,x}$.
\end{corollary}

Let $\pi$ be a supercuspidal representation of $G(\Q_p)$, we consider
\[H(\pi)=\sum_{j\geq 0}(-1)^j\textrm{Hom}_{G(\Q_p)}(\varinjlim_{K}H_c^j(\M_K\times\Cm_p,\ov{\Q}_l),\pi).\]
Assume that $\textrm{Hom}_{G(\Q_p)}(\varinjlim_{K}H_c^j(\M_K\times\Cm_p,\ov{\Q}_l),\pi)$ is of finite length for each $j\geq0$, which should be always the case, then $H(\pi)$ is a well defined element in $\textrm{Groth}_{\ov{\Q}_l}(J_b(\Q_p))$.
\begin{corollary}
 Let $g\in J_b(\Q_p)$ be a regular elliptic semi-simple element. Assume that $\pi$ is of the form $\pi=c-\textrm{Ind}_{K_\pi}^{G(\Q_p)}\lambda$, for some open compact modulo center subgroup $K_{\pi}\subset G(\Q_p)$ and some finite dimensional representation $\lambda$ of $K_\pi$. Then we have
\[tr_{H(\pi)}(g)= \sum_{x\in \textrm{Fix}(g|\Fm^a(\Cm_p))} tr_{\pi}(h_{g,x}).\]
\end{corollary}

As remarked above, we should also prove an analogous Lefschetz trace formula for basic Rapoport-Zink spaces for $GSp4$, by their corresponding locally finite cell decomposition and the compactness of the fundamental domain. The Lefschetz trace formula for these Rapoport-Zink spaces for unitary groups or $GSp4$, and (for $n$ even in the unitary case) should enable us to prove the realization of local Jaquet-Langlands correspondence between irreducible smooth representations of $G(\Q_p)$ and $J_b(\Q_p)$ in the cohomology of these Rapoport-Zink spaces. By the methods of \cite{St} and \cite{Mi4}, it just rests the corresponding representation theoretic problems to solve(which is done for the case of $GSp4$). We will consider these in future works. On the other hand, for the non-basic Rapoport-Zink spaces in these cases, our previous results in \cite{Sh} say that their cohomology is essentially a parabolic induction.
\\

This paper is organized as follows. In section 2 we introduce the unitary group Rapoport-Zink spaces, and investigate the group action on them. In particular we study the Hecke action in details. In section 3, we review Fargues's theory of Harder-Narasimhan filtration for finite flat group schemes. In section 4 we apply the theory of last section to the study of $p$-divisible groups. We define the Harder-Narasimhan polygon for $p$-divisible groups with additional structures. In section 5, we first review Fargues's algorithm for studying non semi-stable $p$-divisible groups. We then adopt and modify this algorithm in the situation with additional structures. In section 6 we find the closed analytic domain $\C$ motivated by the algorithm of section 5. In section 7, we review some basic results of Vollaard-Wedhorn about the geometry of the special fiber of $\widehat{\M}$, and in section 8 we introduce the analytic domain $\D$. In section 9, we prove that $\D$ is relatively compact by introducing some unitary group Shimura varieties and then using the theory of $p$-adic uniformization. In section 10 we prove some results of locally finite cell decomposition of our Rapoport-Zink spaces, as well as $p$-adic period domain and Shimura varieties. In the last section we consider the cohomological application of the locally finite cell decomposition of our Rapoport-Zink spaces, and prove a Lefschetz trace formula by applying Mieda's theorem 3.13 in \cite{Mi2}.\\
 \\
\textbf{Acknowledgments.} I would like to thank Prof. Laurent Fargues sincerely, since without his guide this paper would not be accomplished. I should thank Yoichi Mieda, who had proposed some useful questions after the first version of this article. I should also thank the referee for careful reading and suggestions.

\section{The unitary group Rapoport-Zink spaces and Hecke action}

We consider here a special case of PEL type Rapoport-Zink spaces.

Let $p>2$ be a fixed prime number. Let $\Q_{p^2}$ be the unramified extension of $\Q_p$ of degree $2$ and denote by $\ast$ the nontrivial Galois automorphism of $\Q_{p^2}$ over $\Q_p$. Let $V$ be a finite dimensional $\Q_{p^2}$-vector space with $dim_{\Q_{p^2}}(V)=n$. Let $\lan,\ran: V\times V\ra \Q_p$ be a $\Q_p$-valued skew-hermitian form, and $G$ be the associated reductive group, i.e.,
\[G(R)=\{g\in GL_{\Q_{p^2}\otimes R}(V_R)|\exists c\in R^{\times}: \lan gv,gw\ran=c\lan v,w\ran, \forall v,w \in V_R:=V\otimes R\} \]
for all $\Q_p$-algebra $R$. We remark that there exists a unique skew-hermitian form $\lan,\ran': V\times V\ra \Q_{p^2}$ such that $\lan,\ran= Tr_{\Q_{p^2}/\Q_p}\circ \lan,\ran'$. Moreover, if $\delta \in \Q_{p^2}^{\times}$ with $\delta^\ast=-\delta$, then $(,):=\delta\lan,\ran'$ is a hermitian form, and $G$ is just the unitary similitude group $GU(V,(,))$ of the hermitian space $(V,(,))$. Let $\Z_{p^2}$ be the ring of integers of $\Q_{p^2}$. We assume that there exists a $\Z_{p^2}$-lattice $\Lambda$ such that $\lan,\ran$ induces a perfect $\Z_p$-pairing on $\Lambda$. This implies that $G$ is unramified over $\Q_p$ and has a reductive model over $\Z_p$.

Let $\overline{\Q}_p$ be an algebraic closure of $\Q_p$. Then there is a canonical imbedding
\[G_{\overline{\Q}_p}\subset (Res_{\Q_{p^2}/\Q_p}GL_{\Q_{p^2}}(V))_{\overline{\Q}_p}=GL(V\otimes_{\Q_{p^2},id}\overline{\Q}_p)\times GL(V\otimes_{\Q_{p^2},\ast}\overline{\Q}_p),\]and we have an isomorphism
\[G_{\overline{\Q}_p}\simeq GL(V\otimes_{\Q_{p^2},id}\overline{\Q}_p)\times \mathbb{G}_m .\] Via this isomorphism, we fix a $G(\overline{\Q}_p)$-conjugate class of cocharacter
\[\begin{split} \mu:\, \mathbb{G}_{m\overline{\Q}_p}&\longrightarrow G_{\overline{\Q}_p}\\
&z\mapsto (diag(z,\dots,z,1),z).
\end{split}\]
Let $L=W(\overline{\F}_p)_{\Q}$, $\sigma$ be the Frobenius relative to the field extension $L/\Q_p$. Consider the set $B(G)=G(L)/\sim$ of  $\sigma$-conjugate classes  in $G(L)$, and the Kottwitz set $B(G,\mu)\subset B(G)$ (\cite{Ko2}). In our special case we can have an explicit description of the set $B(G,\mu)$ as a set of polygons, see \cite{BW} 3.1. We consider the basic element $b=b_0 \in B(G,\mu)$, and let $J_b$ be the reductive group of automorphisms of the unitary isocrystal $(V_L,b\sigma,\iota,\lan,\ran)$, which is then an inner form of $G$ over $\Q_p$.

Associated to the above data $(\Q_{p^2},\ast,V,\lan,\ran,b,\mu)$, we have the Rapoport-Zink space $\widehat{\M}$ which is a formal scheme locally formally of finite type over Spf$O_L$. It is a moduli space of $p$-divisible groups with additional structures of the following type: for each $S\in$ Nilp$O_L$, $\widehat{\M}(S)=\{(H,\iota,\lambda,\rho)\}/\simeq$, where
\begin{itemize}
\item $H$ is a $p$-divisible group over $S$;
\item $\iota: \Z_{p^2}\ra End(H)$ is an action of $\Z_{p^2}$ on $H$ satisfying
locally \[Lie(H)=Lie(H)_0\oplus Lie(H)_1, rank_{O_S}Lie(H)_0=1, rank_{O_S}Lie(H)_1=n-1\] where
\[Lie(H)_0=\{x\in Lie(H)|\iota(a)x=ax\}, Lie(H)_1=\{x\in Lie(H)|\iota(a)x=a^\ast x\};\]
\item $\lambda: H\ra H^D$ is a principal $\Z_{p^2}$-linear polarization, here $H^D$ is the dual $p$-divisible group endowed with the $\Z_{p^2}$-action $\iota_{H^D}(a)=(\iota(a^\ast))^D$;
\item $\rho: \mathbf{H}_{\ov{S}}\ra H_{\ov{S}}$ is a quasi-isogeny, such that $\rho^D\circ\lambda\circ\rho$ is a $\Q_p^\times$-multiple of $\lambda$, here $\ov{S}=S\otimes_{\Z_{p^2}}\F_{p^2}$;
\item $(H_1,\iota_1,\lambda_1,\rho_1)\simeq (H_2,\iota_2,\lambda_2,\rho_2)$ if there exists a $\Z_{p^2}$-linear isomorphism $\alpha: H_1\ra H_2$ such that $\rho_2=\rho_1\circ \alpha, \alpha^D\circ\lambda_2\circ \alpha$ is a $\Z_p^\times$-multiple of $\lambda_1$.
\end{itemize}
We call such a $p$-divisible group with additional structures $H=(H,\iota,\lambda)$ a unitary $p$-divisible group. For such a unitary $p$-divisible group $H$, we have \[rank_{O_S}Lie(H)=n, height(H)=2n.\] The height of $\rho$ is a multiple of $n$ by \cite{V} 1.7 or \cite{Chen} and we obtain a decomposition
\[\widehat{\M}=\coprod_{i\in \Z}\widehat{\M}^i,\]where $\widehat{\M}^i$ is the open and closed formal subscheme of $\widehat{\M}$ where $\rho$ has height $in$. Moreover, we have in fact
\[\widehat{\M}^i\neq \emptyset \Leftrightarrow \,in\, \textrm{is even},\]and in this case there is an isomorphism $\widehat{\M}^i\cong \widehat{\M}^0$ induced the action of $J_b(\Q_p)$, see \cite{V} or the section 6 below.

The standard $p$-divisible group $\mathbf{H}=(\mathbf{H},\iota,\lambda)$ is definable over $\F_{p^2}$. We let \[(\mathbf{M,F,V,M=M_0\oplus M_1},\lan,\ran)\] denote its covariant Di\'eudonne module over $W(\F_{p^2})=\Z_{p^2}$, where $\lan,\ran: \mathbf{M\times M}\ra \Z_{p^2}$ is a perfect alternating $\Z_{p^2}$-bilinear pairing satisfying
\[\lan \mathbf{F}x,y\ran=\lan x,\mathbf{V}y\ran^\sigma, \lan ax,y\ran=\lan x,a^\ast y\ran\] for all $x,y\in \mathbf{M}, a\in \Z_{p^2}$, here $\sigma=\ast$ is the Frobenius on $W(\F_{p^2})=\Z_{p^2}$; the decomposition $\mathbf{M=M_0\oplus M_1}$ is induced by the decomposition $\Z_{p^2}\otimes_{\Z_p}W(\F_{p^2}) \simeq  W(\F_{p^2})\times W(\F_{p^2})$ and the $\Z_{p^2}$-action on $\mathbf{M}$. The $\mathbf{F}$ and $\mathbf{V}$ are homogeneous of degree 1 with respect to the above decomposition and $\mathbf{M}_0$ and $\mathbf{M}_1$ are totally isotropic with respect to $\lan,\ran$. The signature condition on the Lie algebra then implies
\[dim_{\F_{p^2}}(\mathbf{M}_0/\mathbf{VM}_1)=1, dim_{\F_{p^2}}(\mathbf{M}_1/\mathbf{VM}_0)=n-1.\] We denote by $(\mathbf{N,F})=(\mathbf{M,F})\otimes \Q_{p^2}$ the isocrystal of $\mathbf{H}$. We can assume that $\mathbf{H}$ is superspecial and that the isocystal $(\mathbf{N,F})$ is generated by the elements $x$ such that $\mathbf{F}^2x=px$, see \cite{VW}. As $\mathbf{F}^2$ is $\Q_{p^2}$-linear, we have $\mathbf{F}^2=pid_\mathbf{N}$ and therefore $\mathbf{F}=\mathbf{V}$. For $i=1,2$, let $\mathbf{N}_i=\mathbf{M}_i\otimes \Q_{p^2}$, then $\mathbf{N}=\mathbf{N}_0\oplus \mathbf{N}_1$ and with respect to this decomposition $\mathbf{F}$ is of degree 1. We fix an element $\delta\in\Z_{p^2}^\times$ such that $\delta^\ast=-\delta$ and define a nondegenerate hermitian form on the $\Q_{p^2}$-vector space $\mathbf{N}_0$ by
\[\{x,y\}:=\delta\lan x,\mathbf{F}y\ran.\]
Recall the reductive group $J_b$ over $\Q_p$ defined by the automorphisms of the unitary isocrystal $(\mathbf{N,F,N=N_0\oplus N_1},\lan,\ran)$, which is an inner form of $G$. We have then an isomorphism of $J_b$ with the unitary similitude group $GU(\mathbf{N}_0,\{,\})$ of the hermitian space $(\mathbf{N}_0,\{,\})$. Thus for $n$ odd, we have in fact an isomorphism \[G\cong J_b;\] while for $n$ even, $J_b$ is the non quasi-split inner form of $G$.
\\

We now describe the group actions on the Rapoport-Zink space $\widehat{\M}$. First, we have a left action of $J_b(\Q_p)$ on $\widehat{\M}$: $\forall g\in J_b(\Q_p)$,
\[ g: \widehat{\M} \longrightarrow \widehat{\M} \]
\[(H,\iota,\lambda,\rho)\mapsto (H,\iota,\lambda,\rho\circ g^{-1}),\]
since $J_b$ can be viewed as the group of self quasi-isogenies of $(\mathbf{H},\iota,\lambda)$.

To consider the action of $G(\Q_p)$ on $\widehat{\M}$, we would rather consider the associated Berkovich analytic space $\M=\widehat{\M}^{an}$ of $\widehat{\M}$. As always, by trivializing the Tate module of the universal $p$-divisible group over $\M$, we can define a tower of Berkovich analytic spaces $(\M_K)_{K\subset G(\Z_p)}$. These spaces are separated smooth good Berkovich analytic spaces over $L$. A point $x\in\M_K$ is given by $(H,\iota,\lambda,\rho,\eta K)$, where $\eta: V\st{\sim}{\ra} V_p(H)$ is the rigidification isomorphism such that $\eta(\Lambda)=T_p(H)$. Then $J_b(\Q_p)$ also acts on each space $\M_K$ in the natural way. Moreover, $G(\Q_p)$ acts on this tower: for $g\in G(\Q_p)$ and $K\subset G(\Z_p)$ such that $g^{-1}Kg\subset G(\Z_p)$, we have an isomorphism
\[g: \M_K\ra\M_{g^{-1}Kg},\]
\[(H,\iota,\lambda,\rho,\eta K)\mapsto (H',\iota',\lambda',\rho',\eta'(g^{-1}Kg)),\] here $(H',\iota',\lambda',\rho',\eta')$ is defined as following. Assume first $g^{-1}\in M_n(\Z_{p^2})$. Then $\Lambda\supset g^{-1}(\Lambda)$. Since $\eta(\Lambda)=T_p(H)$ for the rigidification $\eta$, $\eta(\Lambda/g^{-1}(\Lambda))$ defines a finite flat subgroup of $H$. We take $H':=H/\eta(\Lambda/g^{-1}(\Lambda))$, with the naturally induced additional structures $(\iota',\lambda')$ on $H'$, and we take $\rho'=\pi(mod\,p)\circ\rho$ for the natural projection $\pi: H\ra H'$.
Finally there is a rigidification $\eta': V\ra V_p(H')$ such that the following diagram commutes:
\[\xymatrix{ V\ar[r]^{\eta}\ar[d]^{g^{-1}}& V_p(H)\ar[d]^{V_p(\pi)}\\
V\ar[r]^{\eta'}& V_p(H').}\]
 For the general case, one can always find an integer $r\in\Z$ such that $p^rg^{-1}\in M_n(\Z_{p^2})$, then we can define $(H'',\iota'',\lambda'',\rho'',\eta'')$ as above for $p^{-r}g$. We set $H'=H'',\iota'=\iota'',\lambda'=\lambda'',\rho'=p^{-r}\rho'',\eta''=\eta'$.

For any open compact subgroups $K'\subset K\subset G(\Z_p)$, we denote $\pi_{K',K}: \M_{K'}\ra\M_K$ the natural projection of forgetting levels, which is a finite \'{e}tale morphism of degree $K/K'$.
In particular, for $K\subset G(\Z_p)$ fixed, each $g\in G(\Q_p)$ defines a Hecke correspondence on $\M_K$ by the following diagram:
\[\xymatrix{ &\M_{gKg^{-1}\cap K}\ar[r]^g_{\simeq}\ar[ld]_{\pi_{gKg^{-1}\cap K,K}}& \M_{K\cap g^{-1}Kg}\ar[rd]^{\pi_{K\cap g^{-1}Kg,K}}&\\
 \M_K& & & \M_K ,}
 \] and this Hecke correspondence depends only on the double coset $KgK$. Thus we get an ``action" of $K\setminus G(\Q_p)/K$ on $\M_K$, which commutes with the (left) action of $J_b(\Q_p)$.

 \begin{definition}
 Let $K\subset G(\Z_p)$ be an open compact subgroup. For any subset $A\subset |\M_K|$ of the underlying topological space $|\M_K|$, and any Hecke correspondence $T$ defined by a coset $KgK$ as above, we define the image of $A$ under $T$ by the set \[T.A=\pi_{K\cap g^{-1}Kg,K}g\pi_{gKg^{-1}\cap K,K}^{-1}(A).\] We call the set
 \[ Hecke(A):=\bigcup_{T\in K\setminus G(\Q_p)/K}T.A\] the Hecke orbit of $A$.
 \end{definition}
 \begin{remark}
 \begin{enumerate}
 \item By the above, the Hecke action of $G(\Q_p)$ on the tower $(\M_K)_{K\subset G(\Z_p)}$ is in fact a right action. So maybe we should better write the image of $A$ under $T$ as $A.T$. On the other hand, there is in general no composition law for the action of Hecke correspondences on $\M_K$, since the product (in the usual way) of double cosets $KgK\cdot KhK$ is in general not a single double coset. Therefore, we will write $T$ on the left as $T.A$, and for two Hecke correspondences $T_1,T_2$, $T_2.(T_1.A)$ should always be understood as the image of $T_1.A$ under $T_2$.
 \item More precisely, we have
     \[Kh_2K.(Kh_1K.A)=\bigcup_{KhK\subset Kh_1Kh_2K}KhK.A,\]where the right hand side is the finite union over all the double cosets in $Kh_1Kh_2K$. In particular, $A\subset KhK.(Kh^{-1}K.A)$ and $Hecke(A)=Hecke(T.A)$ for any $T\in K\setminus G(\Q_p)/K$.
 \item Note that if $A$ is an analytic domain, then so is $T.A$ for any $T\in K\setminus G(\Q_p)/K$.
 \end{enumerate}
 \end{remark}

 In the following we will mainly focus on the Hecke action on the space $\M$. We would like to describe the images of a point $x\in \M$ under the action of $G(\Z_p)\setminus G(\Q_p)/G(\Z_p)$ on
 $\M$ explicitly. To this end we first recall the Cartan decomposition to describe the set $G(\Z_p)\setminus G(\Q_p)/G(\Z_p)$ explicitly. Let $A\subset G$ be a maximal $\Q_p$-split torus such that
 \[A(\Q_p)=\{
 \left(\begin{array}{ccc}
 d_1 & & \\
 &\ddots& \\
  & &d_n\\
  \end{array}\right)
                 |d_1d_n^\ast=d_2d_{n-1}^\ast=\cdots=\textrm{constant}\in\Q_p^\times\}.\]
 Then the cocharacter group
 \[X_\ast(A)=\{(a_1,\cdots,a_n)\in \Z^n|a_1+a_n=a_2+a_{n-1}=\cdots=\textrm{constant}\in \Z\},\]
 and we denote the dominant coweights by
 \[X_\ast(A)_+=\{(a_1,\cdots,a_n)\in X_\ast(A)|a_1\geq\cdots\geq a_n\}.\]
 The Cartan decomposition says that the following map is a bijection:
 \[\begin{split}
 &X_\ast(A)_+\longrightarrow G(\Z_p)\setminus G(\Q_p)/G(\Z_p)\\
 &(a_1,\cdots,a_n)\mapsto G(\Z_p)
                  \left( \begin{array}{ccc}
                     p^{a_1} & & \\
                     &\ddots& \\
                     & &p^{a_n}\\
                   \end{array}\right)
                G(\Z_p).
 \end{split}\]

  A point $x\in \M$ corresponds to a tuple $(H/O_{K=\mathcal{H}(x)},\iota,\lambda,\rho)$, as an element of $\M(K)=\widehat{\M}(O_K)$ (Here and in the following, when the level is maximal, by abuse of notation $K$ will denote a complete extension of $\Q_p$. In any case, the precise meaning should be clear from the context.). For \[T=G(\Z_p)
                   \left(\begin{array}{ccc}
                     p^{a_1} & & \\
                     &\ddots& \\
                     & &p^{a_n}\\
                   \end{array}
                   \right)
                 G(\Z_p)\in G(\Z_p)\setminus G(\Q_p)/G(\Z_p),\]  we now give a moduli description of the finite set $T.x$. First assume $a_1\leq 0$. By the definition of the Hecke correspondence $T$, we have
  \[\begin{split}
  T.x=\{y\in\M|&(H_y,\iota_y,\lambda_y,\rho_y)\otimes_{O_{\mathcal{H}(y)}}O_K\simeq (H/G_y,\iota',\lambda',\pi\circ\rho),\\
  &\textrm{where}\, G_y\subset H \, \textrm{is a finite flat subgroup scheme, such that its}\\
  &\textrm{geometric generic fiber}\, G_{y\ov{K}}\simeq \Z_{p^2}/p^{-a_1}\Z_{p^2}\oplus\cdots\oplus \Z_{p^2}/p^{-a_n}\Z_{p^2},\\
  &\iota',\lambda' \textrm{are the naturally induced additional structures}\\
  &\pi: H_{O_K/pO_K}\ra (H/G_y)_{O_K/pO_K} \textrm{ is the natural projection}.\}
  \end{split}
  \]
For the general $T$, note the (right) action of an element $z\in \Q_p^\times\subset G(\Q_p)$ is the same as the (left) action of $z\in \Q_{p}^\times\subset J_b(\Q_p)$, see \cite{RZ} lemma 5.36. Since \[T.x=p^{a_1}(p^{-a_1}T).x,\]here the first scalar $p^{a_1}$ is considered as an element of $G(\Q_p)$, we have the description of the set $(p^{-a_1}T).x$ as in the above way.
Then we consider $p^{a_1}$ as an element of $J_b(\Q_p)$ which just changes the quasi-isogeny. So we can describe the set $T.x$ explicitly in all cases.
\\

We examine the effect of the group actions on connected components. First recall Rapoport-Zink (cf. \cite{RZ},3.52) have defined generally a locally constant mapping \[\varkappa:\widehat{\M}\ra \triangle,\]where $\triangle=\textrm{Hom}_\Z(X^\ast_{\Q_p}(G),\Z)$ and $X^\ast_{\Q_p}(G)$ is the group of $\Q_p$-rational characters of $G$. This mapping satisfies that
  \[\varkappa(gx)=\omega_J(g)+\varkappa(x)\]for all $g\in J_b(\Q_p),x\in\widehat{\M}$. Here $\omega_J:J_b(\Q_p)\ra\triangle$ is defined by $<\omega_J(x),\chi>=v_p(i(\chi)(x))$ where $i: X^\ast_{\Q_p}(G)\ra X^\ast_{\Q_p}(J_b)$ is the natural morphism between the two groups of $\Q_p$-rational characters. In our unitary group case, the similitude morphism $c: G\ra \mathbb{G}_m$ defines the identification $\triangle=\Z$. The mapping
  \[\varkappa:\widehat{\M}\longrightarrow\Z\]
   \[(H,\iota,\lambda,\rho)\mapsto -\frac{ht\rho}{n}.\] The image of $\varkappa$ is then $\Z$ if $n$ is even, and $2\Z$ if $n$ is odd. In section 7 we will review the geometry of the reduced special fiber $\M_{red}$ of $\widehat{\M}$. In particular we find $\M_{red}^0$ is connected and $\pi_0(\widehat{\M})=im\varkappa$. Since one has the equalities of the sets of connected components \[\pi_0(\widehat{\M})=\pi_0(\M_{red})=\pi_0(\M)=\pi_0(\M\times\mathbb{C}_p),\] thus each analytic space $\M^i$ for $i\in\Z$ such that $in$ is even is connected, which is in fact geometrically connected, cf. \cite{Chen} lemme 5.1.2.1.

   Let $D=G/G^{der}=J_b/J_b^{der}$ be the co-center group. More explicitly, we have
   \[D=\{(x,c)\in (Res_{\Q_{p^2}|\Q_p}\mathbb{G}_m)\times \mathbb{G}_m|N_{\Q_{p^2}|\Q_p}(x)=c^n\},\]where
   $N_{\Q_{p^2}|\Q_p}: Res_{\Q_{p^2}|\Q_p}\mathbb{G}_m\ra \mathbb{G}_m$ is the norm morphism. We have the determinant morphisms \[\begin{split}det: &G\longrightarrow D\\&g\mapsto (det_{\Q_{p^2}}(g),c(g)),\end{split}\] and similarly for $det: J_b\ra D$. In her doctoral thesis \cite{Chen}, Chen has associated to the torus $D$ and the cocharacter $det\tilde{\mu}$ ($\tilde{\mu}$ is a variant of $\mu$), a tower of analytic spaces $(\M(D,det\tilde{\mu})_K)_{K\subset D(\Z_p)}$ with mappings $\varkappa_{D,\tilde{\mu}}: \M(D,det\tilde{\mu})_K\ra \triangle$. The geometric points are $\M(D,det\tilde{\mu})_K(\ov{L})=D(\Q_p)/K$. By construction there is an action of $D(\Q_p)\times D(\Q_p)$ on each space $(\M(D,det\tilde{\mu})_K)$ such that on geometric points the action is just the left multiplication:
   \[(a,b).xK=abxK, \,\forall (a,b)\in D(\Q_p)\times D(\Q_p),\,xK\in D(\Q_p)/K.\]
   Moreover, via the morphism $(det,det): G(\Q_p)\times J_b(\Q_p)\ra D(\Q_p)\times D(\Q_p)$, she has constructed a $G(\Q_p)\times J_b(\Q_p)$-equivariant determinant morphism of towers of analytic spaces
   \[(\M_K)_K\longrightarrow (\M(D,\det\tilde{\mu})_{detK})_{detK}\]for $K$ varies as open compact subgroup of $G(\Z_p)$, which is compatible with the mappings $\varkappa$ and $\varkappa_{D,\tilde{\mu}}$. The main results of loc. cit. imply that the geometric fibers of the determinant morphism  \[\M_K\longrightarrow \M(D,\det\tilde{\mu})_{detK}\] are exactly the geometric connected components of $\M_K$.

For the case we are interested, $K=G(\Z_p),\,detK=D(\Z_p)$, the set of geometric connected components of $\M$ is the same with the set of its connected components, which is in bijection with $\M(D,\tilde{\mu})(\ov{L})=D(\Q_p)/D(\Z_p)$. Via the mappings $\varkappa$ and $\varkappa_{D,\tilde{\mu}}$ they are in turn bijection with $im\varkappa=im\varkappa_{D,\tilde{\mu}}$, which is thus $\Z$ if $n$ is even, and $2\Z$ if $n$ is odd. Now the effect of the actions of $G(\Q_p)$ and $J_b(\Q_p)$ on the connected components of $\M$ translates on the last set is as following. First for $g\in J_b(\Q_p)$, we have $\omega_J(g)=v_p(c(g))$ and \[g: \M^0\st{\sim}{\longrightarrow}\M^{-v_p(c(g))}.\]For \[T=G(\Z_p)
                   \left(\begin{array}{ccc}
                     p^{a_1} & & \\
                     &\ddots& \\
                     & &p^{a_n}\\
                   \end{array}
                   \right)
                 G(\Z_p)\in G(\Z_p)\setminus G(\Q_p)/G(\Z_p),\] let $g=diag(p^{a_1},\cdots,p^{a_n})$, then \[det(g)=(p^{a_1+\cdots+a_n},c(g))\in D(\Q_p), c(g)^n=p^{2(a_1+\cdots+a_n)}\] and \[v_p(c(g))=a_1+a_n=v_p(c(g')):=v_p(c(T)),\,\forall g'\in T.\] We have \[T.\M^0\subset\M^{-v_p(c(T))}\] and in fact this is an equality $T.\M^0=\M^{-v_p(c(T))}$. Let $\widetilde{G}=\{g\in G(\Q_p)|c(g)\in \Z_p^\times\}, U_n=\{g\in G(\Q_p)|c(g)=1\}$, then $G^{der}(\Q_p)\subset U_n\subset \widetilde{G}$, and the Hecke correspondences associated to elements in $\widetilde{G}$ stabilize $\M^0$.
\\

 We consider the $p$-adic period mapping \[\pi: \M\ra \Fm^a\subset \Fm^{wa}\subset \Fm,\] where $\Fm=(G_L/P_{\mu L})^{an}\simeq \mathbf{P}^{n-1,an}$ is the $p$-adic projective space over $L=W(\ov{\F}_p)_\Q$, $\Fm^{wa}$ is the weakly admissible locus and $\Fm^a$ is the image of $\pi$, see \cite{Hart2} for some discussion of these objects. We recall the definition of $\pi$ in the following. The universal quasi-isogeny $\rho$ induces an isomorphism
\[V_L\otimes_{O_L}O_{\M}\simeq LieE(\mathcal{H})^{an},\]here $E(\mathcal{H})$ is the universal vector extension of $\mathcal{H}_{\ov{\M}}$ over $\widehat{\M}$, $\ov{\M}\subset\widehat{\M}$ is the closed subscheme defined by $p$. For a $p$-divisible group $H/O_K$, recall we have the exact sequence
\[0\ra \omega_{H^D,K}\ra  LieE(H)_K\ra Lie(H)_K\ra0.\]
If $(H,\iota,\lambda,\rho)$ is associated to a point $x\in\M$, then $\rho$ induces an isomorphism
\[\rho_{\ast}: \mathbf{M}\otimes K\stackrel{\sim}{\ra}LieE(H)_K,\]where $K=\mathcal{H}(x)$ is the complete residue field associated to $x$, the filtration
$\rho_\ast^{-1}(\omega_{H^D,K})\subset V_K=\mathbf{M}\otimes K$ defines a point is the Grassmannin
$\Fm=(G_L/P_{\mu L})^{an}\simeq \mathbf{P}^{n-1,an}$, this is the image $\pi(x)$ of $x\in\M$. There is an action of $J_b(\Q_p)$ on $\Fm^a$ and the mapping $\pi$ is $J_b(\Q_p)$-equivariant. In fact $\pi$ is $G(\Q_p)$-invariant for the Hecke action on $\M$, see the following proposition 2.3. Under the $p$-adic period mapping $\pi:\M\ra \Fm^a$, we have $\pi(\M^i)=\Fm^a$ for each $i\in \Z$ such that $in$ is even, and $\Fm^a$ is connected, cf. \cite{Chen} lemme 5.1.1.1.

When $x\in\M$ is a rigid point, i.e. $K$ is a finite extension of $L$ thus in particular discrete with perfect residue field $\ov{\F}_p$, we have the unitary filtered isocrystal $(V_L,b\sigma,\textrm{Fil}^\bullet V_K,\iota,\lan,\ran)$ associated to $(H,\iota,\lambda,\rho)$. Here the filtration on $V_K$ is defined by \[\textrm{Fil}^{-1}V_K=V_K,\textrm{Fil}^{0}V_K=\rho_\ast^{-1}(\omega_{H^D,K}),\textrm{Fil}^iV_K=0,\,\forall i\neq -1,0.\]The filtered isocrystal $(V_L,b\sigma,\textrm{Fil}^\bullet V_K,\iota,\lan,\ran)$ determines the isogeny class of $(H,\iota,\lambda,\rho)$.
We have the following description of the Hecke orbit of a point $x\in\M$, which is similar to proposition 5.37 in the book \cite{RZ}.
\begin{proposition}
The Hecke orbit of a point $x\in \M$ is exactly the fiber $\pi^{-1}(\pi(x))$ of its image under the $p$-adic period mapping $\pi$.
\end{proposition}
\begin{proof}
Let $x,y\in\M$ be two points, and $(H_1/O_{\mathcal{H}(x)},\iota_1,\lambda_1,\rho_1),(H_2/O_{\mathcal{H}(y)},\iota_2,\lambda_2,\rho_2)$ be the unitary $p$-divisible groups associated to $x$ and $y$ respectively. Then $x,y$ in the same Hecke orbit if and only if there exists a finite extension $K$ of both $\supset\mathcal{H}(x)$ and $\mathcal{H}(y)$, and a (unique) quasi-isogeny $\varphi: H_1\ra H_2$ over $O_K$ lifting
\[\rho_2\circ\rho_1^{-1}: H_{1O_K/pO_K}\stackrel{\rho_1^{-1}}{\longrightarrow}\mathbf{H}_{O_K/pO_K}\stackrel{\rho_2}{\longrightarrow}H_{2O_K/pO_K}.
\]
Let $\mathbf{M}=\mathbb{D}(\mathbf{H})_{W(\ov{\F}_p)\ra\ov{\F}_p}$ be the value on $W(\ov{\F}_p)\ra\ov{\F}_p$ of the covariant cystal $\mathbb{D}(\mathbf{H})$. Then since \[\xymatrix{SpecO_K/pO_K \ar[d]\ar@{^{(}->}[r]&SpecO_K\ar[d]\\
Spec\ov{\F}_p\ar@{^{(}->}[r]&SpecW(\ov{\F}_p) }\] is a PD-morphism, we have
\[\mathbb{D}(\mathbf{H}\otimes O_K/pO_K)_{O_K\twoheadrightarrow O_K/pO_K}=\mathbb{D}(\mathbf{H})_{O_K\twoheadrightarrow O_K/pO_K}=\mathbf{M}\otimes O_K.\]
The quasi-isogenies $\rho_1,\rho_2$ then induce isomorphisms
\[\rho_{i\ast}: \mathbf{M}\otimes K\stackrel{\sim}{\ra}LieE(H_i)_K, i=1,2.\]
The images of $x,y$ under the $p$-adic period mapping $\pi:\M\ra\Fm^a$ by definition are
\[\rho_{i\ast}^{-1}(Fil_{H_i}):=Fil_i\subset\mathbf{M}\otimes K, i=1,2.\]
$\rho_2\circ\rho_1^{-1}$ can be lifted to a quasi-isogeny $\widetilde{\rho_2\circ\rho_1^{-1}}:H_1\ra H_2$, by the theory of Grothendieck-Messing, if and only if the homomorphism
\[\mathbb{D}(\rho_2\circ\rho_1^{-1})_{O_K\twoheadrightarrow O_K/pO_K}: \mathbb{D}(H_1)_{O_K\twoheadrightarrow O_K/pO_K}[\frac{1}{p}]\stackrel{\sim}{\ra} \mathbb{D}(H_2)_{O_K\twoheadrightarrow O_K/pO_K}[\frac{1}{p}]\] send the Hodge filtration $Fil_{H_1}$ to $Fil_{H_2}$. But this is equivalent to say $Fil_1=Fil_2$, i.e. $\pi(x)=\pi(y)$. Thus the proposition holds.

\end{proof}
Thus a point $y\in\M$ is in the Hecke orbit of $x\in\M$, if and only if there exist a finite extension $K$ of both $\mathcal{H}(x)$ and $\mathcal{H}_y$, and an unique quasi-isogeny $H_y\ra H_x$ over $O_K$ lifting the quasi-isogeny $\rho_x\circ\rho_y^{-1}$ over $O_K/pO_K$. Note the last condition is equivalent to that there exists an isogeny $H_y\ra H_x$. If $x\in\M^{rig}$ is a rigid point, then one find easily that its Hecke orbit consist all of rigid points, i.e. $Hecke(x)\subset\M^{rig}$. In this case the Hecke orbit is determined by the filtered isocrystal $(V_L,b\sigma,\textrm{Fil}_{\pi(x)}^\bullet V_K,\iota,\lan,\ran)$.

For a geometric point $\ov{x}\in \M(K), K=\ov{K}$, the geometric fiber $\pi^{-1}(\pi(\ov{x}))$ is then bijective to the set of cosets $G(\Q_p)/G(\Z_p)$, see \cite{RZ} proposition 5.37. By \cite{dJ2} and \cite{Chen}, $\pi: \M\ra \Fm^a$ is an \'{e}tale covering map in the sense that, $\forall y\in \Fm^a$, there exists an open neighborhood $\mathcal{U}\subset \Fm^a$ such that $\pi^{-1}(\mathcal{U})$ is a disjoint union of spaces $\mathcal{V}_i$, each restriction map $\pi_{|\mathcal{V}_i}:\mathcal{V}_i\ra \mathcal{U}$ is finite \'{e}tale. In particular, the Hecke orbit $\pi^{-1}(\pi(x))$ is a discrete subspace of $\M$. Thus it makes sense to talk ``fundamental domain'' for the Hecke action of $G(\Q_p)$ on $\M$. In the following section, we will consider the action of $G(\Q_p)\times J_b(\Q_p)$ on $\M$ and find some ``fundamental domain'' for this action.
\\

We explain how is the Hecke action of $G(\Q_p)$ on a Hecke orbit $\pi^{-1}(\pi(x))$. We first look at the geometric Hecke orbits. Fix a geometric point $\ov{x}$ over $x$ and denote their images under the $\pi$ by $\ov{y}$ and $y$ respectively, let $\pi_1(\Fm^a,\ov{y})$ be the \'{e}tale fundamental group of $\Fm^a$ defined by de Jong in \cite{dJ2}. Then by definition there is an action of $\pi_1(\Fm^a,\ov{y})$ on the geometric Hecke orbit $\pi^{-1}(\ov{y})$. Let us fix a point in this orbit, say $\ov{x}$, then we have an identification $\pi^{-1}(\ov{y})=G(\Q_p)/G(\Z_p)$. Thus $\pi_1(\Fm^a,\ov{y})$ acts on $G(\Q_p)/G(\Z_p)$. On this set $G(\Q_p)/G(\Z_p)$ we have two other group actions, namely the group $G(\Q_p)$ and the Galois group $Gal(\ov{\mathcal{H}(y)}/\mathcal{H}(y))$. The relation between these actions is as follows. Recall the $\Z_p$-local system $\mathcal{T}$ defined by the universal \'{e}table unitary $p$-divisible group on the $p$-adic analytic space $\M$, it descends to a $\Q_p$-local system on $\Fm^a$, which we still denote by $\mathcal{T}$. Since $\Fm^a$ is connected, by de Jong \cite{dJ2} theorem 4.2. we have the equivalence of categories
\[\begin{split}\Q_p-\mathcal{L}oc_{\Fm^a}&\st{\sim}{\longrightarrow} \textrm{Rep}_{\Q_p}(\pi_1(\Fm^a,\ov{y}))\\&\mathcal{E}\mapsto \mathcal{E}_{\ov{y}}\end{split}\]by the monodromy representation functor. Here $\Q_p-\mathcal{L}oc_{X}$ is the category of $\Q_p$-local systems over a Berkovich space $X$ introduced in loc. cit. definition 4.1. One can translate the above equivalence to the case with additional structures by using Tannakian language. In particular the $\Q_p$-local system  $\mathcal{T}$ over $\Fm^a$ defines a representation of the fundamental group:
\[\rho:\pi_1(\Fm^a,\ov{y})\ra G(\Q_p).\] Then the above action of $\pi_1(\Fm^a,\ov{y})$ on $G(\Q_p)/G(\Z_p)$ is compatible with the natural action of $G(\Q_p)$ on the quotient set, through the morphism $\rho$. On the other hand, there is the natural action of the Galois group $Gal(\ov{\mathcal{H}(y)}/\mathcal{H}(y))$ on the geometric Hecke orbit $\pi^{-1}(\ov{y})$, and the quotient set is the Hecke orbit of $x$
\[\pi^{-1}(\pi(x))=\pi^{-1}(y)\simeq \pi^{-1}(\ov{y})/Gal(\ov{\mathcal{H}(y)}/\mathcal{H}(y)).\]
The point $y$ in $\Fm^a$ defines a morphism of fundamental groups
\[\pi_1(y,\ov{y})=Gal(\ov{\mathcal{H}(y)}/\mathcal{H}(y))\ra \pi_1(\Fm^a,\ov{y}),\] and the action of $Gal(\ov{\mathcal{H}(y)}/\mathcal{H}(y))$ and $\pi_1(\Fm^a,\ov{y})$ is compatible on $\pi^{-1}(\ov{y})$ through the above morphism. Thus the three group $\pi_1(\Fm^a,\ov{y}),G(\Q_p)$ and $Gal(\ov{\mathcal{H}(y)}/\mathcal{H}(y))$ act compatibly on $G(\Q_p)/G(\Z_p)$ via the morphisms
 \[Gal(\ov{\mathcal{H}(y)}/\mathcal{H}(y))\ra \pi_1(\Fm^a,\ov{y})\ra G(\Q_p).\]Here although we will not need it in the following, we remark that the monodromy representations of geometric fundamental groups $\pi_1(\M\times\mathbb{C}_p,\ov{x})$ and $\pi_1(\Fm^a\times\mathbb{C}_p,\ov{y})$ factor through $G^{der}(\Z_p)$ and $G^{der}(\Q_p)$ respectively:
 \[\pi_1(\M\times\mathbb{C}_p,\ov{x})\longrightarrow G^{der}(\Z_p)\]
 \[\pi_1(\Fm^a\times\mathbb{C}_p,\ov{y})\longrightarrow G^{der}(\Q_p).\]Moreover these monodromy representations are maximal in the sense that both images are dense in the targets respectively, cf. \cite{Chen} th\'eor\`eme 5.1.2.1. Let $\Gamma$ be the image of the later. Note that $\pi_1(\Fm^a\times\mathbb{C}_p,\ov{y})$ also acts on $G(\Q_p)/G(\Z_p)$ compatibly with the action of $G(\Q_p)$. By de Jong's description of $\M$ in term of lattice in the $\Q_p$-local system $\mathcal{T}$ over $\Fm^a$, one have the bijection \[\pi_0(\M)\simeq \Gamma\setminus G(\Q_p)/G(\Z_p).\]

 The action of the Hecke correspondences on the geometric orbit is then quite easy: with the fixed point $\ov{x}$ in $\pi^{-1}(\ov{y})$ and the identification $\pi^{-1}(\ov{y})=G(\Q_p)/G(\Z_p)$, a correspondence $G(\Z_p)gG(\Z_p)$ sends a coset $hG(\Z_p)$ to the set of cosets $\{h'G(\Z_p)|\,h'G(\Z_p)\subset G(\Z_p)g^{-1}hG(\Z_p)\}$. Thus the Hecke action is compatible with the natural action of $G(\Q_p)$ on $G(\Q_p)/G(\Z_p)$, and thus compatible with the action of the Galois group $Gal(\ov{\mathcal{H}(y)}/\mathcal{H}(y))$. The Hecke action on $\pi^{-1}(\ov{y})/Gal(\ov{\mathcal{H}(y)}/\mathcal{H}(y))$ then induces the Hecke action on the orbit $\pi^{-1}(\pi(x))$.

\section{Harder-Narasimhan filtration of finite flat group schemes}
In order to study the Rapoport-Zink space $\M$, from this section to the end of section 5, we will turn to the study of finite flat group schemes and $p$-divisible groups over complete valuation rings following the ideas in \cite{F2} and \cite{F3}. In this section we first recall briefly the theory of Harder-Narasimhan filtration of finite flat group schemes which is presented in detail in \cite{F2}.

Let $K|\Q_p$ be a complete rank one valuation field extension, $O_K$ be the integer ring of $K$, and $\C$ be the exact category of commutative finite flat group schemes with order some power of $p$ over $SpecO_K$. For $G\in\C$, recall there is an operation of schematic closure which is the inverse of taking generic fibers, and induces a bijection between the following two finite sets
\[\{\textrm{closed subgroups of}\,G_K\}\stackrel{\sim}{\longrightarrow}\{\textrm{finite flat subgroups of }\,G\,\textrm{over}\,O_K\}.\]
 There are two additive functions \[ht: \C\ra \mathbb{N}\]\[deg: \C\ra\R_{\geq0},\] where $htG$ is the height of $G\in\C$, and $degG$ is the valuation of the discriminant of $G$, which is defined as
  \[degG=\sum a_i,\, \textrm{if}\, \omega_G=\oplus O_K/p^{a_i}O_K.\] See \cite{F2} section 3 for more properties of the function $deg$. For a group scheme $0\neq G\in \C$, we set
\[\mu(G):=\frac{degG}{htG},\]and call it the slope of $G$. The basic properties of the slope function is as follows.
\begin{itemize}
 \item One always have $\mu(G)\in [0,1]$, with $\mu(G)=0$ if and only if $G$ is \'{e}tale and $\mu(G)=1$ if and only if $G$ is multiplicative.
 \item If $G^D$ is the Cartier dual of $G$ then $\mu(G^D)=1-\mu(G)$.
 \item For a $p$-divisible group $H$ of dimension $d$ and height $h$ over $O_K$, then for all $n\geq 1$ one has $\mu(H[p^n])=\frac{d}{h}$.
 \item If \[0\ra G'\ra G\ra G''\ra0\] is an exact sequence of non trivial groups in $\C$, then we have $inf\{\mu(G'),\mu(G'')\}\leq \mu(G)\leq sup\{\mu(G'),\mu(G'')\}$, and if $\mu(G')\neq \mu(G'')$ we have in fact $inf\{\mu(G'),\mu(G'')\}< \mu(G)< sup\{\mu(G'),\mu(G'')\}$.
 \item If $f: G\ra G'$ is a morphism which induces an isomorphism between their generic fibers, then we have $\mu(G)\leq \mu(G')$, with equality holds if and only if $f$ is an isomorphism.
 \item If \[0\longrightarrow G'\stackrel{u}{\longrightarrow} G\stackrel{v}{\longrightarrow} G''\] is a sequence of non trivial groups such that $u$ is a closed immersion, $u\circ v=0$, and the morphism induced by $v$
\[G/G'\ra G''\] is an isomorphism in generic fibers, then we have \begin{itemize}
 \item $\mu(G)\leq sup\{\mu(G'),\mu(G'')\}$;
 \item if $\mu(G')\neq \mu(G'')$ then $\mu(G)< sup\{\mu(G'),\mu(G'')\}$;
 \item if $\mu(G)=sup\{\mu(G'),\mu(G'')\}$ then $\mu(G)=\mu(G')=\mu(G'')$ and the sequence $0\ra G'\ra G\ra G''\ra0$ is exact.
 \end{itemize}

 \end{itemize}

For a group $0\neq G\in \C$, we call $G$ semi-stable if for all $0 \subsetneq G'\subset G$ we have $\mu(G')\leq \mu(G)$. In \cite{F2}, Fargues proved the following theorem.
\begin{theorem}[\cite{F2}, Th\'eor\`eme 2]There exists a Harder-Narasimhan type filtration for all $0\neq G\in\C$, that is a chain of finite flat subgroups in $\C$
\[0=G_0\subsetneq G_1\subsetneq \cdots \subsetneq G_r=G,\]with the group schemes $G_{i+1}/G_i$ are semi-stable for all $i=0,\dots,r-1$, and \[\mu(G_1/G_0)>\mu(G_2/G_1)>\cdots >\mu(G_r/G_{r-1}).\] Such a filtration is then unique characterised by these properties.
\end{theorem}

So $G$ is semi-stable if and only if its Harder-Narasimhan filtration is $0\subsetneq G$. We can define a concave polygon $HN(G)$ of any $0\neq G\in\C$ by its Harder-Narasimhan filtration, and call it the Harder-Narasimhan polygon of $G$. It is defined as function
\[HN(G): [0,htG]\ra [0,degG],\]
such that \[HN(G)(x)=degG_i+\mu(G_{i+1}/G_i)(x-htG_i)\]if $x\in [htG_i,htG_{i+1}]$. We will also identify $HN(G)$ with its graph, that is a polygon with starting point $(0,0)$, terminal point $(htG,degG)$, and over each integral $[htG_i,htG_{i+1}]$ it is a line of slope $\mu(G_{i+1}/G_i)$. An important property of this polygon is that, for all finite flat subgroups $G'\subset G$, the point $(htG',degG')$ is on or below the polygon $HN(G)$, that is $HN(G)$ is the concave envelop of the points $(htG',degG')$ for all $G'\subset G$. We denote $\mu_{max}(G)$ the maximal slope of $HN(G)$, and $\mu_{min}(G)$ the minimal slope of $HN(G)$.

We recall some useful facts.
\begin{proposition}[\cite{F2}, Proposition 10]Let $0\ra G'\ra G\stackrel{v}{\ra} G''\ra0$ be an exact sequence of finite flat group schemes in $\mathcal{C}$. Suppose $\mu_{min}(G')\geq\mu_{max}(G'')$. Then if $0=G'_0\subsetneq G'_1\subsetneq\cdots\subsetneq G'_{r'}=G'$ is the Harder-Narasimhan filtration of $G'$ and
$0=G''_0\subsetneq G''_1\subsetneq\cdots\subsetneq G''_{r''}=G''$ that of $G''$, then the Harder-Narasimhan filtration of $G$ is
\[0=G'_0\subsetneq G'_1\subsetneq\cdots\subsetneq G'_{r'}=G'\subsetneq v^{-1}(G''_0)\subsetneq v^{-1}(G_1'')\subsetneq\cdots\subsetneq v^{-1}(G''_{r''})=G\] if $\mu_{min}(G')>\mu_{max}(G'')$, and
\[0=G'_0\subsetneq G'_1\subsetneq\cdots\subsetneq G'_{r'-1}\subsetneq v^{-1}(G_1'')\subsetneq\cdots\subsetneq v^{-1}(G''_{r''})=G\]if $\mu_{min}(G')=\mu_{max}(G'')$.
 In particular the extension of two semi-stable groups of the same slope $\mu$ is semi-stable of slope $\mu$.
 \end{proposition}
 \begin{proposition}[\cite{F2}, Corollaire 6]
 For $\mu\in[0,1]$ fixed, the category of semi-stable finite flat group schemes of slope $\mu$ and the trivial group 0 is a sub abelian category of the category of fppf sheaves over $SpecO_K$.
 \end{proposition}
 \begin{proposition}[\cite{F2}, Corollaire 7]
 For a semi-stable group $G$, the kernel of multiplication by $p^m$ is flat and semi-stable of slope $\mu(G[p^m])=\mu(G)$. If $p^mG\neq 0$ then $p^mG$ is also semi-stable of slope $\mu(p^mG)=\mu(G)$.
 \end{proposition}

The Harder-Narasimhan filtration of finite flat group schemes is compatible with additional structures. First, the Harder-Narasimhan filtration of $0\neq G\in \C$ is stable under $End(G)$. So if $\iota: R\ra End(G)$ is some action of an $O_K$-algebra $R$, then very scran $G_i$ in the Harder-Narasimhan filtration of $G$ is a $R$-subgroup via $\iota$. Second, if the Harder-Narasimhan filtration of $G$ is
\[0=G_0\subsetneq G_1\subsetneq \cdots \subsetneq G_r=G\]with slopes $\mu_1>\cdots>\mu_r$,
then the Harder-Narasimhan filtration of the Cartier dual $G^D$ of $G$ is
\[0=(G/G_r)^D\subsetneq (G/G_{r-1})^D\subsetneq\cdots\subsetneq (G/G_1)^D\subsetneq G^D\]
with slopes $1-\mu_r>\cdots>1-\mu_1$. In particular, if $\lambda: G\stackrel{\sim}{\ra}G^D$ is a polarization, then it induces isomorphisms \[ G_i\simeq (G/G_{r-i})^D, i=1,\dots,r\]and thus $\mu_i+\mu_{r+1-i}=1,i=1,\dots,r$.

Finally we have the semi-continuity of the function $HN$ for a family finite flat group schemes.
\begin{theorem}[\cite{F2}, Th\'eor\`eme 3]
 Let $K|\Q_p$ be a complete discrete valuation field extension, and $\mathcal{X}$ be a formal scheme of formally locally of finite type over $SpfO_K$. Let $G$ be a locally free finite group scheme over $\mathcal{X}$ of constant height $h=ht(G)$. Then the map $x\mapsto HN(G_x)$ from the underlying topological space of the associated Berkovich analytic space $\mathcal{X}^{an}$ to the space of Harder-Narasimhan polygons is continuous. Moreover, it is semi-continuous for the G-topology on $\mathcal{X}^{an}$ defined by analytic domains in the following sense. If $\P:[0,h]\ra \R$ is a fixed polygon such that the abscissas its break points are integers. Then
     \[\{x\in|\mathcal{X}^{an}||HN(G_x)\leq \P\}\] is a closed analytic domain in $\mathcal{X}^{an}$, whose associated rigid space is an admissible open in $\mathcal{X}^{rig}$. In particular if the degree function $x\mapsto degG_x$ is constant on $\mathcal{X}^{an}$, then the semi-stable locus
     \[\{x\in|\mathcal{X}^{an}||\,G_x \,\textrm{is semi-stable}\}\]is a closed analytic domain of $\mathcal{X}^{an}$.
 \end{theorem}

 \section{Harder-Narasimhan polygon of $p$-divisible groups}

One can then use the theory of Harder-Narasimhan filtration of finite flat group schemes to study $p$-divisible groups, $p$-adic analytic Rapoport-Zink spaces and Shimura varieties. Let $H/O_K$ be a $p$-divisible group of dimension $d$ and height $h$, where $O_K$ as above. Then for any $m\geq1$, we have a function $HN(H[p^m])$. We normalize it as a function
\[ \begin{split}\frac{1}{m}HN(H[p^m])(m\cdot): [0,h]&\longrightarrow [0,d]\\&x\mapsto \frac{1}{m}HN(H[p^m])(mx).\end{split}\]
 In \cite{F3}, Fargues proved as $m \ra \infty$ these functions uniformly convergent to a continue function, which we call the (normalized) Harder-Narasimhan polygon of $H$
 \[HN(H): [0,h]\ra [0,d],\]and in fact
 \[HN(H)(x)=\inf_{m\geq1}\frac{1}{m}HN(H[p^m])(mx)\]for all $x\in[0,h]$. Moreover, this function is invariant when $H$ varies in its isogeny class: $HN(H)=HN(H')$ for any $p$-divisible group $H'$ isogenous to $H$. A not evident fact is that $HN(H)$ is in fact the polygon attached to a Harder-Narasimhan type filtration of the rational Hodge-Tate module, so it is really a polygon! In the case the valuation on $K$ is discrete, then it is the Harder-Narasimhan polygon of the crystalline representation defined by the rational Tate module $V_p(H)$ for suitably defined slope function, which in turn can also be formulated in the associated admissible filtered isocrystals, see \cite{F3} sections 8, 10, \cite{Sh} section 4.

 As mentioned in the introduction, one of the main results in \cite{F3} is the following inequality between the Harder-Narasimhan and Newton polygons
 \begin{theorem}[\cite{F3}, Th\'eor\`eme 6, 21]
 \[HN(H)\leq Newt(H_k).\]
 \end{theorem}
 In fact when the base valuation ring $O_K$ is not necessary discrete, one has to assume $H$ is ``modulaire'' in the sense of definition 25 in loc. cit., which is naturally satisfied for $p$-divisible groups coming from points in the Berkovich analytic Rapoport-Zink spaces. The proof of the above theorem for $p$-divisible groups over complete rank one discrete valuation $O_K|\Z_p$ with perfect residue field is easy. It comes from the fact that the reduction functor between the two categories of $p$-divisible groups up to isogenies
   \[\textrm{pdiv}_{O_K}\otimes \Q\longrightarrow \textrm{pdiv}_{k}\otimes \Q\] is exact and preserving the height and dimension functions. One can also rewrite these polygons in terms of the associated filtered isocystal and explain the inequality by the theory of filtered isocystals. For the non-discrete case, Fargues has used heavily $p$-adic Hodge theory and studied the Harder-Narasimhan filtration of the Banach-Colmez spaces. For more detail see section 10, 11 of loc. cit. In \cite{Sh} we have introduced Harder-Narasimhan polygons for $p$-divisible groups with additional structures, which include our unitary $p$-divisible groups in section 2 as a special case. We recall the definition here.

 \begin{definition}[\cite{Sh}, Definition 3.1]
Let $S$ be a formal scheme and $F|\Q_p$ be a finite unramified extension. By a $p$-divisible group with additional structures over $S$ for $F|\Q_p$, we mean
\begin{itemize}
\item in the EL case, a pair $(H,\iota)$, where $H$ is a $p$-divisible group over $S$, and $\iota: O_F\ra \textrm{End}(H)$ is a homomorphism of algebras;
\item in the PEL symplectic case, a triplet $(H,\iota,\lambda)$, where $H$ is a $p$-divisible group over $S$, $\iota: O_F\ra \textrm{End}(H)$ is homomorphism of algebras, $\lambda: (H,\iota)\stackrel{~}{\ra}(H^D,\iota^D)$ is a polarization, i.e. an $O_F$-equivariant isomorphism of $p$-divisible groups. Here $H^D$ is the Cartier-Serre dual of the $p$-divisible group $H$, $\iota^D: O_F\ra \textrm{End}(H^D)=\textrm{End}(H)^{opp}$ is induced by $\iota$, such that $\lambda^D=-\lambda$, under the identification $H=H^{DD}$;
\item in the PEL unitary case, a triplet $(H,\iota,\lambda)$, where $H,\iota$ is similar as the symplectic case, $\lambda: (H,\iota)\stackrel{~}{\ra}(H^D,\iota^D\circ \ast)$ is a polarization, $\ast$ is a nontrivial involution on $F$. Here $\iota^D$ is as above, but such that $\lambda^D=\lambda$, under the identification $H=H^{DD}$.
\end{itemize}
Similarly, one can define finite locally free (=flat, in the case $S$ is noetherian or the spec of a local ring) group schemes with additional structures in the same way.
\end{definition}
 If $(H,\iota,\lambda)$ is a $p$-divisible group with additional structures in the PEL cases, then $\forall n\geq 1, (H[p^n],\iota,\lambda)$ is a finite locally free group schemes with the naturally induced additional structures. Similar remark holds for the $EL$ case.

 Let $F|\Q_p$ be a finite unramified extension of degree $d$, $(H,\iota,\lambda)$ (resp. $(H,\iota)$) be a $p$-divisible group with PEL (resp. EL) additional structures for $F|\Q_p$ over a complete rank one valuation ring $O_K|\Z_p$. Then we define the Harder-Narasimhan polygon of $(H,\iota,\lambda)$ (resp. $(H,\iota)$) as the normalization of $HN(H)$ as following.
 \begin{definition}[\cite{Sh}, Definition 3.4] With the notation above, we define
 \[HN(H,\iota,\lambda)(\textrm{resp.}\;(H,\iota))=\frac{1}{d}HN(H)(d\cdot)\]as a function $[0,htH/d]\ra [0,dimH/d]$,
 which we will also identify with its graph as a polygon in $[0,htH/d]\times[0,dimH/d]$.
 \end{definition}

 One can easily deduce from their definition and theorem 4.1 the following generalization.
 \begin{proposition}[\cite{Sh}, Proposition 3.6]Let $(H,\iota,\lambda)$ be a $p$-divisible group with additional structures over a complete rank one valuation ring $O_K|\Z_p$. When $O_K$ is not of discrete valuation we assume in addition that $H$ is modular in the sense of definition 25 in \cite{F3}. Then we have the basic inequality
 \[HN(H,\iota,\lambda)\leq Newt(H_k,\iota,\lambda).\]
 Similar conclusion holds for the EL case.
 \end{proposition}

  Finally we also have the semi-continuity of the function $HN$ for a family of $p$-divisible groups. Fix a finite unramified extension $F|\Q_p$ of degree $d$, and we consider $p$-divisible groups with additional structures for $F|\Q_p$. The following proposition can be deduced directly from the case without additional structures, see \cite{F3} proposition 4. To fix notations we just state it for the PEL cases.
 \begin{proposition}Let $K|\Q_p$ be a complete discrete valuation field, and $\mathcal{X}$ be a formal scheme locally formally of finite type over $SpfO_K$. Let $(H,\iota,\lambda)$ be a $p$-divisible group with additional structures over $\mathcal{X}$ of dimension $\frac{dn}{2}$ and height $dn$ constant. Then the normalized Harder-Narasimhan function on the underlying topological space of $\mathcal{X}^{an}$ is semi-continuous: if $\P: [0,n]\ra[0,n/2]$ is a concave function such that $\P(0)=0,\P(n)=n/2$, then the subset
 \[\{x\in|\mathcal{X}^{an}||HN(H_x,\iota,\lambda)\geq \P\}\] is closed.
 \end{proposition}
 In particular, with the above notation this proposition permits us to define a stratification of the underlying topological space of $X:=\mathcal{X}^{an}$. More precisely let Poly denote the set of concave polygons starting from the point (0,0) to the point $(n,n/2)$, such that the abscissas of its break points are integers. Then we have a stratification by Harder-Narasimhan polygons
 \[X=\coprod_{\P\in \textrm{Poly}}X^{HN=\P},\]
 where \[X^{HN=\P}=\{x\in |X|| HN(H_x,\iota,\lambda)=\P\}\]which is a locally closed subset of $X$ by proposition 4.5. On the other hand there is a stratification by Newton polygons
 \[X=\coprod_{\P\in \textrm{Poly}}X^{Newt=\P},\]
 where \[X^{Newt=\P}=\{x\in |X|| Newt(H_{xk(x)},\iota,\lambda)=\P\}=sp^{-1}(X^{Newt=\P}_{red})\]which is a locally closed analytic domain of $X$. Here $k(x)$ is the residue field of the complete valuation $O_{\mathcal{H}(x)}$ associated to $x$, $X_{red}$ is the reduced special fiber of the formal scheme $\mathcal{X}$, $X^{Newt=\P}_{red}$ is the Newton polygon strata of $X_{red}$ for the polygon $\P$, and $sp: X\ra X_{red}$ is the specialization map, which is anti-continuous in the sense that $sp^{-1}(Y)$ is an open (resp. a closed) subset of $X$ if $Y$ is a Zariski closed (resp. open) subset of $X_{red}$. Let $\P_{ss}\in$Poly denote the line of slope $\frac{1}{2}$ between (0,0) and $(n,\frac{n}{2})$. Then proposition 4.4 tells us we have the inclusion
 \[X^{Newt=\P_{ss}}\subset X^{HN=\P_{ss}}.\]

Now we look at the unitary Rapoport-Zink space $\widehat{\M}$ introduced in section 2. Since it is basic, we have
\[\M=\M^{Newt=\P_{ss}}.\]Thus there is just one Harder-Narasimhan strata, i.e. the whole space
\[\M=\M^{HN=\P_{ss}}.\]In section 9 we will look at the Harder-Narasimhan stratification for some unitary $p$-adic Shimura varieties.

\section{An algorithm for $p$-divisible groups with additional structures}

In \cite{F3}, Fargues introduced an algorithm for $p$-divisible groups over complete valuation rings of rank one which is an extension of $\Z_p$, to produce $p$-divisible groups more close to those of HN-type, see loc. cit. for the definition of $p$-divisible groups of HN-type. For the case which we are interested in, it suffices to consider the formal $p$-divisible groups with special fibers supersingular and semi-stable $p$-divisible groups. Let $K|\Q_p$ be a complete field extension for a rank one valuation, and $O_K$ be its ring of integers. Let $H$ be a $p$-divisible group over $O_K$. Recall the following definition of Fargues (\cite{F3}, lemme 2, d\'efinition 4):
\begin{definition}
$H$ is called semi-stable if it satisfies one of the following three equivalent conditions:
\begin{itemize}
\item $H[p]$ is semi-stable;
\item for all $m\geq 1$, $H[p^m]$ is semi-stable;
\item for all finite flat subgroup scheme $G\subset H$, $\mu(G)\leq \mu_H:=\frac{dimH}{htH}(=\mu(H[p^m]),\,\forall m\geq1)$.
\end{itemize}
\end{definition}
For a finite flat group scheme $G$ over $O_K$, let $\mu_{max}(G)$ be the maximal slope of the Harder-Narasimhan polygon $HN(G)$ of $G$, then it is semi-stable if and only if $\mu_{max}(G)=\mu_G$. Thus for the $p$-divisible group $H$, it is semi-stable if and only if one of the following two another conditions holds:
\begin{itemize}
\item $\mu_{max}(H[p])=\mu_H$;
\item for all $m\geq 1$, $\mu_{max}(H[p^m])=\mu_H$.
\end{itemize}

For the $p$-divisible group $H$ over $O_K$, for all $k\geq 1$, we set
\[G_k=\textrm{ the first scran of the Harder-Narasimhan filtration of $H[p^k]$}.\]
Then one has for all $i\geq j\geq 1$
\[G_j=G_i[p^{j}]\subset G_i, p^{i-j}G_i\subset G_j,\]see \cite{F3} lemme 3 and the remark below there. In particular, the slopes $\mu(G_k)=\mu_{max}(H[p^k])$ do not change when $k\geq 1$ varies. We denote $\mu_{max}(H):=\mu_{max}(H[p^k])$ for any $k\geq 1$, and one can find that
\[\mu_{max}(H)=sup\{\mu(G)|G\subset H\}=sup\{\mu(G)|G\subset H[p]\}.\] We have thus always $\mu_{max}(H)\geq \mu_H$, and the equality holds if and only if $H$ is semi-stable.

 We set\[\Fm_H=\varinjlim_{k\geq1}G_k \subset H,\] considered as a sub-fppf sheaf of $H$. Then we have for all $k\geq1$, $\Fm_H[p^k]=G_k$ is a finite flat group scheme over $O_K$, and $\Fm_H=H$ if and only if $H$ is semi-stable.

Suppose $H$ is not semi-stable. Then lemme 4 of loc. cit. tells us there are two possibilities:
\begin{itemize}
\item $\Fm_H$ is a finite flat group scheme over $O_K$, that is there exists some $k_0\geq1$ such that $\Fm_H=\Fm_H[p^{k_0}]$;
\item there exists some integer $k_0\geq 1$ such that $\Fm_H/\Fm_H[p^{k_0}]$ is a semi-stable sub-$p$-divisible group of $H/\Fm_{H}[p^{k_0}]$ with $\mu_{\Fm_H/\Fm_H[p^{k_0}]}=\mu_{max}(H)>\mu_H$ .
\end{itemize}

As said above, we will only be interested in formal $p$-divisible groups over $O_K$, such that their special fibers are supersingular. We call such a formal $p$-divisible group basic. From now on we will suppose that $H$ is a basic modular (see d\'efinition 25 of \cite{F3}) $p$-divisible group over $O_K$. Then we only have the first possibility for $\Fm_H$, i.e. it is a finite flat subgroup scheme of $H$. This comes from the facts $HN(H)\leq Newt(H_k)$ (thus both of them are the line between $(0,0)$ and $(h,d)$) and one can read off the Harder-Narasimhan polygon of $H$ from the algorithm below.

The algorithm of Fargues for such a $p$-divisible group $H$, defines a sequence of $p$-divisible groups $(H_i)_{i\geq 1}$ by induction, with an isogeny $\phi_i: H_i\ra H_{i+1}$ for each $i\geq 1$. For $i=1$, we set $H_1=H$, and if $H_i\neq 0$, we set
\[H_{i+1}=H_i/\Fm_{H_i},\] and $\phi_i: H_i\ra H_{i+1}$ is the natural projection; if $H_i=0$ we set $H_{i+1}=0$ and $\phi_i$ the trivial morphism. Then by construction, \[\mu_{max}(H_{i+1})<\mu_{max}(H_i)=\mu(\Fm_{H_i})\] if $H_{i+1}\neq 0$. Note $\mu_H=\mu_{H_{i+1}}\leq \mu_{max}(H_{i+1})$ if $H_{i+1}\neq 0$. The section 8 of
 \cite{F3} tells us that if the valuation on $K$ is discrete, then the algorithm stops after finite times, i.e. $H_i=0$ for $i$ large enough. For the general valuation case, the main theorem of loc. cit. says if $(dimH,htH)=1$, then $H_i=0$ for $i>>0$.

 Until the end paragraph of this section we assume the valuation on $K$ is discrete. Then by the above discussion, for a basic formal $p$-divisible group $H$ over $O_K$, if it is not semi-stable, we have a sequence of $p$-divisible groups with each arrow between them an isogeny:
\[\xymatrix{&H=H_1\ar[r]^{\phi_1}\ar@/^2pc/[rrr]^{\phi}&H_2\ar[r]^{\phi_2} &\cdots\ar[r]^{\phi_{r}}& H_{r+1},}\] with $H_1,\cdots,H_{r}$ not semi-stable and $H_{r+1}$ semi-stable of slope $\mu=\frac{d}{h}$,
and for each $i=1,\cdots,r$, the kernel of the isogeny $ker(\phi_i)$ is the first scran of the Harder-Narasimhan filtration of $H_{i}[p^{n_i}]$ for some $n_i >>0$. There exists some $N>>0$ such that the kernel $ker\phi$ of the composition of these isogenies $\phi$ is contained in $H[p^N]$. By construction, $ker\phi$ is in fact a scran in the Harder-Narasimhan filtration of $H[p^N]$ and these $ker\phi_i$'s are exactly the sub-quotient factors of the Harder-Narasimhan filtration of $ker\phi$. If we denote $\mu_i=\mu(ker\phi_i)$, then
\[\mu_1>\mu_2>\cdots >\mu_r >\frac{d}{h}.\]
\\

Now we consider $p$-divisible groups with additional structures. First, the above construction still works in the EL case, that is $p$-divisible groups with actions of the integer ring $O_F$ of some finite unramified extension $F|\Q_p$, since the Harder-Narasimhan filtration is of $O_F$-invariant, cf. \cite{F2} or section 3. Next, we consider PEL cases, that is a $p$-divisible group $H$ over $O_K$, with action $\iota: O_F\ra End(H)$ of the integer ring $O_F$ of some finite unramified extension $F|\Q_p$, and a polarization $\lambda: H\ra H^D$, such that $\iota$ and $\lambda$ are compatible in the sense of definition 4.2. In particular, we can apply the unitary $p$-divisible groups studied above to this situation. So let $(H,\iota,\lambda)$ be a $p$-divisible group with (PEL) additional structures, such that $H$ is basic. Assume $H$ is not semi-stable, then we have a sequence of $O_F$-linear isogenies of $p$-divisible groups with actions of $O_F$:
\[\xymatrix{&H=H_1\ar[r]^{\phi_1}\ar@/^2pc/[rrr]^{\phi}&H_2\ar[r]^{\phi_2} &\cdots\ar[r]^{\phi_{r}}& H_{r+1},}\] with $H_1,\cdots,H_{r}$ not semi-stable and $H_{r+1}$ semi-stable of slope $\mu=\frac{1}{2}$, and there exists some $N>>0$ such that $E:=ker\phi\subset H[p^N]$ and $E\nsubseteq H[p^{N-1}]$. Then $E$ is a scran in the Harder-Narasimhan filtration of $H[p^N]$. Let
\[0=E_0\subsetneq E_1\subsetneq \cdots \subsetneq E_r=E\] be the Harder-Narasimhan filtration of $E$, then we have
\[\begin{split}
&E_i/E_{i-1}\simeq ker\phi_i\\
&\mu_1>\cdots >\mu_r>\frac{1}{2},
\end{split}\] where $\mu_i:=\mu(ker\phi_i)$  for $i=1,\cdots,r$. The polarization $\lambda$ on $H$ now induces a polarization on $H[p^N]$: $\lambda: H[p^N]\stackrel{\sim}{\ra}H[p^N]^D$. Thus there is a perfect pairing \[H[p^N]\times H[p^N]\ra \mu_{p^{2nN}}.\] Let $E_i^{\perp}$ be the orthogonal subgroup of $H[p^N]$ under this pairing, for $i=1,\cdots,r$. Since $E_i$ is a scran of the Harder-Narasimhan filtration of $H[p^N]$, so is $E_i^{\perp}$ by the compatibility of Harder-Narasimhan filtration with polarizations, \cite{F2} 5.2 or section 3. Moreover, we have the following inclusions:
\[0\subsetneq E_1\subsetneq\cdots\subsetneq E_{r-1} \subsetneq E_r=E\subset E^\perp=E_r^\perp\subsetneq E_{r-1}^\perp\subsetneq\cdots \subsetneq E_1^\perp\subsetneq H[p^N],\] and the equalities
\[\mu_i+\mu(E_{i-1}^\perp/E_i^\perp)=1,\]for $i=1,\cdots,r$. Recall that we have \[\mu_1>\mu_2>\cdots>\mu_r>\frac{1}{2}.\] For all $k\geq N$, the finite flat group schemes $E_1,\cdots,E_r\subset H[p^k]$ are the same and do not depend on $k$, but their orthogonal groups $E_1^{\perp_k},\cdots,E_r^{\perp_k}$ for the pairing of $H[p^k]$ do depend the group $H[p^k]$. Thus for $k\geq N$ varies, just the height of $ht(E^{\perp_k}/E)$ varies and $ht(E_i/E_{i-1})=ht(E^{\perp_k}_{i-1}/E^{\perp_k}_i)$ do not change for $i=1,\dots,r-1$.

There are two different cases: $E=E^\perp$ or $E\subsetneq E^\perp$.
\begin{enumerate}
\item $E=E^\perp$, i.e. there is no slope of $\frac{1}{2}$ in the Harder-Narasimhan filtration of $H[p^N]$. This case is good, since the polarization $\lambda$ on $H$ then induces a polarization $\lambda': H/E\ra (H/E)^D$, such that $\pi\circ\lambda'\circ\pi^D=p^N\lambda$, where $\pi: H\ra H/E$ is the natural projection, i.e. we have the following commutative diagram:
   \[ \xymatrix{ &H\ar[r]^{p^N\lambda}\ar[d]^{\pi}&H^D\\
   &H/E\ar[r]^{\lambda'}&(H/E)^D\ar[u]_{\pi^D}.}\]
   Thus in this case we get a $p$-divisible group with naturally induced additional structures $(H/E,\iota',\lambda')$.

\item $E\subsetneq E^\perp$, i.e. $E^\perp/E$ is a factor in the Harder-Narasimhan filtration of $H[p^N]$ of slope $\frac{1}{2}$. Note there is a natural perfect pairing \[(E^\perp/E)\times (E^\perp/E)\ra \mu_{p^{ht(E^\perp/E)}}.\] Let $C:=E^\perp/E$. For a subgroup $C'\subset C$, let $C^{'\perp}$ be the orthogonal complement of $C'$ for the above pairing on $C=E^\perp/E$. We make the following claim:
\begin{claim}
there is a filtration of sub-semi-stable groups of slope $\frac{1}{2}$
\[0\subsetneq C_1\subsetneq\cdots\subsetneq C_k\subset C_k^\perp\subsetneq\cdots\subsetneq C_1^\perp\subsetneq C,\]
such that $p (C_k^\perp/C_k)=0$.
\end{claim}

In fact, if we let $m$ be the minimal integer such that $p^mC=0$, if $m=1$ we are done; so assume $m\geq 2$ now. Consider $0\neq p^{m-1}C\subsetneq C$, which is also semi-stable of slope $\frac{1}{2}$. Then we have a filtration of semi-stable groups of slope $\frac{1}{2}$
 \[0\neq p^{m-1}C\subsetneq (p^{m-1}C)^\perp=C[p^{m-1}]\subsetneq C.\]
 Now \[ht((p^{m-1}C)^\perp/p^{m-1}C)<htC,\] and set $C'=(p^{m-1}C)^\perp/p^{m-1}C$, by induction we thus have the above claim.

Now we can translate the above filtration to a filtration of subgroups of $E^\perp\subset H[p^N]$, that is there exists a filtration
\[E\subsetneq E^1\subsetneq \cdots\subsetneq E^{k}\subset E^{k\perp}\subsetneq\cdots\subsetneq E^\perp,\]such that $E^i/E=C^i\subset E^\perp/E=C$. Let $E':=E^k$, then since $E'/E$ is semi-stable, $H/E'$ is semi-stable. We still have two cases.
\begin{enumerate}
\item If $C_k=C_k^\perp$ that is $E'=E^{'\perp}$,  this is still good in this case: we have the semi-stable $p$-divisible group with additional structures $(H/E',\iota',\lambda')$.

\item If $E'\subsetneq E^{'\perp}$, we have the following proposition.

\begin{proposition}
Let the notation be as above, and $(H,\iota,\lambda)$ be a $p$-divisible group with additional structures. Assume that $E'\subsetneq E^{'\perp}$. Then after changing $N$ to $N+1$ if $N$ is odd in the PEL unitary case, there exists some finite extension $K'|K$ and a totally isotrope subgroup $E''=(E'')^\perp$ of $H[p^N]$ over $O_{K'}$.

\end{proposition}
\begin{proof}
Let $V:=(E^{'\perp}/E')(\ov{K})$, since $p(E^{'\perp}/E')=0$, it is a $\F_p$-vector space equipped with an action of the Galois group $Gal(\ov{K}/K)$. Moreover, there is an induced $\F_{p^d}$-action $\iota:\F_{p^d}\ra End(V)$($d=[F:\Q_p],\iota: O_F\ra End(H)$), so we can view $V$ as a $\F_{p^d}$-vector space via $\iota$. The pairing $\lan,\ran$ on $E^{'\perp}/E'$ induces a hermitian form $V\times V\ra \F_{p^d}$. By assumption $dim_{\F_{p^d}}V=2m$ for some integer $m\geq1$. Thus there exists a maximal totally isotrope subspace $W\subset V, W=W^\perp$, and a finite extension $K'|K$ such that $W$ is stable by $Gal(\ov{K}/K')$. Then the schematic closure of $W$ in $E^{'\perp}/E'$ over $O_{K'}$ corresponds to a totally isotrope subgroup $E''=(E'')^\perp$ of $H[p^N]$ over $O_{K'}$.
\end{proof}

Let $K'|K$, and $E''\subset H[p^N]$ be as above. Then the $p$-divisible group $H/E''$ over $O_{K'}$ admits naturally induced additional structures: $\iota': O_F\ra End(H/E''), \lambda': H/E''\ra (H/E'')^D$ such that the following diagram commutes:
\[ \xymatrix{ &H\ar[r]^{p^N\lambda}\ar[d]^{\pi}&H^D\\
   &H/E''\ar[r]^{\lambda'}&(H/E'')^D\ar[u]_{\pi^D}.}\]
Recall $H/E'$ is a semi-stable $p$-divisible group. Thus we have an isogeny $f: H/E''\ra H/E'$ of $p$-divisible groups over $O_{K'}$ such that $p(kerf)=0$.

\end{enumerate}

\end{enumerate}

Now for the case that the valuation ring $O_K$ is not discrete, let $(H,\iota,\lambda)$ be a $p$-divisible group with additional structures over $O_K$. Assume $H$ is not semi-stable. We still have Fargues's algorithm
\[H=H_1\stackrel{\phi_1}{\ra}H_2\st{\phi_2}{\ra}\cdots\stackrel{\phi_{i-1}}{\ra}H_i\stackrel{\phi_{i}}{\ra}\cdots,\]
with \[\Fm_H=ker\phi_1,\; \Fm_{H_2}=ker\phi_2,\cdots.\]If the algorithm stops after finite times, i.e. there exists some $r$ such that $H_r\neq 0$ and $H_i=0$ for all $i\geq r+1$, in which case $H_r$ is semi-stable, we can continue our procedure as above to find the $E,\; E'$ and $E''$. Thus once the algorithm stops after finite times, we can continue as above and get a modified algorithm for the case with additional structures.

\section{The analytic domain $\C$}

We return to the study of the $p$-adic analytic unitary group Rapoport-Zink space $\M$ introduced in section 2. Let $K|L=W(\ov{\F}_p)_\Q$ be a complete discrete valuation field, then for any $K$-valued point $x\in \M(K)$, by the algorithm of last section, we can associate to it an another point $x'\in \M(K')$, where $K'|K$ is a finite extension, such that if $(H_x,\iota,\lambda)$ (resp. $(H_{x'},\iota',\lambda')$) is the $p$-divisible group associated to $x$ (resp. $x'$), then we have an isogeny $\phi: H_x\ra H_{x'}$ of $p$-divisible groups over $O_{K'}$, satisfying the following commutative diagram:
\[ \xymatrix{ &H_x\ar[r]^{p^N\lambda}\ar[d]^{\phi}&H_x^D\\
   &H_{x'}\ar[r]^{\lambda'}&(H_{x'})^D\ar[u]_{\phi^D},}\]
   for some integer $N$, and there is a finite flat subgroup scheme $G\subset H_{x'}[p]$, such that $H_{x'}/G$ is semi-stable. Motivated by this, we introduce a subspace $\C\subset\M$ as follows.
\begin{definition}
We define a subspace $\C\subset\M$ as
\[\begin{split}\C=\{x\in\M|\,&\exists\,\textrm{finite extension}\,K'|\mathcal{H}(x),\,\textrm{and a finite flat subgroup}\, G\subset H_x[p]\,\textrm{over}\,O_{K'},\\ &\textrm{such that}\, H_x/G \, \textrm{is semi-stable over}\,O_{K'}\}.
\end{split}\]
\end{definition}
Let $\M^{ss}\subset\M$ be the semi-stable locus, that is
\[\M^{ss}=\{x\in\M|\,H_x \,\textrm{is semi-stable}\}.\] Then we have the inclusion \[\M^{ss}\subset \C.\]

\begin{proposition}
The subset $\M^{ss}$ and $\C$ are closed analytic domains of $\M$.
\end{proposition}
\begin{proof}
The fact that $\M^{ss}\subset\M$ is a closed analytic domain is an easy consequence of Theorem 3.5. So we concentrate here to prove that $\C\subset\M$ is a closed analytic domain.

Let $\mathcal{N}$ be the basic Rapoport-Zink analytic space for $Res_{\Q_{p^2}|\Q_p}GL_{n}$ obtained by forgetting the polarization from $\M$. Then there is a natural closed immersion $\M\subset\mathcal{N}$. We fix such an imbedding. We have the inclusions $G(\Q_p)\subset GL_{n}(\Q_{p^2})$, and $G(\Z_p)\setminus G(\Q_p)/G(\Z_p)\hookrightarrow GL_{n}(\Z_{p^2})\setminus GL_{n}(\Q_{p^2})/GL_{n}(\Z_{p^2})$. We have in fact a $G(\Q_p)$-equivariant imbedding of tower of analytic spaces $\M_{K\cap G(\Z_p)}\subset \mathcal{N}_K$ for open compact subgroups $K\subset GL_n(\Q_{p^2})$. We have $\M^{ss}=\N^{ss}\bigcap\M$. By definition, the subset $\C$ is exactly the intersection with $\M$ of some Hecke translations of the semi-stable locus $\mathcal{N}^{ss}\subset\mathcal{N}$:
\[\C=(\bigcup_{\underline{a}}T_{\underline{a}}.\mathcal{N}^{ss})\bigcap\M,\]where the index set is all $\underline{a}=(a_1,\dots,a_n)\in\{(a_1,\dots,a_n)|a_1\geq\cdots\geq a_n, \textrm{and}\,a_i\in\{0,-1\}\}$, and \[T_{\underline{a}}=GL_{n}(\Z_{p^2})\left(\begin{array}{ccc}
                     p^{a_1} & & \\
                     &\ddots& \\
                     & &p^{a_n}\\
                   \end{array}
                   \right) GL_{n}(\Z_{p^2})\in
GL_{n}(\Z_{p^2})\setminus GL_{n}(\Q_{p^2})/GL_{n}(\Z_{p^2}).\]
Now since $\mathcal{N}^{ss}\subset \mathcal{N}$ is a closed analytic domain, so is $\C\subset\M$ by the definition of Hecke correspondences.

\end{proof}

We observe the characterization of points in $\C$ as follows.
\begin{proposition}
A point $x\in\M$ is in $\C$ if and only if the algorithm for the $p$-divisible group $H_x$ associated to $x$ stops after finite times, and $N_x=1$ , where $N_x$ is the smallest integer such that $ker\phi_x\subset H_x[p^{N_x}]$, $\phi_x$ is the composition of the isogenies when applying the algorithm to $H_x$.
\end{proposition}
\begin{proof}
If the algorithm for $H_x$ stops after finite times and $N_x=1$, then by definition $x\in\C$. To prove the other direction, we have the following general lemma.
\begin{lemma}
Let $H/O_K$ be a basic $p$-divisible group over a complete rank one valuation ring $O_K|\Z_p$, and $G\subset H$ be a finite flat subgroup scheme. If $H/G$ is semi-stable, then $\Fm_H\subset G$. In particular, if the sequence of isogenies of $p$-divisible groups coming from the algorithm above
 \[\xymatrix{&H=H_1\ar[r]^{\phi_1}\ar@/^2pc/[rrr]^{\phi}&H_2\ar[r]^{\phi_2} &\cdots\ar[r]^{\phi_{r}}& H_{r+1}}\] is such that $H_1,\dots,H_{r}$ are not semi-stable ($H_{r+1}$ may be semi-stable or may be not),
then $ker\phi\subset G$.
\end{lemma}
\begin{proof}
Let $0\neq G'\subset H, G'\nsubseteq G$ be a finite flat subgroup not contained in $G$. Consider the morphism $\varphi: G'\ra H[p^N]/G$ for $N>>0$. Then it is non zero. Let $G'''$ (resp. $G''$) be the flattening schematic image (resp. kernel) of $\varphi$, then we have the following sequence which is exact in generic fiber:
\[0\ra G''\ra G'\ra G'''\ra0.\]
Since $H/G$ is semi-stable, and $G'''\subset H/G$ is a finite flat subgroup, by definition \[\mu(G''')\leq \mu_{H/G}=\mu_H.\] On the other hand we have $G''\subset G$. If $G''=0$, then $\mu(G')\leq\mu(G''')\leq\mu_H$. If $G''\neq 0$, $\mu(G')\leq sup\{\mu(G''),\mu(G''')\}\leq sup\{\mu_{max}(G),\mu_H\}$. We have two cases:
\begin{enumerate}
\item if $\mu_{max}(G)\leq \mu_H$, then since $\mu(G')\leq \mu_H$ for any $0\neq G'\nsubseteq G$, we have $H$ is semi-stable. In particular $\Fm_H=0\subset G$.
\item if $\mu_{max}(G)>\mu_H$, then since $\mu(G')< \mu_{max}(G)$, we have $\mu_{max}(H)=\mu_{max}(G)$, and $\Fm_H\subset G$ is the first scran of the Harder-Narasimhan filtration of $G$.
\end{enumerate}
Thus the lemma holds.
\end{proof}

With the lemma above, we can now easily deduce the proposition. If $x\in \C$, then by the definition of $\C$, there exist a finite extension $K'|K=\mathcal{H}(x)$ and a finite flat subgroup $G\subset H[p]$ over $O_{K'}$, such that $H/G$ is semi-stable over $O_{K'}$. The algorithm for $H$ over $O_{K'}$ is just the base change of that over $O_K$. By the lemma, $ker\phi\subset G\subset H[p]$. Thus the algorithm stops after finite times and $N_x=1$.

\end{proof}

\begin{example}
For $n=1$ the Rapoport-Zink space $\M$ is trivial: each connected component $\M^i$ is just a point. Thus in this case $\C=\M$. For $n\geq 2$, it is unluckily difficult to describe the domain $\C$ explicitly. Here we calculate $\C$ for the case $n=2$. In this case each reduced special fiber $\M^i_{red}$ is just a point, while the analytic space $\M^i$ is of dimension 1. Let $\mathcal{N}$ be the basic Rapoport-Zink analytic space for $Res_{\Q_{p^2}|\Q_p}GL_{2}$ obtained by forgetting the polarization from $\M$. We have \[\C=(\bigcup_{\underline{a}}T_{\underline{a}}.\mathcal{N}^{ss})\bigcap\M.\]There are 3 possibilities for the index $
\underline{a}=(a_1,a_2)$: (0,0),(0,-1),(-1,-1). Let $(H,\iota)$ be the $p$-divisible group associated to a point $x \in \mathcal{N}^{ss}$. The Hecke correspondences $T_{(0,0)}$ is the identity, and $T_{(-1,-1)}.x$ is the quotient of $H$ by $H[p]$ with its additional structure, thus $T_{(-1,-1)}.x\subset\mathcal{N}^{ss}$. For the Hecke correspondence $T_{(0,-1)}$, a point $y\in T_{(0,-1)}.x$ corresponds to a height 2 finite flat subgroup $G\subset H[p]$, and the $p$-divisible group associated to $y$ is the quotient $(H/G,\iota')$. Since $H$ is semi-stable, we have $\mu(G)\leq \frac{1}{2}$, i.e. $degG\leq 1$. The arbitrary possibility of choices of $G$ will make $T_{(0,-1)}.\N^{ss}\nsubseteq \N^{ss}$. Thus the inclusion $\M^{ss}\subset \C$ is strict.
\end{example}

Similarly we have a characterization of the points in the Hecke orbit of $\C$.
\begin{proposition}
A point $x\in\M$ is in $\bigcup_{T\in G(\Z_p)\setminus G(\Q_p)/G(\Z_p)}T.\C$ if and only if the algorithm above for the $p$-divisible group $H_x$ associated to $x$ stops after finite times.

\end{proposition}
For a Hausdorff paracompact strictly Berkovich analytic space $X$ over a non-archmidean field, we denote by $X^{rig}$ the associated rigid analytic space in the sense of Tate. As a set, $X^{rig}\subset X$ is the subset of rigid analytic points. Then the Hecke orbit of $\C^{rig}$ cover $\M^{rig}$.
\begin{proposition}
We have the following equality of sets
\[\M^{rig}=\bigcup_{T\in G(\Z_p)\setminus G(\Q_p)/G(\Z_p)}T.\C^{rig}.\]
\end{proposition}
\begin{proof}
By the inequality \[HN(H,\iota,\lambda)\leq Newt(H_k,\iota,\lambda),\] there is just one Harder-Narasimhan strata, i.e. the whole space $\M=\M^{HN=\P_{ss}}$. By the algorithm, we have the equality $\M^{HN=\P_{ss},rig}=\bigcup_{T\in G(\Z_p)\setminus G(\Q_p)/G(\Z_p)}T.\C^{rig}.$
\end{proof}
We want some locally finite covering of $\M$. One may wonder whether the family of analytic domains $(T.\C)_{T\in G(\Z_p)\setminus G(\Q_p)/G(\Z_p)}$ is such a locally finite covering of $\M$.
Unfortunately, the analytic domain $\C$ is so big that the union $\bigcup_{T\in G(\Z_p)\setminus G(\Q_p)/G(\Z_p)}T.\C$  is far from locally finite. We have to refine this family. Nevertheless, this family is indeed a covering of $\M$, although it is not locally finite. In the next section, we will review some basic facts about the geometry of reduced special fiber $\M_{red}$ over $\ov{\F}_p$ obtained by Vollaard-Wedhorn in \cite{V} and \cite{VW}. Then we will define some smaller analytic domain $\D\subset \C$ such that under the Hecke action and the action of $J_b(\Q_p)$, we can get some locally finite covering of $\M^{rig}$. Finally by some gluing arguments, and the equivalence between the categories of Berkovich anayltic spaces and rigid analytic spaces satisfying certain conditions, we can get an equality for all analytic points
\[\M=\bigcup_{T\in G(\Z_p)\setminus G(\Q_p)/G(\Z_p)}T.\C.\]That is, for all $x\in\M$, the algorithm for $H_x$ stops after finite times.

\section{Bruhat-Tits stratification of $\M_{red}$}

We recall some of basic results of Vollaard-Wedhorn in \cite{V} and \cite{VW}. First we remark that the descent data on $\M_{red}$ is effective, i.e., there is a model $\M'_{red}$ over $\F_{p^2}$ of $\M_{red}$. The results of Vollaard-Wedhorn are rather about the scheme $\M'_{red}$, but we just state them for $\M_{red}$ here. Recall that the formal Rapoport-Zink space $\widehat{\M}$ has a decomposition according to the height of the universal quasi-isogeny:
\[\widehat{\M}=\coprod_{i\in \Z,in\, even}\widehat{\M}^i,\] where $\widehat{\M}^i$ is the open and closed formal subscheme of $\widehat{\M}$ such that for any scheme $S\in$ Nilp$O_L$, a $S$-valued point of $\widehat{\M}$
  $(H,\iota,\lambda,\rho)\in\widehat{\M}^i(S)$ if $ht\rho=in$, which is not empty if and only if $in$ is even. Let $\M^i=(\widehat{\M}^{i})^{an}$, and $\M_{red}^i$ be its reduced special fiber for such an $i$, we have decompositions
\[\M=\coprod_{i\in\Z, in \,even}\M^i,\,\M_{red}=\coprod_{i\in\Z, in \,even}\M_{red}^i.\]

By Theorem 4.2.(1) of \cite{VW}, $\M_{red}^i$ is connected of pure dimension $[\frac{n-1}{2}]$. Thus $\M^i$ is an connected analytic space of dimension $n-1$. For each $i\in\Z$ such that $in$ is even, there is a $g\in J_b(\Q_p)$ such that $g(\widehat{\M}^i)=\widehat{\M}^0$, in particular $\widehat{\M}^i$ is isomorphic to $\widehat{\M}^0$, and so are their analytic fibers $\M^i,\M^0$ and reduced special fibers $\M_{red}^i,\M_{red}^0$.
So we just need to consider the $\M^0$ and $\M_{red}^0$. It turns out the geometry of $\M_{red}^0$ over $\ov{\F}_p$ is controlled by the Bruhat-Tits building $\mathcal{B}(J_b^{der},\Q_p)$ of the derived subgroup $J_b^{der}$ of $J_b$ over $\Q_p$.

More precisely, for each $i\in \Z$ such that $in$ is even, let
\[\mathcal{L}_i:=\{\Lambda\subset \mathbf{N}_0 \,\Z_{p^2}\textrm{-lattice}\,|p^{i+1}\Lambda^{\vee}\subsetneq \Lambda\subset p^i\Lambda^\vee\},\]where $\Lambda^\vee=\{x\in\mathbf{N}_0|\,\{x,\Lambda\}\subset\Z_{p^2}\}$.
One can construct an abstract simplicial complex $\mathcal{B}_i$ from $\mathcal{L}_i$: an $m$-simplex of $\mathcal{B}_i$ is a subset $S\subset \mathcal{L}_i$ of $m+1$ elements which satisfies the following condition. There exists an ordering $\Lambda_0,\dots,\Lambda_m$ of the elements of $S$ such that
\[p^{i+1}\Lam_m\subsetneq\Lam_0\subsetneq\Lam_1\subsetneq\cdots\subsetneq\Lam_m.\]
There is an obvious action of $J_b^{der}(\Q_p)$ on $\mathcal{L}_i$. By Theorem 3.6 of \cite{V}, for each fixed $i$, we have a natural $J_b^{der}(\Q_p)$-equivariant isomorphism of $\mathcal{B}_i$ with the associated simplicial complex of the Bruhat-Tits building $\mathcal{B}(J^{der}_b,\Q_p)$. Thus we can identify $\mathcal{L}_i$ with the set of vertices of $\mathcal{B}(J^{der}_b,\Q_p)$. For $\Lam\in\mathcal{L}_i$ the index $t(\Lam):=[\Lam:p^{i+1}\Lam^\vee]$ of $p^{i+1}\Lam^\vee$ in $\Lam$ is always an odd number with $1\leq t(\Lam)\leq n$, and for any odd number $d$ with $1\leq d\leq n$ there exists a $\Lam\in\mathcal{L}_i$ such that $t(\Lam)=d$. Moreover two lattices $\Lam_1,\Lam_2\in\mathcal{L}_i$ are in the same $J_b^{der}(\Q_p)$-orbit if and only if $t(\Lam_1)=t(\Lam_2)$. And the neighborhood vertices of $\Lam\in\mathcal{L}_i$ in the building is exactly the set \[\{\Lam'\in\mathcal{L}_i|\Lam'\subset\Lam, \,\textrm{or}\, \Lam\subset\Lam'\}.\] If $n$ is even, we choose and fix a $g_1\in J_b(\Q_p)$ such that $g_1:\widehat{\M}^0\st{\sim}{\ra}\widehat{\M}^1$. We fix a bijection
\[\varphi_i:\mathcal{L}_0\ra\mathcal{L}_i\] once for all, such that
\[\varphi_i(\Lam)=\begin{cases} p^{\frac{i-1}{2}}g_1\Lam& i \, odd\\
p^{\frac{i}{2}}\Lam & i\,even. \end{cases}\]

Now for each $i\in\Z$ such that $in$ is even and each $\Lam\in\mathcal{L}_i$, we define a subscheme $\M_{\Lam}$ of $\M_{red}^i$. First we associate two $p$-divisible groups $H_{\Lam^-}$ and $H_{\Lam^+}$ over $\F_{p^2}$ with $\Z_{p^2}$-linear polarizations $\lambda_{\Lam^+}$ and $\lambda_{\Lam^-}$ respectively. To this end, set
\[\begin{split}&\Lam_0^+:=\Lam\\
&\Lam_1^+:=\mathbf{V}^{-1}(\Lam_0^+)\\
&\Lam^+:=\Lam_0^+\oplus\Lam_1^+\\
&\Lam^-:=p^i(\Lam^+)^\vee.
\end{split}\]
Since $\mathbf{F}=\mathbf{V}$ the $\Lam^{\pm}$ are Dieudonn\'{e} submodules of the isocrystal $\mathbf{N}$, and the pairing $p^{-i}\lan,\ran$ on $\mathbf{N}$ induces a perfect $\Z_{p^2}$-pairing on $\Lam^{\pm}$. Thus $\Lam^{\pm}$ define unitary $p$-divisible groups $H_{\Lam^{\pm}}$, with the $\Z_{p^2}$-linear polarizations $\lambda_{\Lam^{\pm}}$ and $p^{-i}\lan,\ran$ induces an isomorphism \[H_{\Lam^+}\stackrel{\sim}{\ra} H_{\Lam^-}^D.\] Moreover, we have $\Z_{p^2}$-linear quasi-isogenies
\[\rho_{\Lam^{\pm}}: H_{\Lam^{\pm}}\ra \mathbf{H}\] which are compatible with the polarizations on the two sides. We have the following commutative diagram:
\[\xymatrix{H_{\Lam^+}\ar[r]^\sim \ar[d]^{\rho_{\Lam^+}}& H_{\Lam^-}^D\\
\mathbf{H}\ar[r]^{\rho}&\mathbf{H}^D\ar[u]_{\rho_{\Lam^-}^D}.
}\]
By construction we have always $\Lam^-\subset\Lam^+$, which corresponds to the composition of quasi-isogenies
\[H_{\Lam^-}\stackrel{\rho_{\Lam^-}}{\longrightarrow}\mathbf{H}
\stackrel{\rho_{\Lam^+}^{-1}}{\longrightarrow}H_{\Lam^+}.\]

We a fixed vertex $\Lam\in\mathcal{L}_i$. For any $\ov{\F}_{p^2}$-scheme $S$ and a $S$-valued point $(H,\iota,\lambda,\rho)\in\M_{red}^i(S)$ we define quasi-isogenies
\[\begin{split} &\rho_{H,\Lam^+}: H\stackrel{\rho^{-1}}{\longrightarrow}\mathbf{H}_S\stackrel{(\rho_{\Lam^+})^{-1}_S}{\longrightarrow}(H_{\Lam^+})_S,\\
&\rho_{\Lam^-,H}: (H_{\Lam^-})_S\stackrel{(\rho_{\Lam^-})_S}{\longrightarrow}\mathbf{H}_S\stackrel{\rho}{\longrightarrow}H.
\end{split}\]
Then one has that
\[ht(\rho_{H,\Lam^+})=ht(\rho_{\Lam^-,H})=t(\Lam)\]
and that $\rho_{H,\Lam^+}$ is an isogeny if and only if $\rho_{\Lam^-,H}$ is an isogeny.

\begin{definition}
We define the subfunctor $\M_{\Lam}\subset\M_{red}^i$ as
\[\M_{\Lam}(S)=\{(H,\iota,\lambda,\rho)\in\M_{red}^i(S)|\rho_{\Lam^-,H} \,\textrm{is an isogeny}\}\]
for any $\ov{\F}_p$-scheme $S$.
\end{definition}

Then the main theorems of \cite{VW} tell us the following facts.
\begin{theorem}
\begin{enumerate}
\item $\M_\Lam$ is represented by a smooth projective closed subscheme of dimension $\frac{1}{2}(t(\Lam)-1)$  of $\M_{red}^i$, which we will also denote by $\M_\Lam$. It is in fact a generalized Deligne-Lusztig variety for the maximal reductive quotient $\ov{J}^{der,red}_{\Lam}$ over $\F_p$ of the special fiber of the Bruhat-Tits group scheme $J^{der}_\Lam$ attached to the vertex $\Lam\in \mathcal{B}(J^{der}_{b},\Q_p)$.
\item for two lattices $\Lam_1,\Lam_2\in \mathcal{L}_i$, $\M_{\Lam_1}\subset\M_{\Lam_2}$ if and only if $\Lam_1\subset\Lam_2$. In this case $t(\Lam_1)\leq t(\Lam_2)$, and the equality holds if and only if $\Lam_1=\Lam_2$.
\item for two lattices $\Lam_1,\Lam_2\in \mathcal{L}_i$, the following assertions are equivalent:
\begin{itemize}
\item $\Lam_1\cap\Lam_2\in\mathcal{L}_i$;
\item $\Lam_1\cap\Lam_3$ contains a lattice of $\mathcal{L}_i$;
\item $\M_{\Lam_1}\cap\M_{\Lam_2}\neq\emptyset$.
\end{itemize}
If these conditions are satisfied we have
\[\M_{\Lam_1}\cap\M_{\Lam_2}=\M_{\Lam_1\cap\Lam_2},\]where $\M_{\Lam_1}\cap\M_{\Lam_2}$ is the scheme-theoretic intersection in $\M_{red}^i$.
\item for two lattices $\Lam_1,\Lam_2\in \mathcal{L}_i$, the following assertions are equivalent:
\begin{itemize}
\item $\Lam_1+\Lam_2\in\mathcal{L}_i$;
\item $\Lam_1+\Lam_2$ is contained in a lattice of $\mathcal{L}_i$;
\item $\M_{\Lam_1}$ and $\M_{\Lam_2}$ are both contained in $\M_{\Lam}$ for some $\Lam\in\mathcal{L}_i$.
\end{itemize}
If these conditions are satisfied, $\M_{\Lam_1+\Lam_2}$ is the smallest subscheme of the form $\M_{\Lam}$ that contains $\M_{\Lam_1}$ and $\M_{\Lam_2}$.
\item let $t_{max}=n$ if $n$ is odd, and $t_{max}=n-1$ if $n$ is even, then the irreducible components of $\M_{red}^i$ are exactly the subschemes $\M_\Lam$ with $t(\Lam)=t_{max}$.
\item let \[\begin{split}&\mathcal{L}_\Lam:=\{\Lam'\in\mathcal{L}_i|\Lam'\subsetneq\Lam\},\\
&\M_\Lam^0:=\M_\Lam\setminus \bigcup_{\Lam'\in\mathcal{L}_\Lam}\M_{\Lam'},
\end{split}\] then $\M_{\Lam}^0$ is open and dense in $\M_\Lam$, and we have a stratification of $\M_{red}^i$:
\[\M_{red}^i=\coprod_{\Lam\in\mathcal{L}_i}\M_\Lam^0.\]
\end{enumerate}
\end{theorem}
\begin{proof}
These are the contents of lemma 3.2, theorem 3.10, corollary 3.11, theorem 4.1, 4.2, and proposition 4.3 of \cite{VW}.
\end{proof}

Note that the stratification $\M_{red}^i=\coprod_{\Lam\in\mathcal{L}_i}\M_\Lam^0$ is $J^{der}_b(\Q_p)$-equivariant in the sense that we have \[g\M^0_{\Lam}=\M_{g\Lam}^0,\]for any $g\in J^{der}_b(\Q_p)$ and $\Lam\in\mathcal{L}_i$. For an algebraically field $k|\F_{p^2}$ and a $k$-valued point $x\in\M_{red}^i(k)$, if we denote $M$ the corresponding unitary Dieudonn\'{e} module viewed as a lattice in $\mathbf{N}$ via the quasi-isogeny $\rho_x^{-1}: H_x\ra\mathbf{H}$, then we have the following equivalent assertions
\begin{itemize}
\item$x\in\M_\Lam(k)$ ;
\item $M\subset (\Lam^+)_k$;
\item $M_0\subset (\Lam)_k$;
\item $(\Lam^-)_k\subset M$.
\end{itemize}

\section{The analytic domain $\D$}

Recall that we have defined some closed analytic domains
$\M^{ss}\subset\C\subset\M,$
where for any complete valuation field extension $K|L=W(\ov{\F}_p)_\Q$,
\[\begin{split}
\M^{ss}(K)=&\{(H,\iota,\lambda,\rho)\in\M(K)|\,H\,\textrm{is semi-stable}\}\\
\C(K)=&
\{(H,\iota,\lambda,\rho)\in\M(K)|\,\exists \,\textrm{finite extension}\,K'|K, \textrm{and a finite flat subgroup}\\
&G\subset H[p]\,\textrm{over}\,O_{K'},\textrm{such that}\, H/G \,\textrm{is semi-stable}\}.
\end{split}\]
Now since we have the decomposition
\[\M=\coprod_{i\in\Z,in \,even}\M^i,\]
we set \[\C^i:=\C\cap\M^i,\] which is still a closed analytic domain in $\M^i$ and $\M$, and we have an induced decomposition of analytic spaces
\[\C=\coprod_{i\in\Z,in\, even}\C^i.\]

We choose an element $g_1\in J_b(\Q_p)$ such that the action by $g_1$ on $\M$ induces
isomorphisms:
\[g_1:\M^0\stackrel{\sim}{\ra}\begin{cases}\M^1& n\, even\\
\M^2&n\,odd.\end{cases}\] For example let \[g_1=diag(\underbrace{p^{-1},\cdots,p^{-1}}_{\frac{n}{2}},\underbrace{1,\cdots,1}_{\frac{n}{2}})\] if $n$ is even (by fixing a base one can view $J_b(\Q_p)\subset GL_n(\Q_{p^2})$) and $g_1=p^{-1}\in J_b(\Q_p)$ if $n$ is odd. We have then \[\C^i=\begin{cases}g_1^{i}\C^0&n\,even,\,i\,even\\
  p^{\frac{-i+1}{2}}g_1\C^0&n\,even,\,i\,odd\\
g_1^{\frac{i}{2}}\C^0&n\,odd.\end{cases}\]
Note the element $p^{-1}\in J_b(\Q_p)$ induces an isomorphism
$p^{-1}:\M^0\stackrel{\sim}{\ra}\M^2$. We denote by $g_2:=p^{-1}g_1^{-1}$ if $n$ is even, i.e. the above $g_1$ is such that
$p^{-1}=\begin{cases}g_1g_2& n\,even\\
g_1&n\, odd.\end{cases}$
Set $\C'=\begin{cases}\C^0\coprod g_1\C^0& n\,even\\
\C^0&n\,odd,\end{cases}$ then we have
\[\C=\coprod_{i\in\Z}p^{-i}\C',\,\textrm{and}\, \pi(\C)=\pi(\C'),\] where $\pi:\M\ra\Fm^a\subset\mathbf{P}^{n-1,an}$ is the $p$-adic period mapping over $L$. Thus the Hecke orbits of the two analytic domains are the same:
\[\bigcup_{T\in G(\Z_p)\setminus G(\Q_p)/G(\Z_p)}T.\C= \bigcup_{T\in G(\Z_p)\setminus G(\Q_p)/G(\Z_p)}T.\C'.\]

Using the geometric description of the reduced special fiber $\M_{red}$ in last section, we have the following covering of $\M^0$ by open subsets:
\[\M^0=\bigcup_{\Lam\in\mathcal{L}_0, t(\Lam)=t_{max}}sp^{-1}(\M_\Lam),\]where $sp:\M^0\ra\M^0_{red}$ is the specialization map. We also have the Bruhat-Tits stratification of $\M^0$ by locally closed analytic subspaces\[\M^0=\coprod_{\Lam\in\mathcal{L}_0}sp^{-1}(\M_{\Lam}^0).\]By definition, a point $x\in sp^{-1}(\M_\Lam)$ if and only if the composition \[(H_{\Lam^-})_{O_K/pO_K}\stackrel{\rho_{\Lam^-}}{\longrightarrow} \mathbf{H}_{O_K/pO_K}\stackrel{\rho}{\longrightarrow} H_{xO_K/pO_K}\] is an isogeny, and $x\in sp^{-1}(\M_\Lam^0)$ if and only if the above composition is an isogeny and it does not factor through $H_{\Lam'}$ for any $\Lam'\subsetneq\Lam$. Recall that there is a natural isogeny $\iota_{\Lam,\Lam'}: H_{\Lam^-}\ra H_{\Lam^{'-}}$ corresponding the inclusion $\Lam'\subset\Lam$, and we have the compatibility $\rho_{\Lam^-}=\rho_{\Lam^{'-}}\circ\iota_{\Lam,\Lam'}$.

We fix a choice $\Lam=\Lam_0\in\mathcal{L}_0$ with $t(\Lam)=t_{max}$.

\begin{definition}
We define an analytic domain in $\M^0$ \[\D:=\C\cap sp^{-1}(\M_{\Lam})=\C^0\cap sp^{-1}(\M_{\Lam}),\] which is locally closed.
\end{definition}

\section{Some unitary group Shimura varieties and the relatively compactness of $\D$}
In the following  we prove that the underlying topological space ${|\D|}$ of $\D$ is relatively compact, that is the topological closure $\ov{|\D|}$ in $|\M^0|$ (or $|\C^0|$) is compact. For this, we will use some unitary group Shimura varieties and the theory of $p$-adic uniformization. The PEL data $(B,\ast,\mathbb{V},\lan,\ran,h,O_B,\Lam)$ for defining these Shimura varieties are as following.
\begin{itemize}
\item $B$ is a simple $\Q$-algebra such that $B\otimes_{\Q}\R\simeq M_m(\mathbb{C})$ and $B\otimes_\Q\Q_p\simeq M_m(\Q_{p^2})$, for some integer $m$.
\item $\ast$ is a positive involution on $B$.
\item $\mathbb{V}$ is a non-zero finitely generated left $B$-module such that $n=dim_{\Q}(\mathbb{V})/2m$.
\item $\lan,\ran:\mathbb{V}\times\mathbb{V}\ra\Q$ is a non-degenerate skew-hermitian $\Q$-valued form.
Let $G:=GU_B(\mathbb{V},\lan,\ran)$ denote the reductive algebraic group over $\Q$ of $B$-linear symplectic similitudes of $(\mathbb{V},\lan,\ran)$.
\item $h: Res_{\mathbb{C}|\R}\mathbb{G}_m\ra G_\R$ is a homomorphism of real algebraic groups such that it defines a Hodge structure of type $\{(-1,0),(0,-1)\}$ on $\mathbb{V}$ and $\lan\cdot,h(\sqrt{-1})\cdot\ran: \mathbb{V}_\R\times \mathbb{V}_\R\ra\R$ is symmetric and positive definite.
\item $O_B$ is a $\ast$-invariant $\Z_{(p)}$-order of $B$ such that $O_B\otimes\Z_p$ is a maximal order of $B_{\Q_p}$. We can and we do fix an isomorphism $B_{\Q_p}\simeq M_m(\Q_{p^2})$ such that $O_B\otimes\Z_p$ is identified with $M_m(\Z_{p^2})$.
\item $\Lam$ is an $O_B$-invariant $\Z_p$-lattice of $\mathbb{V}_{\Q_p}$ such that the alternating form on $\Lam$ induced by $\lan,\ran$ is a perfect $\Z_p$-form.
\end{itemize}
The first condition implies that the center $\mathbb{K}$ of $B$ is a quadratic imaginary extension of $\Q$ and $p$ is inert in $\mathbb{K}$. The derived subgroup $G^{der}$ is an inner form of the quasi-split special unitary group $SU(n)$ for the extension $\mathbb{K}|\Q$. The assumption $B\otimes_\Q\R\simeq M_m(\mathbb{C})$ implies $G_\R$ is isomorphic to the group of unitary similitudes $GU(r,s)$ of a hermitian form of signature for some nonnegative integers $r$ and $s$ such that $r+s=n$. We will assume $r=1,s=n-1$. The reflex field $E$ will be $\mathbb{K}$ if $1\neq n-1$ i.e. $n\neq 2$ and $\Q$ if $n=2$. Up to Morita equivalence, the localization of the  above PEL data at $p$ then induces the local PEL data for defining the Rapoport-Zink space $\widehat{\M}$.

For a sufficient small open compact subgroup $K^p\subset G(\A^p_f)$, the associated Shimura variety $Sh_{K^p}$ over the integer ring $O_{E_p}$ of the local field $E_p$ ($p$ is inert in $E$) is the moduli space of abelian varieties with additional structures in the following sense. For any $O_{E_p}$-scheme $S$, $Sh_{K^p}(S)=\{(A,\iota,\lambda,\eta)\}/\simeq$ where
\begin{itemize}
\item $A$ is an abelian scheme over $S$ of relative dimension equal to dim$_{\mathbb{K}}\mathbb{V}$.
\item $\iota: O_B\otimes\Z_p\ra End(A)\otimes \Z_{(p)}$ is a nonzero homomorphism of $\Z_{(p)}$-algebras, such that the induced action of $O_B$ on the Lie algebra $Lie(A)$ satisfies $rank_{O_S}Lie(A)_1=m, rank_{O_S}Lie(A)_2=(n-1)m$, where $Lie(A)_1$ (resp. $Lie(A)_2$) is the subsheaf of $Lie(A)$ that $O_{\mathbb{K}_p}$ acts via the the natural inclusion $O_{\mathbb{K}_p}\subset O_B\otimes\Z_p$ (the composition of the nontrivial automorphism $\ast$ and the natural inclusion).
\item $\lambda: A\ra A^D$ is a principal $O_B\otimes\Z_p$-linear polarization, such that the involution $\ast$ on $B$ is compatible with the Rosati involution on $End(A)_\Q$ induced by $\lambda$, under the homomorphism $B\ra End(A)\otimes\Q$.
\item $\eta: \mathbb{V}\otimes\A^p_f\stackrel{\sim}{\ra}H_1(A,\A^p_f)$ mod $K^p$ is a $K^p$-level structure.
\item $(A_1,\iota_1,\lambda_1,\eta_1)\simeq (A_2,\iota_2,\lambda_2,\eta_2)$, if there exists an $O_B$-linear isogeny $\phi: A_1\ra A_2$ of degree prime to $p$ such that $\phi^\ast(\lambda_2)=a\lambda_1, \phi\circ\eta_1=\eta_2$ for some $a\in\Q^\times$.
\end{itemize}
Note $n$ is divisible by $m$ and in fact the rank of $\mathbb{V}$ as a $B$-module is $\frac{n}{m}$. In particular if $n=m$, the Shimura varieties $Sh_{K^p}$ for $K^p$ varies are all proper over $O_{E_p}$. In fact using Morita equivalence one sees in this case these are the simple Shimura varieties studied by Harris-Taylor in \cite{HT}.

Now let $S\in$Nilp$\Z_{p^2}$. To any $S$-valued point $(A,\iota,\lambda,\eta)\in Sh_{K^p}(S)$, we attach to it a unitary $p$-divisible group of signature $(1,n-1)$ as follows. Let $H'=A[p^\infty]$ be the $p$-divisible group of $A$. Then $O_B\otimes\Z_p=M_m(\Z_{p^2})$ acts on $H'$. By Morita equivalence, the functors $H'\mapsto \Z_{p^2}^m\otimes_{M_m(\Z_{p^2})}H'$ and $H\mapsto \Z_{p^2}^m\otimes_{\Z_{p^2}}H$ are mutually quasi-inverse between the category of $p$-divisible groups $H'$ over $S$ with a left $M_m(\Z_{p^2})$-action and the category of $p$-divisible groups with a $\Z_{p^2}$-action over $S$. We set $H:=\Z_{p^2}^m\otimes_{M_m(\Z_{p^2})}A[p^\infty]$, and denote its $\Z_{p^2}$-action still by $\iota$. The principal polarization $\lambda$ on $A$ then induces a $\Z_{p^2}$-linear principal polarization on $H$ which we still denote by $\lambda$.

To sate the link with the Rapoport-Zink space $\widehat{\M}$, we take the base change to $SpecO_L$ ($L=W(\ov{\F}_p)_\Q, O_L=W(\ov{\F}_p$)) of the Shimura variety, which we still denote by $Sh_{K^p}$ by abuse of notation. The special fiber $\ov{Sh}_{K^p}$ of $Sh_{K^p}$ over $\ov{\F}_p$ then admits the Newton polygon stratification.
\[\ov{Sh}_{K^p}=\coprod_{b\in B(G,\mu)}\ov{Sh}_{K^p}^b.\]
The Kottwitz set $B(G,\mu)$ of all Newton polygons can be written down explicitly as in \cite{BW} 3.1. In particular, one finds that every non-basic polygon has contacted points with the $\mu$-ordinary polygon. This key special phenomenon will at last lead to the relatively compactness of our analytic domain $\D$ in the $p$-adic analytic Rapoport-Zink space $\M$.
As shown in \cite{BW}, each strata $\ov{Sh}_{K^p}^b$ is non-empty, and any non-basic strata is in fact a leaf in the sense of \cite{M1}. The basic strata, which we denote by $\ov{Sh}^{b_0}_{K^p}$, was studied in \cite{V} and \cite{VW},  by studying the reduced special fiber $\M_{red}$ of the Rapoport-Zink space $\widehat{\M}$ and the uniformization of $\ov{Sh}^{b_0}_{K^p}$ by $\M_{red}$.

Let $\widehat{Sh}^{b_0}_{K^p}$ be the completion of $Sh_{K^p}$ along $\ov{Sh}^{b_0}_{K^p}$, then the chapter 6 of \cite{RZ} tell us there is an isomorphism of formal schemes over $SpfO_L$:
\[\widehat{Sh}^{b_0}_{K^p}\simeq I(\Q)\setminus\widehat{\M}\times G(\A^p_f)/K^p\simeq \coprod_{i\in I(\Q)\setminus G(\A^p_f)/K^p}\widehat{\M}/\Gamma_i.\]
Note the group $G$ satisfies Hasse principal: $ker^1(\Q,G)=1$ (cf. \cite{VW}). Here $I$ is an inner form of $G$ over $\Q$, which is anisotropic modulo center and such that $I_{\Q_p}=J_b, I_{\A^p_f}=G_{\A^p_f}$. Note the index set $I(\Q)\setminus G(\A^p_f)/K^p$ is finite, and if $g_1,\dots,g_k\in G(\A^p_f)$ is a set of representatives, then $\Gamma_i=I(\Q)\cap g_iK^pg_i^{-1}$ for $i=1,\dots,k$. The subgroups $\Gamma_i\subset J_b(\Q_p)$ are discrete and cocompact modulo center. Since $K^p$ is sufficiently small, $\Gamma_i$ is torsion free for all $i$.

We can describe the isomorphism \[I(\Q)\setminus\widehat{\M}\times G(\A^p_f)/K^p\ra\widehat{Sh}^{b_0}_{K^p}\] as follows. Let $(\mathbf{A},\iota,\lambda,\eta)$ be a $\ov{\F}_p$-valued point in $\ov{Sh}_{K^p}^{b_0}$, and the isomorphism we construct will depend on such a choice. Let $(\mathbf{H},\iota,\lambda)$ be the unitary $p$-divisible group associated to this abelian variety as above. We take $(\mathbf{H},\iota,\lambda)$ as the standard unitary $p$-divisible group for defining $\widehat{\M}$. For any $S\in$Nilp$O_L$, there is a map \[\widehat{\M}(S)\times G(\A^p_f)/K^p\ra Sh_{K^p}(S)\]
\[((H,\iota,\lambda,\rho),gK^p)\mapsto (A,\iota',\lambda',\eta g^{-1}K^p),\]such that there is an unique quasi-isogeny
$\mathbf{A}_{\ov{S}}\ra A_{\ov{S}}$ compatible with additional structures inducing $\rho$ when taking $p$-divisible groups. This map factors though the action of $I(\Q)$ and functorially for $S$. Let $\widehat{Sh}_{K^p}$ be the formal completion of $Sh_{K^p}$ along its special fiber $\ov{Sh}_{K^p}$. Then the above map induces a closed immersion of formal algebraic spaces
\[I(\Q)\setminus\widehat{\M}\times G(\A^p_f)/K^p\ra\widehat{Sh}_{K^p},\]
and one can prove that the left hand side is in fact a formal scheme and the image is $\widehat{Sh}^{b_0}_{K^p}$. Thus there is an isomorphism of formal schemes.

 We denote $\widehat{Sh}^{an,b}_{K^p}:=(\widehat{Sh}^{b}_{K^p})^{an}$ the Berkovich analytic space associated to the formal scheme $\widehat{Sh}^{b}_{K^p}$, the completion of $Sh_{K^p}$ along the strata $\ov{Sh}^b_{K^p}$, which is also the tube $sp^{-1}(\ov{Sh}^b_{K^p})$ of $\ov{Sh}^b_{K^p}$ in $\widehat{Sh}^{an}_{K^p}\subset Sh^{an}_{K^p}$. Here $sp: \widehat{Sh}^{an}_{K^p}\ra\ov{Sh}_{K^p}$ is the specialization map, and $Sh^{an}_{K^p}$ is the $p$-adic analytification of the generic fiber $Sh_{K^p}\times L=Sh_{G(\Z_p)\times K^p}$ over $L$, the last $\subset$ is a closed immersion. It is an isomorphism if and only if $Sh_{K^p}$ is proper over $SpecO_L$. We have the Newton polygon stratification of $\widehat{Sh}^{an}_{K^p}$ by locally closed subspaces:
 \[\widehat{Sh}^{an}_{K^p}=\coprod_{b\in B(G,\mu)}\widehat{Sh}^{an,b}_{K^p}.\]

  Analogously to the case of Rapoport-Zink spaces, there is a tower of analytic spaces $(\widehat{Sh}_{K_p\times K^p}^{an})_{K_p\subset G(\Z_p)}$ indexed by open compact subgroups $K_p\subset G(\Q_p)$, together with a family of closed immersions $(\widehat{Sh}_{K_p\times K^p}^{an})_{K_p}\subset (Sh^{an}_{K_p\times K^p})_{K_p}$, such that $\widehat{Sh}_{G(\Z_p)\times K^p}^{an}=\widehat{Sh}_{K^p}^{an}$. Then $G(\Q_p)$ acts on this tower and this gives the $p$-adic Hecke correspondence on each $\widehat{Sh}_{K_p\times K^p}^{an}$. The family of closed immersions $(\widehat{Sh}_{K_p\times K^p}^{an})_{K_p}\subset (Sh^{an}_{K_p\times K^p})_{K_p}$ is Hecke-equivariant, here the $G(\Q_p)$ action on the right hand side is the $p$-adic analytification of the Hecke action of $G(\Q_p)$ on $(Sh^{an}_{K_p\times K^p})_{K_p\subset G(\Z_p)}$. By taking the inverse images under the natural projection $\widehat{Sh}^{an}_{K_p\times K^p}\ra \widehat{Sh}^{an}_{K^p}$, we can define subspaces $\widehat{Sh}^{an,b}_{K_p\times K^p}\subset \widehat{Sh}^{an}_{K_p\times K^p}$, which are Hecke-invariant, thus we have $G(\Q_p)$-equivariant stratifications \[\widehat{Sh}^{an}_{K_p\times K^p}=\coprod_{b\in B(G,\mu)}\widehat{Sh}^{an,b}_{K_p\times K^p}.\]

 Now pass to the $p$-adic analytic side, we have a family of isomorphisms of analytic spaces
\[ I(\Q)\setminus\M_{K_p}\times G(\A^p_f)/K^p\simeq \coprod_{i\in I(\Q)\setminus G(\A^p_f)/K^p}\M_{K_p}/\Gamma_i\simeq \widehat{Sh}^{an,b_0}_{K_p\times K^p}.\]These isomorphisms are Hecke-equivariant for the action of $G(\Q_p)$ on the two sides. If we let $K^p$ varies, then they are Hecke-equivariant for the action of $G(\A_f)$ on the two sides.

We now look at the Harder-Narasimhan stratification of $\widehat{Sh}^{an}_{K_p\times K^p}$, see section 4.
For a unitary $p$-divisible group $(H,\iota,\lambda)$ over a complete valuation ring $O_K|\Z_p$ of rank one, recall we have the Harder-Narasimhan polygon \[HN(H,\iota,\lambda):=\frac{1}{2}HN(H)(2\cdot)=\lim_{k\ra\infty}\frac{1}{2k}HN(H[p^k])(2k\cdot)\] as a function $[0,n]\ra [0,n/2],$ which we also identify with its graph. For a point $x\in \widehat{Sh}^{an}_{K_p\times K^p}$, denote by $(A_x,\iota_x,\lambda_x,\eta_{px}\times\eta^p_x)$ the abelian scheme over $O_K:=O_{\mathcal{H}(x)}$ associated to $x$, let $(H_x=\Z_{p^2}^m\otimes_{M_m(\Z_{p^2})}A_x[p^\infty],\iota_x,\lambda_x)$ be as above the unitary $p$-divisible group obtained after Morita equivalence from $(A_x,\iota_x,\lambda_x,\eta_{px}\times\eta^p_x)$. Let $HN(x):=HN(H_x,\iota_x,\lambda_x)$. Thus we have defined a function \[HN: \widehat{Sh}^{an}_{K_p\times K^p}\ra \textrm{Poly},\]
here Poly denotes the set of concave polygons in $[0,n]\times[0,n/2]$ bounded by the $\mu$-ordinary Hodge polygon. We identify the set $B(G,\mu)$ with a finite subset of Poly by associating each $b\in B(G,\mu)$ its polygon. By proposition 4.5, this function $HN$ is semi-continuous.

\begin{definition}
For each $\P\in$ Poly, we define the subset\[\widehat{Sh}^{an,HN=\P}_{K_p\times K^p}:=HN^{-1}(\P)=\{x\in\widehat{Sh}^{an}_{K_p\times K^p}|\,HN(x)=\P\},\] which is then a locally closed subset.
\end{definition}
We thus obtain a Harder-Narasimhan stratification of the underlying topological space $|\widehat{Sh}^{an}_{K_p\times K^p}|$ by locally closed subset
\[|\widehat{Sh}^{an}_{K_p\times K^p}|=\coprod_{\P\in\textrm{Poly}}\widehat{Sh}^{an,HN=\P}_{K_p\times K^p}.\]Let $\P_{ss}$ be the basic element in Poly, then the strata $\widehat{Sh}^{an,HN=\P_{ss}}_{K_p\times K^p}$ is an open subset, thus there is an analytic structure on it so that the inclusion $\widehat{Sh}^{an,HN=\P_{ss}}_{K_p\times K^p}\subset \widehat{Sh}^{an}_{K_p\times K^p}$ is an open immersion. For general Harder-Narasimhan strata, there is in general no obvious analytic structure on it. But, fortunately, in our case we have the following strong conclusion.

\begin{proposition}
The Harder-Narasimhan stratification and the Newton polygon stratification for $\widehat{Sh}^{an}_{K_p\times K^p}$ coincide.
\end{proposition}
\begin{proof}
This comes from the fact that, for a unitary $p$-divisible group $(H,\iota,\lambda)$ over $O_K$, we have the inequalities
\[HN(H,\iota,\lambda)\leq Newt(H_{k},\iota,\lambda)\leq Hodge(H_k,\iota,\lambda),\]
and if there is a contact point $x$ of the Newton polygon $Newt(H_{k},\iota,\lambda)$ and the Hodge polygon $Hodge(H_k,\iota,\lambda)$, then the Harder-Narasimhan polygon $HN(H,\iota,\lambda)$ also passes at $x$, see \cite{Sh1}, corollary 5.3.
If one draw all the possible Newton polygons in our cases, then one finds immediately the proposition holds.
\end{proof}

The underlying topological space of $\widehat{Sh}^{an}_{K^p}$ is compact. We now consider the image $\mathcal{E}$ of the subspace $\C\subset \M$ under the $p$-adic uniformization morphism \[\coprod_{i=1}^k\M/\Gamma_i\simeq \widehat{Sh}^{an,b_0}_{K^p}=\widehat{Sh}^{an,HN=\P_{ss}}_{K^p}\subset\widehat{Sh}^{an}_{K^p}.\] Since $\C$ is $J_b(\Q_p)$-stable, we have $\mathcal{E}\simeq\coprod_{i=1}^k\C/\Gamma_i.$

\begin{proposition}
The subset $\mathcal{E}$ is a closed analytic domain in $\widehat{Sh}^{an}_{K^p}$, thus it is compact.
\end{proposition}
\begin{proof}
Let $\mathcal{H}/\widehat{Sh}_{K^p}$ be the $p$-divisible group associated to the universal Abelian scheme $\mathcal{A}$ after applying the Morita equivalence. Then we can describe the locus $\mathcal{E}\subset\widehat{Sh}^{an}_{K^p}$ by
\[\begin{split}\mathcal{E}=\{x\in\widehat{Sh}^{an}_{K^p}|&\,\exists\,\textrm{finite extension}\, K'|\mathcal{H}(x), \textrm{and a finite flat group}\; G\subset \mathcal{H}_x[p]\\
&\textrm{over}\; O_{K'}, \textrm{such that}\; \mathcal{H}_x/G \textrm{ is semi-stable over}\,O_{K'}.\}
\end{split}\]
For simplifying  notation, we denote $\mathcal{X}=\widehat{Sh}_{K^p}$ and $X=\widehat{Sh}^{an}_{K^p}$. By forgetting the polarization, we can construct a tower of analytic spaces $(X_K)_{K\subset GL_{n}(\Z_{p^2})}$ with $X=X_{GL_{n}(\Z_{p^2})}$, and an action of $GL_{n}(\Z_{p^2})$ (not $GL_{n}(\Q_{p^2})$!) on this tower. Let $\pi: Y=X_{Id+pM_{n}(\Z_{p^2})}\ra X$ be the natural finite \'{e}tale morphism. Then $Y$ classifies the level structures $\eta: (\Z/p\Z)^{2n}\stackrel{\sim}{\ra}\pi^\ast\mathcal{H}^{an}[p]$. After a possible admissible formal blow-up we can find a $p$-adic admissible formal model $\mathcal{Y}$ of $Y$, and a morphism $f: \mathcal{Y}\ra \mathcal{X}$ such that $f^{an}=\pi$. Consider the finite flat formal group scheme $f^\ast\mathcal{H}[p]$ over $\mathcal{Y}$. Then for any subgroup $M\subset (\Z/p\Z)^{2n}$, there exists a finite flat formal subgroup scheme $\mathcal{G}_M\subset f^\ast\mathcal{H}[p]$ such that $\mathcal{G}_M^{an}=\eta(M)$. Let $\mathcal{H}_M:=f^\ast\mathcal{H}/\mathcal{G}_M$ be the $p$-divisible group over $\mathcal{Y}$, then the following finite union
\[F:=\bigcup_M\{y\in Y|\mathcal{H}_{M,y}/O_{\mathcal{H}(y)} \textrm{is semi-stable}\}\]
is a closed analytic domain of $Y$ since each one on the right hand side is. Now the subset $\mathcal{E}\subset X$ is exactly the image $\pi(F)$ of $F$ under the finite \'{e}tale morphism $\pi:Y\ra X$, so it is a closed analytic domain.
\end{proof}

\begin{corollary}
The underlying topological space $|\D|$ of $\D$ is relatively compact.
\end{corollary}
\begin{proof}
Since $\mathcal{E}$ is a finite disjoint union of the form $\C/\Gamma$, therefore each $\C/\Gamma$ is compact by the above proposition. Since $\M_\Lam$ is an irreducible component of $\M_{red}$, and $\D=\C\cap sp^{-1}(\M_\Lam)\subset sp^{-1}(\M_\Lam)$, we may chose $K^p$ sufficiently small such that one  associated $\Gamma$ satisfies that $\forall id\neq \gamma\in\Gamma, \gamma\D\cap\D=\emptyset$. So we have a topological imbedding
\[\D\hookrightarrow\C/\Gamma,\]since the right hand side is compact, we can conclude.
\end{proof}

\section{Locally finite cell decompositions}

We will construct a locally finite covering of the unitary Rapoport-Zink space $\M$ from the locally closed analytic domain $\D$. Recall we have fixed a choice $\Lam\in\mathcal{L}_0$ with $t(\Lam)=t_{max}$. Let $Stab(\Lam)\subset J^{der}_b(\Q_p)$ be the stabilizer subgroup of $\Lam$ in $J^{der}_b(\Q_p)$. Then by our definition, $\D$ is stable under the action of $Stab(\Lam)$. Recall in section 8 we introduced the analytic domains $\C',\C^0$. We have the following covering of $\C^0$
\[\C^0=\bigcup_{g\in J^{der}_b(\Q_p)/Stab(\Lam)}g\D,\] which is locally finite, i.e. for any $g\in J^{der}_b(\Q_p)/Stab(\Lam)$, there are finite $g'\in J^{der}_b(\Q_p)/Stab(\Lam)$ such that $g\D\bigcap g'\D\neq \emptyset$, since
\[\M_{red}^0=\bigcup_{g\in J^{der}_b(\Q_p)/Stab(\Lam)}g\M_\Lam\] is a locally finite union of its irreducible components.

Recall in section 8 we have the following equality
\[
\bigcup_{T\in G(\Z_p)\setminus G(\Q_p)/G(\Z_p)}T.\C= \bigcup_{T\in G(\Z_p)\setminus G(\Q_p)/G(\Z_p)}T.\C'\]
If $n$ is odd, this equals to
\[=\bigcup_{T\in G(\Z_p)\setminus G(\Q_p)/G(\Z_p)}T.\C^0=\bigcup_{\begin{subarray}{c}T\in G(\Z_p)\setminus G(\Q_p)/G(\Z_p)\\g\in J^{der}_b(\Q_p)/Stab(\Lam)\end{subarray}
}T.g\D;\]
if $n$ is even, the above equals to
\[=
\bigcup_{\begin{subarray}{c}T\in G(\Z_p)\setminus G(\Q_p)/G(\Z_p)\\j=0,1\end{subarray}
}T.g_1^j\C^0=\bigcup_{\begin{subarray}{c}T\in G(\Z_p)\setminus G(\Q_p)/G(\Z_p)\\j=0,1\\g\in J^{der}_b(\Q_p)/Stab(\Lam)\end{subarray}}T.g_1^jg\D.\]
We would like to prove the last unions in the above two cases are locally finite. For this, it suffices to prove the following union
\[\bigcup_{\begin{subarray}{c}T\in G^{der}(\Z_p)\setminus G^{der}(\Q_p)/G^{der}(\Z_p)\\
g\in J^{der}_b(\Q_p)/Stab(\Lam)\end{subarray}}T.g\D\] is locally finite. To this end, we just need to prove the following holds
\[\#\{(T,g)\in G^{der}(\Z_p)\setminus G^{der}(\Q_p)/G^{der}(\Z_p)\times J^{der}_b(\Q_p)/Stab(\Lam)|\,T.g\D\cap\D\neq\emptyset\}<\infty.\]
This comes from the following several propositions.

\begin{proposition}
The Bruhat-Tits stratification of the analytic space \[\M^0=\coprod_{\Lam\in\mathcal{L}_0}sp^{-1}(\M_\Lam^0)\] by locally closed spaces is invariant under the Hecke action of $G^{der}(\Q_p)$, i.e., for each tube $sp^{-1}(\M_\Lam^0)$, we have \[T.sp^{-1}(\M_\Lam^0)\subset sp^{-1}(\M_\Lam^0)\]for any $T\in G^{der}(\Z_p)\setminus G^{der}(\Q_p)/G^{der}(\Z_p)$.
\end{proposition}
\begin{proof}
We just check that $T.sp^{-1}(\M_\Lam)\subset sp^{-1}(\M_\Lam)$, for any $T\in G^{der}(\Z_p)\setminus G^{der}(\Q_p)/G^{der}(\Z_p), \Lam\in\mathcal{L}_0$. The case for $sp^{-1}(\M_\Lam^0)$ is similar. Assume that \[T=G^{der}(\Z_p)
                   \left(\begin{array}{ccc}
                     p^{a_1} & & \\
                     &\ddots& \\
                     & &p^{a_n}\\
                   \end{array}
                   \right)
                 G^{der}(\Z_p)\in G^{der}(\Z_p)\setminus G^{der}(\Q_p)/G^{der}(\Z_p),\]\[ a_1\geq\cdots\geq a_n, a_1+a_n=a_2+a_{n-1}=\cdots=0,\]then $a_1\geq 0, a_n-a_1=-2a_1\leq 0$, and we consider the Hecke correspondence \[p^{-a_1}T=G^{der}(\Z_p)
                   \left(\begin{array}{cccc}
                     1 & & & \\
                     &p^{a_2-a_1}& & \\
                     & &\ddots& \\
                     & & &p^{a_n-a_1}\\
                   \end{array}
                   \right)
                 G^{der}(\Z_p).\]
We just need to check that \[p^{-a_1}T.sp^{-1}(\M_\Lam)\subset p^{-a_1}sp^{-1}(\M_\Lam),\]here the $p^{-a_1}$ on the  right hand side is considered as an element of $J_b(\Q_p)$: it induces an isomorphism
\[p^{-a_1}:\M^{0}\ra\M^{2a_1},\] under which the image $p^{-a_1}sp^{-1}(\M_\Lam)$ of $sp^{-1}(\M_\Lam)$ is $sp^{-1}(\M_{p^{a_1}\Lam})$. Assume $x\in sp^{-1}(\M_\Lam)$, the associated unitary $p$-divisible group $(H,\iota,
\lambda,\rho:\mathbf{H}_{O_K/pO_K}\ra H_{O_K/pO_K})$ is such that the composition
\[(H_{\Lam^-})_{O_K/pO_K}\stackrel{\rho_{\Lam^-}}{\longrightarrow} \mathbf{H}_{O_K/pO_K}\stackrel{\rho}{\longrightarrow} H_{O_K/pO_K}\] is an isogeny.
Let $y\in p^{-a_1}T.x$ be a any point such that the associated unitary $p$-divisible group is $(H/E,\iota,\lambda,\mathbf{H}_{O_K/pO_K}\stackrel{\rho}{\ra}H_{O_K/pO_K}\stackrel{\pi}{\ra}(H/E)_{O_K/pO_K}),$ where $E\subset H$ is a finite flat subgroup scheme such that its geometric generic fiber
\[E_{\ov{\eta}}\simeq \Z_{p^2}/p^{a_1-a_2}\Z_{p^2}\oplus\cdots\oplus\Z_{p^2}/p^{a_{1}-a_n}\Z_{p^2}.\]
Since \[(p^{a_1}\Lam)^+=p^{a_1}\Lam^+\]\[(p^{a_1}\Lam)^-=p^{2a_1}(p^{a_1}\Lam^+)^\vee=p^{a_1}\Lam^-,\]and the quasi-isogeny $\rho_{p^{a_1}\Lam^-}: H_{p^{a_1}\Lam^-}\ra \mathbf{H}$ is given by the composition $\rho_{\Lam^-}\circ\phi:$
\[H_{p^{a_1}\Lam^-}\stackrel{\phi}{\ra} H_{\Lam^-}\stackrel{\rho_{\Lam^-}}{\longrightarrow}\mathbf{H},\]where the first is the isogeny induced by the natural inclusion $p^{a_1}\Lam^-\subset \Lam^-$, thus its composition $\pi\circ\rho\circ\rho_{\Lam^-}\circ\phi$  with
\[\pi\circ\rho: \mathbf{H}_{O_K/pO_K}\stackrel{\rho}{\ra} H_{O_K/pO_K}\stackrel{\pi}{\ra} (H/E)_{O_K/pO_K}\] is an isogeny. That is $y\in sp^{-1}(\M_{p^{a_1}\Lam})\subset\M^{2a_1}$. So we have \[p^{-a_1}T.sp^{-1}(\M_\Lam)\subset p^{-a_1}sp^{-1}(\M_\Lam)=
sp^{-1}(\M_{p^{a_1}\Lam}).\]
\end{proof}

Recall in the proof of proposition 6.2 we have the description
\[\D=(\bigcup_{\underline{a}}T_{\underline{a}}.\mathcal{N}^{ss})\bigcap sp^{-1}(\M_\Lam),\]for our fixed $\Lam\in\mathcal{L}_0$ with $t(\Lam)=t_{max}$ and the closed immersion $\M\subset \mathcal{N}$. Here as before $\mathcal{N}$ is the basic Rapoport-Zink for $Res_{\Q_{p^2}|\Q_p}GL_n$ obtained from $\M$ by forgetting the polarization. We prove the following proposition, then it will be clear that the locally finiteness holds for $\M$.
\begin{proposition}Let $\widetilde{J_b}$ be the associated inner form of $Res_{\Q_{p^2}|\Q_p}GL_n$ for $\mathcal{N}$,
 and $U\subset \mathcal{N}_{red}$ be an open compact subset such that $\widetilde{J_b}(\Q_p)U=\mathcal{N}_{red}$, cf. \cite{F1} 2.4. Let $Z=\Q_{p^2}^\times$ be the center of $GL_n(\Q_{p^2})$ and $\widetilde{J_b}(\Q_p)$. Set
\[\D':=\mathcal{N}^{ss}\bigcap sp^{-1}(U),\] then we have
\[\#\{[T]\in (GL_{n}(\Z_{p^2})\setminus GL_{n}(\Q_{p^2})/GL_{n}(\Z_{p^2}))/Z| [T].\D'/Z\cap \D'/Z\neq \emptyset\}<\infty.\]
\end{proposition}
\begin{proof}
Recall that the criterion of quasi-compactness of an open subset $U\subset \mathcal{N}_{red}$ (\cite{F1} crit\`{e}re de quasicompacit\'{e} 2.4.14) : $U$ is quasi-compact if and only if there exist a Diedonn\'{e} lattice $M\in \mathcal{N}_{red}(\ov{k})$ and an integer $N$, such that
\[U(\ov{k})\subset\{M'\in \mathcal{N}_{red}(\ov{k})|p^NM\subset M'\subset p^{-N}M\},\]
or equivalently, there exists an integer $N$, such that the universal quasi-isogeny $\rho^{univ}$ satisfies that $p^N\rho^{univ}$ and $p^N(\rho^{univ})^{-1}$ are isogenies. The formal Rapoport-Zink space $\widehat{\mathcal{N}}$ decomposes as disjoint union according to the height of the universal quasi-isogeny:
\[\widehat{\mathcal{N}}=\coprod_{i\in \Z}\widehat{\mathcal{N}}^i,\]where the height of the universal quasi-isogeny is $2i$ over $\widehat{\mathcal{N}}^i$. Let $\widetilde{\widehat{\mathcal{N}}}:=\coprod_{i=0}^{n-1}\widehat{\mathcal{N}}^i$, which is in bijection with the quotient $\widehat{\mathcal{N}}/Z$. Let us denote by $\widetilde{\mathcal{N}}$ and $\widetilde{\mathcal{N}}_{red}$ the analytic generic fiber and reduced special fiber respectively of $\widetilde{\widehat{\mathcal{N}}}$. Then there is a metric function \[d: \widetilde{\mathcal{N}}_{red}(\ov{k})\times\widetilde{\mathcal{N}}_{red}(\ov{k}) \ra \mathbb{N}\]defined as \[d((H_1,\iota_1,\rho_1),(H_2,\iota_2,\rho_2))=q(\rho_1^{-1}\circ\rho_2)+q(\rho_2^{-1}\circ\rho_1),\]here $q(\rho)=htp^{n(\rho)}\rho$ and $n(\rho)$ is the smallest integer such that $p^{n(\rho)}\rho$ is an isogeny for a quasi-isogeny $\rho$. Then an open subset $U\subset\widetilde{\mathcal{N}}_{red}$ is quasi-compact if and only if there exist an integer $N$, such that $d(x,y)\leq N$ for all points $x,y\in U$.

To prove the proposition we may assume $U\subset\widetilde{\mathcal{N}}_{red}$. Let $(H_1,\iota_1,\rho_1),(H_2,\iota_2,\rho_2)$ be the $p$-divisible groups associated to two points $x_1,x_2\in U$. Let $M_1=\rho_{1\ast}^{-1}(\mathbb{D}(H_{1\ov{k}})),M_2=\rho_{2\ast}^{-1}(\mathbb{D}(H_{2\ov{k}}))$, and $inv(M_1,M_2)=(a_1,\dots,a_n)\in\Z^n_+$ be their relative invariant. Then one check that easily
\[d((H_1,\iota_1,\rho_1),(H_2,\iota_2,\rho_2))=n(a_1-a_n).\]So it is bounded by some fixed integer $N$ dependent only by $U$.

Now $\D'\subset \widetilde{\mathcal{N}}\subset\mathcal{N}$, and the Hecke correspondences on $\mathcal{N}$ induce an action of the set $ (GL_{n}(\Z_{p^2})\setminus GL_{n}(\Q_{p^2})/GL_{n}(\Z_{p^2}))/Z$ on $\widetilde{\mathcal{N}}$. Let $x\in\D',y\in\D'\cap [T].x$ for some $[T]\in (GL_{n}(\Z_{p^2})\setminus GL_{n}(\Q_{p^2})/GL_{n}(\Z_{p^2}))/Z, T.\D'\cap \D'\neq \emptyset$. Since by Cartan decomposition we have the bijection \[GL_{n}(\Z_{p^2})\setminus GL_{n}(\Q_{p^2})/GL_{n}(\Z_{p^2})\st{\sim}{\lra}\Z^n_+.\]If $T=GL_n(\Z_{p^2})\left( \begin{array}{ccc}
                     p^{a_1} & & \\
                     &\ddots& \\
                     & &p^{a_n}\\
                   \end{array}\right)GL_n(\Z_{p^2})$ , then for $z\in Z$
                   \[zT=GL_n(\Z_{p^2})\left( \begin{array}{ccc}
                     p^{a_1+v_p(z)} & & \\
                     &\ddots& \\
                     & &p^{a_n+v_p(z)}\\
                   \end{array}\right)GL_n(\Z_{p^2}).\] Thus $[T]$ corresponds to a class $[(a_1,\dots,a_n)]\in\Z^n_+/\Z$ where the action of $\Z$ on $\Z^n_+$ is the natural translation.
Let $(H/O_K,\iota,\rho)$ be the $p$-divisible group associated to $x$ over $O_{K=\mathcal{H}(x)}$. Then the $p$-divisible group associated to $y$ over $O_K$ is $(H/G,\iota',p^{-a_1+c}\pi\circ\rho)$, here $c\in\Z$ is some integer, $G\subset H$ is a finite flat group scheme such that $G_{\ov{K}}\simeq \Z/p^{a_1-a_2}\Z\oplus\cdots\oplus\Z/p^{a_{1}-a_n}\Z$. Moreover, since both $H$ and $H/G$ are semi-stable, $G$ is therefore semi-stable, \cite{F3} lemme 11. By the following lemma, the covariant Dieudonn\'{e} module \[\mathbb{D}(G_{\ov{k}})\simeq W/p^{a_1-a_2}W\oplus\cdots\oplus W/p^{a_{1}-a_n}W.\] The relative invariant of $\mathbb{D}(H_{\ov{k}})$ and $\mathbb{D}((H/G)_{\ov{k}})$ is then just $(0,a_2-a_1,\dots,a_{n}-a_1)$, and by the above $n(a_1-a_n)\leq N$. This plus the fact that $a_1\geq\cdots\geq a_n, a_i\in\Z,i=1,\dots,n$ with some easy combination argument imply that there are finite possibilities of the class $[(a_1,\dots,a_n)]\in\Z^n_+/\Z$, thus finite possibilities of $[T]\in (GL_{n}(\Z_{p^2})\setminus GL_{n}(\Q_{p^2})/GL_{n}(\Z_{p^2}))/Z$. This finishes the proof.

\end{proof}

\begin{lemma}
Let $G$ be a semi-stable finite flat group scheme over $O_K$. Suppose that $G_{\ov{K}}\simeq \Z/p^{a_1}\Z\oplus\cdots\oplus\Z/p^{a_n}\Z$, then we have \[\mathbb{D}(G_{\ov{k}})=W/p^{a_1}W\oplus\cdots\oplus W/p^{a_n}W,\] here $\mathbb{D}(G_{\ov{k}})$ is the covariant Dieudonn\'{e} module of $G_{\ov{k}}$.
\end{lemma}
\begin{proof}
If $pG=0$ then the above is evident. We assume that $p^2G=0$ here, the general case follows by induction. Under this assumption, there exists an $1\leq m\leq n$ such that $a_1=\cdots=a_k=2, a_{k+1}=\cdots=a_n=1$, and $G_{\ov{K}}\simeq(\Z/p^2\Z)^m\oplus (\Z/p\Z)^{n-m}$. Since $G$ is semi-stable, the following sequence
\[0\longrightarrow G[p]\lra G\st{p}{\lra}pG\lra 0\]is exact. So we get an exact sequence
\[0\lra G[p]_{\ov{k}}\lra G_{\ov{k}}\lra (pG)_{\ov{k}}\lra 0,\]and passing to (covariant) Dieudonn\'{e} module we get an exact sequence
\[0\lra \mathbb{D}(G[p]_{\ov{k}})\lra \mathbb{D}(G_{\ov{k}})\lra \mathbb{D}((pG)_{\ov{k}})\lra 0.\]
We have $htG[p]=n, ht(pG)=m, htG=n+m$. Assume that $\mathbb{D}(G_{\ov{k}})\simeq (W/p^2W)^{m'}\oplus(W/pW)^{n-m'}$ for some $1\leq m' \leq n$. Then we have $dim_{\F_p}\mathbb{D}(G[p]_{\ov{k}})=n,dim_{\F_p}\mathbb{D}((pG)_{\ov{k}})=m'$. Since
\[htG[p]=dim_{\F_p}\mathbb{D}(G[p]_{\ov{k}}),ht(pG)=dim_{\F_p}\mathbb{D}((pG)_{\ov{k}}),\]
i.e. $m=m'$, the lemma follows.
\end{proof}

\begin{remark}
The above proposition and its proof hold generally for all EL Rapoport-Zink spaces.
\end{remark}

\begin{proposition}
The union \[\bigcup_{T\in G^{der}(\Z_p)\setminus G^{der}(\Q_p)/G^{der}(\Z_p)}T.\D\] is locally finite.
\end{proposition}
\
\begin{proof}
Since $\C=(\bigcup_{\underline{a}}T_{\underline{a}}.\mathcal{N}^{ss})\bigcap\M$ (where $\underline{a}=(a_1,\dots,a_n)\in\Z_+^n$ such that $a_i\in\{0,-1\}$ for all $1\leq i\leq n$) and $\D=\C\bigcap sp^{-1}(\M_\Lam)$, we can choose some open compact subset $U\supset \M_\Lam$ in $\N_{red}$ such that $\widetilde{J_b}(\Q_p)U=\N_{red}$.
Then $\D\subset \bigcup_{\underline{a}}T_{\underline{a}}.\D'$ for $\D':=\N^{ss}\bigcap sp^{-1}(U)$. Denote $\D''=\bigcup_{\underline{a}}T_{\underline{a}}.\D'$, which is a finite union of closed analytic domains. By the above proposition we know that there are only finite $T\in G^{der}(\Z_p)\setminus G^{der}(\Q_p)/G^{der}(\Z_p)$ such that $T.\D'\bigcap \D'\neq \emptyset$. Therefore, there are also only finite $T\in G^{der}(\Z_p)\setminus G^{der}(\Q_p)/G^{der}(\Z_p)$ such that $T.\D''\bigcap \D''\neq \emptyset$. This implies the number of $T\in G^{der}(\Z_p)\setminus G^{der}(\Q_p)/G^{der}(\Z_p)$ such that $T.\D\bigcap \D\neq \emptyset$ is finite.
\end{proof}

\begin{corollary}
The unions \[\bigcup_{\begin{subarray}{c}T\in G(\Z_p)\setminus G(\Q_p)/G(\Z_p)\\g\in J^{der}_b(\Q_p)/Stab(\Lam)\end{subarray}
}T.g\D\] for $n$ odd and \[\bigcup_{\begin{subarray}{c}T\in G(\Z_p)\setminus G(\Q_p)/G(\Z_p)\\j=0,1\\g\in J^{der}_b(\Q_p)/Stab(\Lam)\end{subarray}}T.g_1^jg\D\]for $n$ even are both locally finite.
\end{corollary}
\begin{remark}
By the proofs of the above two propositions, we see that if \[T=G(\Z_p)\left( \begin{array}{ccc}
                     p^{a_1} & & \\
                     &\ddots& \\
                     & &p^{a_n}\\
                   \end{array}\right)
                G(\Z_p)\] corresponds to the point \[(a_1,\cdots,a_n)\in X_\ast(A)_+\subset \Z^n_+\] by the Cartan decomposition (see section 2), then the set \[\{T'\in G(\Z_p)\setminus G(\Q_p)/G(\Z_p)|T.\D\bigcap T.'\D\neq \emptyset\}\]corresponds to the set of points in some neighborhood of $(a_1,\cdots,a_n)\in X_\ast(A)_+\subset \Z^n_+$ (for the natural topology).
\end{remark}

With the notations above, we now can state the main theorem of this paper. The proof is based on some gluing arguments, and the following basic observation: let $Y\subset X$ be two Hausdorff paracompact strictly analytic spaces over a complete non-archimedean field $k$, such that the inclusion of $Y$ as a subspace of $X$ induces the identity of their associated rigid analytic spaces $Y^{rig}=X^{rig}$, then we have $Y=X$. Here we require that the analytic Grothendieck topologies are the same. Note if one just has the equality of the underlying sets $|Y^{rig}|=|X^{rig}|$, one can not deduce $Y=X$ and in fact there are many counter examples.

\begin{theorem}
We have a locally finite covering of the Berkovich analytic space $\M$:
\[\M=\bigcup_{\begin{subarray}{c}T\in G(\Z_p)\setminus G(\Q_p)/G(\Z_p)\\g\in J^{der}_b(\Q_p)/Stab(\Lam)\end{subarray}
}T.g\D\] if $n$ is odd, and
\[\M=\bigcup_{\begin{subarray}{c}T\in G(\Z_p)\setminus G(\Q_p)/G(\Z_p)\\j=0,1\\g\in J^{der}_b(\Q_p)/Stab(\Lam)\end{subarray}
}T.g_1^jg\D
\]if $n$ is even.
\end{theorem}
\begin{proof}
Take an open quasi-compact subset $U\subset \M^0_{red}$ such that $U\cap \M_\Lam\neq \emptyset$ and $J_b(\Q_p).U=\M_{red}$. Since $U$ is quasi-compact, it intersects with only finite number irreducible components $\M_{\Lam_i}$, $\Lam_i\in \mathcal{L}_0, t(\Lam_i)=t_{max}, i=1,\dots,k$ with $\Lam=\Lam_1$. Let $g_i\in J^{der}(\Q_p)/Stab(\Lam)$ be such that $g_i(\Lam)=\Lam_i$. Then we have the inclusion
\[\D':=\C\bigcap sp^{-1}(U)\subset \bigcup_{i=1}^kg_i\D.\] Note that $\D'$ is a closed analytic domain of $\M$. It is compact since $\D$ is locally compact.
So we have equalities of locally finite coverings
\[\bigcup_{\begin{subarray}{c}T\in G(\Z_p)\setminus G(\Q_p)/G(\Z_p)\\g\in J^{der}_b(\Q_p)/Stab(\Lam)\end{subarray}
}T.g\D=\bigcup_{\begin{subarray}{c}T\in G(\Z_p)\setminus G(\Q_p)/G(\Z_p)\\g\in J^{der}_b(\Q_p)/Stab(\Lam)\end{subarray}
}T.g\D'\] if $n$ is odd; and
\[\bigcup_{\begin{subarray}{c}T\in G(\Z_p)\setminus G(\Q_p)/G(\Z_p)\\j=0,1\\g\in J^{der}_b(\Q_p)/Stab(\Lam)\end{subarray}
}T.g_1^jg\D
=\bigcup_{\begin{subarray}{c}T\in G(\Z_p)\setminus G(\Q_p)/G(\Z_p)\\j=0,1\\g\in J^{der}_b(\Q_p)/Stab(\Lam)\end{subarray}
}T.g_1^jg\D'
\]if $n$ is even.
Since $\D'$ is closed, and the above analytic covering $(T.g\D')_{T,g}$ or $(T.g_1^jg\D')_{j,T,g}$ obtained by translations of $\D'$ is locally finite, by \cite{B}
we can glue them into a sub-analytic space $\M'\subset\M$, such that the underlying set of $\M'$ is given by the union as above. On the other hand, the rigid covering $(T.g\D^{'rig})_{T,g}$ or $(T.g_1^jg\D^{'rig})_{j,T,g}$  of $\M^{rig}$ by admissible open subsets can always be glued as a rigid space $\M'_0$, which is the associated rigid space of $\M'$: $\M^{'rig}=\M'_0$. Then these rigid coverings are admissible since the analytic coverings are locally finite, so from the equalities as sets \[\M^{rig}=\bigcup_{\begin{subarray}{c}T\in G(\Z_p)\setminus G(\Q_p)/G(\Z_p)\\g\in J^{der}_b(\Q_p)/Stab(\Lam)\end{subarray}}T.g\D^{'rig}\] if $n$ is odd; and
\[\M^{rig}=\bigcup_{\begin{subarray}{c}T\in G(\Z_p)\setminus G(\Q_p)/G(\Z_p)\\j=0,1\\g\in J^{der}_b(\Q_p)/Stab(\Lam)\end{subarray}}T.g_1^jg\D^{'rig}
\]if $n$ is even,
 we have the equality of $\M'_0=\M^{rig}$ as rigid spaces. By the equivalence of the category of Hausdorff paracompact strictly analytic Berkovich spaces and the category of quasi-separated quasi-paracompact rigid analytic spaces, we must have the equality $\M'=\M$.

\end{proof}
\begin{remark}
The above argument also works for theorem 27 in \cite{F3}.
\end{remark}
We have the following corollary when applying the theorem to the $p$-adic period domain $\Fm^a$.
\begin{corollary}
Let $\pi:\M\ra \Fm^a\subset \mathbf{P}^{n-1,an}$ be the $p$-adic period mapping, then we have a locally finite covering
\[\Fm^a=\bigcup_{g\in J^{der}_b(\Q_p)/Stab(\Lam)}g\pi(\D)\] if $n$ is odd, and
\[\Fm^a=\bigcup_{\begin{subarray}{c}j=0,1\\g\in J^{der}_b(\Q_p)/Stab(\Lam)\end{subarray}}g_1^jg\pi(\D)\]if $n$ is even.
\end{corollary}

We look at the cases with level structures. Let $K\subset G(\Z_p)$ be an open compact subgroup and $\pi_K: \M_K\ra\M$ be the natural projection, which is a $J_b(\Q_p)$-equivariant finite \'etale surjection and also compatible with the Hecke actions. Denote $\D_K=\pi^{-1}_K(\D)$, then $g\D_K=\pi^{-1}_K(g\D)$ for all $g\in J_b(\Q_p)$, and
\[Kh_1K.g\D_K=Kh_2K.g\D_K\] for $Kh_1K,Kh_2K\in K\setminus G(\Q_p)/K$ having the same image under the projection
\[K\setminus G(\Q_p)/K\ra G(\Z_p)\setminus G(\Q_p)/K.\]The last equality holds since any $h\in G(\Z_p)$ acts trivially on $\M$, therefore $Khh_1K.\pi^{-1}_K(g\D)=Kh_1K.\pi^{-1}_K(g\D)$.
\begin{corollary}
We have a locally finite covering of the analytic space $\M_K$
\[\M_K=\bigcup_{\begin{subarray}{c}T\in G(\Z_p)\setminus G(\Q_p)/K\\g\in J^{der}_b(\Q_p)/Stab(\Lam)\end{subarray}
}T.g\D_K\] if $n$ is odd, and
\[\M_K=\bigcup_{\begin{subarray}{c}T\in G(\Z_p)\setminus G(\Q_p)/K\\j=0,1\\g\in J^{der}_b(\Q_p)/Stab(\Lam)\end{subarray}
}T.g_1^jg\D_K
\]if $n$ is even.
\end{corollary}
We will consider some cohomological application of this corollary in the next section.

Finally we have a corollary for Shimura varieties.
\begin{corollary}
Let $Sh_{K^p}$ be as the Shimura variety introduced in section 9, $\widehat{Sh}^{an}_{K^p}$ be the generic analytic fiber of its $p$-adic completion $\widehat{Sh}_{K^p}$, and $\widehat{Sh}^{an,b_0}_{K^p}$ be the tube in $\widehat{Sh}^{an}_{K^p}$ over the basic strata $\ov{Sh}_{K^p}^{b_0}$, which is an open subspace. Let $\widehat{Sh}^{an}_{K_p\times K^p}\ra \widehat{Sh}^{an}_{K^p}$ be the covering in level $K_p\subset G(\Z_p)$ (an open compact subgroup), and $\widehat{Sh}^{an,b_0}_{K_p\times K^p}$ be the inverse image of $\widehat{Sh}^{an,b_0}_{K^p}$.  Denote $\C'_{K_p}$ the inverse image of $\C'$ in $\M_{K_p}$, $\mathcal{E}'_{K_p}$ the image of $\C'_{K_p}$ under the $p$-adic uniformization \[ I(\Q)\setminus\M_{K_p}\times G(\A^p_f)/K^p\simeq \coprod_{i\in I(\Q)\setminus G(\A^p_f)/K^p}\M_{K_p}/\Gamma_i\simeq \widehat{Sh}^{an,b_0}_{K_p\times K^p}.\]
\begin{enumerate}

\item Let $\Gamma=\Gamma_i$ be one of the above discrete, torsion free, cocompact modulo center subgroups of $J_b(\Q_p)$, and $\Gamma^{der}=\Gamma\cap J^{der}_b(\Q_p)$, $D_{K_p}=D_{iK_p}$ be the image of $\D_{K_p}$ under the morphism $\M_{K_p}\ra\M_{K_p}/\Gamma$, then we have a covering
    \[\M_{K_p}/\Gamma=\bigcup_{\begin{subarray}{c}T\in G(\Z_p)\setminus G(\Q_p)/K_p\\ g\in\Gamma^{der}\setminus J^{der}_b(\Q_p)/Stab(\Lam)\end{subarray}}T.gD_{K_p}\] if $n$ is odd, and
    \[\M_{K_p}/\Gamma=\bigcup_{\begin{subarray}{c}T\in G(\Z_p)\setminus G(\Q_p)/K_p\\j=0,1\\ g\in\Gamma^{der}\setminus J^{der}_b(\Q_p)/Stab(\Lam)\end{subarray}}T.g_1^jgD_{K_p}\] if $n$ is even.
\item Under the above notation, we have a covering \[\mathcal{E}'_{K_p}=\coprod_{i\in I(\Q)\setminus G(\A^p_f)/K^p}\bigcup_{g\in\Gamma^{der}\setminus J^{der}_b(\Q_p)/Stab(\Lam)}gD_{iK_p}\] if $n$ is odd, and \[\mathcal{E}'_{K_p}=\coprod_{i\in I(\Q)\setminus G(\A^p_f)/K^p}\bigcup_{\begin{subarray}{c}j=0,1\\g\in\Gamma^{der}\setminus J^{der}_b(\Q_p)/Stab(\Lam)\end{subarray}}g_1^jgD_{iK_p}\] if $n$ is even.  We have a covering
\[\widehat{Sh}^{an,b_0}_{K_p\times K^p}=\bigcup_{T\in G(\Z_p)\setminus G(\Q_p)/K_p}T.\mathcal{E}'_{K_p}.\]
\end{enumerate}

\end{corollary}

\section{Cohomological application of the locally finite cell decomposition: a Lefschetz trace formula}
In this section we study some cohomological applications of the locally finite cell decomposition of the tower $(\M_K)_{K\subset G(\Z_p)}$, using the similar ideas in \cite{Sh1}.

We first review some basic facts. Fix a prime $l\neq p$. Let $\ov{\Q}_l$ (resp. $\ov{\Q}_p$) be a fixed algebraic closure of $\Q_l$ (resp. $\Q_p$), and $\Cm_p$ be the completion of $\ov{\Q}_p$ for its valuation which extends that of $\Q_p$. For any open compact subgroup $K\subset G(\Z_p)$ and integer $j\geq 0$, the $j$-th cohomology with compact support of $\M_K\times\Cm_p$ with coefficient in $\ov{\Q}_l$ is
\[H^j_c(\M_K\times \Cm_p, \ov{\Q}_l)=\varinjlim_{U}\varprojlim_{n}H^j_c(U\times\Cm_p,\Z/l^n\Z)\otimes\ov{\Q}_l,\]
where the injective limit is taken over all relatively compact open subsets $U\subset\M_K$, see \cite{F1} section 4. Recall in section 2 we introduced the group $\triangle=Hom(X^\ast_{\Q_p}(G),\Z)\simeq \Z$ and there is a mapping
\[\varkappa:\widehat{\M}\ra \Z,\]with the image $\triangle'$ is $\Z$ if $n$ is even and $2\Z$ if $n$ if odd.
This mapping satisfies that
  \[\varkappa(gx)=\omega_J(g)+\varkappa(x)\]for all $g\in J_b(\Q_p),x\in\widehat{\M}$. Here $\omega_J:J_b(\Q_p)\ra\triangle$ is defined by $<\omega_J(x),\chi>=v_p(i(\chi)(x))$ where $i: X^\ast_{\Q_p}(G)\ra X^\ast_{\Q_p}(J_b)$ is the natural morphism between the two groups of $\Q_p$-rational characters. We denote
  \[J_b^1=\bigcap_{\chi\in X^\ast_{\Q_p}(G)}ker|i(\chi)|,\]which we consider as a subgroup of $J_b(\Q_p)$. Here
  \[\begin{split}|i(\chi)|:&J_b\longrightarrow \Z\\&x\mapsto v_p(i(\chi)(x)).\end{split}\]We have a decomposition
  \[\M_K=\coprod_{i\in\Z,\;in\; even}\M_K^i\] as the case of $\M$ by considering the height of quasi-isogenies, and in fact $\M_K^i=\pi_K^{-1}(\M^i)$ for the projection $\pi_K:\M_K\ra\M$.
  The group $J_b(\Q_p)$ acts transitively on $\triangle'$ and $\M_K^0$ is stable under the group $J_b^1$ for the action of $J_b(\Q_p)$. We have the equalities for cohomology groups
  \[H_c^j(\M_K\times\Cm_p,\ov{\Q}_l)=\bigoplus_{i\in\Z,\; in\; even}H_c^j(\M_K^i\times\Cm_p,\ov{\Q}_l)=\textrm{c-Ind}^{J_b(\Q_p)}_{J_b^1}H_c^j(\M_K^0\times\Cm_p,\ov{\Q}_l),\]
  for the last equality see lemme 4.4.10 of \cite{F1}. In fact there are also actions of $G(\Q_p)$ and $W_E$ on $\triangle'$, where $E$ is the reflex local field. Let $(G(\Q_p)\times J_b(\Q_p)\times W_E)^1$ be the subgroup of $
G(\Q_p)\times J_b(\Q_p)\times W_E$ which acts trivially on $\triangle'$. Let $K$ vary as open compact subgroup of $G(\Z_p)$, we have equalities of $
G(\Q_p)\times J_b(\Q_p)\times W_E$-representations
\[\begin{split}\varinjlim_{K}H_c^j(\M_K\times\Cm_p,\ov{\Q}_l)&=\bigoplus_{i\in\Z,\; in \; even}\varinjlim_{K}H_c^j(\M_K^i\times\Cm_p,\ov{\Q}_l)\\&=\textrm{c-Ind}^{G(\Q_p)\times J_b(\Q_p)\times W_E}_{(G(\Q_p)\times J_b(\Q_p)\times W_E)^1}\varinjlim_{K}H_c^j(\M_K^0\times\Cm_p,\ov{\Q}_l),\end{split}\]see remarque 4.4.11 of loc. cit. In the following we will forget the action of $W_E$ and just consider the cohomology groups as $G(\Q_p)\times J_b(\Q_p)$-representations.

 The dimension of $H_c^j(\M_K^0\times\Cm_p,\ov{\Q}_l)$ as $\ov{\Q}_l$-vector space is infinite. However, as $J_b^1$-representation, it is of finite type, see loc. cit. proposition 4.4.13. As in section 10, we fix a $\Lam\in\mathcal{L}_0$ such that $t(\Lam)=t_{max}$. Recall the subscheme $\M_\Lam\subset\M^0_{red}$, which is an irreducible component of $\M^0_{red}$, and the set of all irreducible components of $\M^0_{red}$ is exactly $\{g\M_\Lam=\M_{g\Lam}|g\in J^{der}_b(\Q_p)/Stab(\Lam)\}$. We have a locally finite covering by open subsets
\[\M^0=\bigcup_{g\in J^{der}_b(\Q_p)/Stab(\Lam)}gsp^{-1}(\M_\Lam).\]
We have $J^{der}_b(\Q_p)\subset J^1_b$, and the action of $J^{der}_b(\Q_p)$ on $\mathcal{L}_0$ naturally extends to an action of $J^1_b$, and we still denote by $Stab(\Lam)$ the stabilizer of $\Lam$ in $J^1_b$. We set $U:=\pi_K^{-1}(sp^{-1}(\M_\Lam))\subset\M_K^0$, then by theorem 3.3 (ii) of \cite{Hu1}
\[dim_{\ov{\Q}_l}H_c^j(U\times\Cm_p,\ov{\Q}_l)<\infty.\]
We have a $J^1_b$-equivariant spectral sequence
\[E_1^{p,q}=\bigoplus_{\begin{subarray}{c}\alpha\subset J^1_b/Stab(\Lam)\\ \#\alpha=-p+1\end{subarray}}H^q_c(U(\alpha)\times\Cm_p,\ov{\Q}_l)\Rightarrow H_c^{p+q}(\M_K^0\times\Cm_p,\ov{\Q}_l),\]
where $p\leq 0, 0\leq q\leq n-1, U(\alpha)=\bigcap_{g\in\alpha}gU$. The $J^1_b$ action on $E^{p,q}_1$ is
\[\forall h\in J^1_b,\; h: H^q_c(U(\alpha)\times\Cm_p,\ov{\Q}_l)\st{\sim}{\ra}H^q_c(hU(\alpha)\times\Cm_p,\ov{\Q}_l).\]
Denote $K_\alpha=\bigcap_{g\in\alpha}gStab(\Lam)g^{-1}$, then $H^q_c(U(\alpha)\times\Cm_p,\ov{\Q}_l)$ is a smooth $\ov{\Q}_l$-representation of $K_\alpha$, and $E^{p,q}_1$ can be rewritten as
\[E^{p,q}_1=\bigoplus_{[\alpha]\in J^1_b\setminus(J^1_b/K)^{-p+1}}\textrm{c-Ind}_{K_\alpha}^{J^1_b}H^q_c(U(\alpha)\times\Cm_p,\ov{\Q}_l). \]
Since $(gU)_{g\in J^1_b/Stab(\Lam)}$ is a locally finite covering of $\M_K$,
\[\#\{[\alpha]\in J^1_b\setminus(J^1_b/Stab(\Lam))^{-p+1}|U(\alpha)\neq\emptyset\}<\infty,\]
i.e. the above direct sum has just finitely many non zero terms.

Let $\gamma=(h,g)\in G(\Q_p)\times J_b(\Q_p)$ be a fixed element with both $h$ and $g$ regular elliptic. Then there is a fundamental system of neighborhoods of 1 in $G(\Q_p)$ consisting of open compact subgroups $K\subset G(\Z_p)$ which are normalized by $h$. From now on let $K\subset G(\Z_p)$ be a sufficiently small open compact subgroup such that $hKh^{-1}=K$. Consider the locally finite cell decomposition of $\M_K^0$
\[\M_K^0=\bigcup_{\begin{subarray}{c}T\in G(\Z_p)\setminus G(\Q_p)/K\\g'\in J^{der}_b(\Q_p)/Stab(\Lam)\end{subarray}
}((T.g'\D_K)\bigcap\M_K^0)\] if $n$ is odd, and
\[\M_K^0=\bigcup_{\begin{subarray}{c}T\in G(\Z_p)\setminus G(\Q_p)/K\\j=0,1\\g'\in J^{der}_b(\Q_p)/Stab(\Lam)\end{subarray}
}((T.g_1^jg'\D_K)\bigcap\M_K^0)
\]if $n$ is even. Here by replacing $\D$ by $\D'$ in the proof of theorem 10.8 we can assume $\D$ is compact.
Thus the notation $\D$ in this section is a little different from that in previous sections. If $n$ is odd, for any $g'\in J_b^{der}(\Q_p)/Stab(\Lam), T\in G(\Z_p)\setminus G(\Q_p)/K$,
\[(T.g'\D_K)\bigcap\M_K^0\neq \emptyset\Leftrightarrow v_p(detT)=0,\]in which case we have
\[T.g'\D_K\subset \M_K^0.\]If $n$ is even, for any $g'\in J_b^{der}(\Q_p)/Stab(\Lam), j=0,1, T\in G(\Z_p)\setminus G(\Q_p)/K$,
\[(T.g_1^jg'\D_K)\bigcap\M_K^0\neq\emptyset\Leftrightarrow -\frac{2}{n}v_p(detT)+j=0,\]
in which case we have
\[T.g_1^jg'\D_K\subset\M_K^0.\]Thus we can rewrite the locally finite cell decomposition of $\M_K^0$ as
\[\M_K^0=\bigcup_{\begin{subarray}{c}T\in G(\Z_p)\setminus G(\Q_p)/K\\v_p(detT)=0\\g'\in J^{der}_b(\Q_p)/Stab(\Lam)\end{subarray}
}T.g'\D_K\] if $n$ is odd, and
\[\M_K^0=\bigcup_{\begin{subarray}{c}T\in G(\Z_p)\setminus G(\Q_p)/K\\j=0,1,-\frac{2}{n}v_p(detT)+j=0\\g'\in J^{der}_b(\Q_p)/Stab(\Lam)\end{subarray}
}T.g_1^jg'\D_K
\]if $n$ is even.
We will write the cells $T.g'\D_K$ and $T.g_1^jg'\D_K$ as $\D_{T,g',K}$ and $\D_{T,j,g',K}$ respectively.
If $n$ is odd, for any $T\in G(\Z_p)\setminus G(\Q_p)/K, j\in 2\Z, g'\in J^{der}_b(\Q_p)$ such that $-\frac{2}{n}v_p(detT)+j=0$, we denote also
\[\emptyset\neq \D_{T,j,g',K}:=T.p^{\frac{-j}{2}}g'\D_K\subset\M_K^0.\] If $n$ is even, for any $T\in G(\Z_p)\setminus G(\Q_p)/K, j\in\Z,g'\in J^{der}_b(\Q_p)/Stab(\Lam)$ such that $-\frac{2}{n}v_p(detT)+j=0$, we denote also
\[\emptyset\neq\D_{T,j,g',K}:=\begin{cases}T.p^{\frac{-j}{2}}g'\D_K & (j\;even)\\T.p^{\frac{1-j}{2}}g_1g'\D_K &(j\;odd),
\end{cases}\]
which is a compact analytic domain in $\M_K^0$.
Since $(z,z^{-1})\in G(\Q_p)\times J_b(\Q_p)$ acts trivially on $\M_K$ for any $z\in \Q_p^\times$,
with these notations we have
\[\D_{T,g',K}=\D_{Tz,z^{-1}g',K},\;\D_{T,j,g',K}=\D_{Tz,j+2v_p(z),g',K},\;\forall z\in \Q_p^\times\]
For the $\gamma=(h,g)$ above we suppose further $v_p(deth)+v_p(detg)=0$. Then $\gamma(\M_K^0)=\M_K^0$. To describe the action of $\gamma$ on the cells, we have to introduce some more natural parameter set of cells.

Consider the product $G\times J_b$ as reductive group over $\Q_p$, then $\mathbb{G}_m$ acts on it through the embedding $z\mapsto (z,z^{-1})$. Let $\B(G\times J_b,\Q_p)$ be the (extended) Bruhat-Tits building of $G(\Q_p)\times J_b(\Q_p)$ over $\Q_p$, and $\B=\B(
G\times J_b,\Q_p)/\Q_p^\times$ be its quotient by $\Q_p^\times$ through the embedding above. More precisely, if $\B(G^{ad},\Q_p)$ (resp. $\B(J^{ad}_b,\Q_p)$) is the Bruhat-Tits building of the adjoint group $G^{ad}$ (resp. $J^{ad}_b$) over $\Q_p$, which is isomorphic to the Bruhat-Tits building of the derived group $G^{der}$ (resp. $J^{der}_b$) over $\Q_p$, then the (extended) Bruhat-Tits building of $G$ (resp. $J_b$) over $\Q_p$ is $\B(G^{ad},\Q_p)\times\R$ (resp. $\B(J^{ad}_b,\Q_p)$) over $\Q_p$. By definition the quotient building $\B$ is
 \[\B\simeq (\B(G^{ad},\Q_p)\times\R\times\R\times\B(J^{ad}_b,\Q_p))/\sim,\]
 where \[(x,s,t,y)\sim (x',s',t',y')\Leftrightarrow\]\[ x=x',y=y',s-s'=t'-t=r(\;n\;odd)\;\textrm{or}\;2r(\;n\;even), \textrm{for some}\;r\in\Z.\]Any point of $\B$ can be written in the form $[x,s,t,y]=[x,s',t',y]$ where $s'\in\R, t'\in [0,1) (n\; odd)$ $\textrm{or}\;[0,2) (n\; even)$ are uniquely determined. The action of $G(\Q_p)\times J_b(\Q_p)$ on $\B$ is given by
 \[\forall (h,g)\times G(\Q_p)\times J_b(\Q_p), (h,g)[x,s,t,y]=[h^{-1}x,s+\frac{2}{n}v_p(deth),t+\frac{2}{n}v_p(detg),gy].\]
 If we consider the right action of $G(\Q_p)$ on $\B(G^{ad},\Q_p)$ by $xh:=h^{-1}x$, we can write $(h,g)[x,s,t,y]=[xh,s+\frac{2}{n}v_p(deth),t+\frac{2}{n}v_p(detg),gy]$.
 The sets of vertices of $\B(G^{ad},\Q_p)$ and $\B(J^{ad}_b,\Q_p)$ can be described as in section 7, as certain sets of lattices in $\Q_{p^2}^n$.
By fixing a choice $(\Z_{p^2}^n,\Lam)$ with $t(\Lam)=t_{max}$, we can identify the following set
\[(G(\Z_p)\setminus G(\Q_p)\times J_b(\Q_p)/Stab_{J^{der}_b(\Q_p)}(\Lam))/\Q_p^\times\]
with a subset of the set of vertices $\B^0$ (the quotient by $\Q_p^\times$ of vertices in $\B(G\times J_b,\Q_p)$), such that the projections to $\B(G^{ad},\Q_p)^0$ and $\B(J^{ad}_b,\Q_p)^0$ are vertices of types determined by $\Z_{p^2}^n, \Lam$ respectively. In the following we will simply denote $Stab_{J^{der}_b(\Q_p)}(\Lam)$ by $Stab(\Lam)$ as in section 10. For any open compact subgroup $K\subset G(\Z_p)$, we can identify
\[\I_K:=(G(\Z_p)\setminus G(\Q_p)/K\times J_b(\Q_p)/Stab(\Lam))/\Q_p^\times\]with a subset of $\B^0/K\subset\B/K$. We can write an element of $\I_K$ in the form $[T,g']=[x,\frac{2}{n}v_p(detT),\frac{2}{n}v_p(detg'),y]$ for some $x\in \B(G^{ad},\Q_p)^0,y\in \B(J^{ad}_b,\Q_p)^0$ uniquely determined by $T$ and $g'$. There is a map
\[\begin{split} G(\Z_p)\setminus G(\Q_p)/K&\times J_b(\Q_p)/Stab(\Lam)\longrightarrow\Z\\
&(T,g')\mapsto -\frac{2}{n}(v_p(detT)+v_p(detg')),\end{split}\]with image $\Z$ if $n$ is even and $2\Z$ if $n$ is odd.
Let $(G(\Z_p)\setminus G(\Q_p)/K\times J_b(\Q_p)/Stab(\Lam))^0$ be the inverse image of $0$, then $\Q_p^\times$ acts on this subset. We denote
\[\I_K^0:=(G(\Z_p)\setminus G(\Q_p)/K\times J_b(\Q_p)/Stab(\Lam))^0/\Q_p^\times.\] In fact there is a well defined map\[\begin{split}\varphi:\B&\longrightarrow\R\\ [x,s,t,y]&\mapsto -s-t,\end{split}\]with each fiber stable for the action of $K$. Then we have $\I_K^0=\I_K\bigcap \varphi^{-1}(0)^0/K$.

 Now for any
$[T,g']\in\I_K$,
the subset
\[\D_{[T,g'],K}:=T.g'\D_K\subset\M_K\] is well defined, which is a compact analytic domain. If $[T,g']\in\I_K^0$,
\[\D_{[T,g'],K}\subset\M_K^0.\]We can rewrite the cell decomposition of $\M_K$ and $\M_K^0$ as
\[\begin{split}&\M_K=\bigcup_{[T,g']\in\I_K}\D_{[T,g'],K},\\
&\M_K^0=\bigcup_{[T,g']\in\I_K^0}\D_{[T,g'],K}.\end{split}\]
For $\gamma=(h,g)\in G(\Q_p)\times J_b(\Q_p)$ such that $hKh^{-1}=K$, it induces an action on $\I_K$ by $[T,g']\mapsto[Th,gg']$.
 On the other hand the automorphism $\gamma:\M_K\ra\M_K$ induces an action of $\gamma$ on the cells compatible with its action on the parameter set above:
 \[\gamma(\D_{[T,g'],K})=\D_{[Th,gg'],K}.\]
 If we assume as above $v_p(deth)+v_p(detg)=0$ then $\gamma$ acts on $\I_K^0$ by $[T,g']\mapsto [Th,gg']$, and the automorphism $\gamma:\M_K^0\ra\M_K^0$ induces a compatible action on the cells as above.

Recall there is a metric $d(\cdot,\cdot)$ on the building $\B$. If we denote the metrics on $\B(G^{ad},\Q_p)$ and $\B(J^{ad}_b,\Q_p)$ by $d_1,d_2$ respectively, for $[x,s,t,y],[x',s',t',y']\in \B$ with $t,t'\in [0,1) (n\;odd)$ $\textrm{or}\;[0,2) (n\;even)$, we have \[d([x,s,t,y],[x',s',t',y'])=\sqrt{d_1(x,x')^2+d_2(y,y')^2+(s-s')^2+(t-t')^2}.\] It induces a metric $\ov{d}$ on the quotient space $\B/K$:
\[\ov{d}(xK,yK)=\inf_{k,k'\in K}d(xk,yk')=\inf_{k\in K}d(xk,y)=\inf_{k\in K}d(x,yk).\]
Since $K\subset G(\Z_p)$ is compact, one checks it easily this is indeed a metric. For any fixed $\rho>0$ and $o\in \B^0/K$, the closed ball $B(o,\rho)$ of $\B/K$ contains only finitely many points of the discrete subsets $\B^0/K,\I_K,\I_K^0$.
For the $\gamma=(h,g)\in G(\Q_p)\times J_b(\Q_p)$ with $hKh^{-1}=K$, one checks easily by definition the action of $\gamma$ on the parameter set $\I_K$ is isometric:
\[\ov{d}(\gamma x,\gamma x)=\ov{d}(x,x), \;\forall x\in \I_K.\]
Note for any $[T_1,g_1],[T_2,g_2]\in \I_K$, $\D_{[T_1,g_1],K}\bigcap\D_{[T_2,g_2],K}\neq\emptyset$ implies that
$v_p(detT_1)+v_p(detg_1)=v_p(detT_2)+v_p(detg_2)$.  If we write \[[T_1,g_1]=[x_1K,s_1,t_1,y_1], [T_2,g_2]=[x_2K,s_2,t_2,y_2]\] with \[x_1,x_2\in \B(G^{ad},\Q_p), y_1,y_2\in \B(J^{ad}_b,\Q_p), s_1,s_2\in\Z\subset\R, t_1,t_2\in\{0\} (n\; odd)\; \textrm{or}\{0,1\} (n\; even),\] (i.e. $\exists r_1,r_2\in\Z, s.t.\; \frac{2}{n}v_p(detT_i)=s_i+nr_i,\frac{2}{n}v_p(detg_i)=t_i+nr_i, i=1,2$,) then $s_1+t_1=s_2+t_2, s_1-s_2=t_2-t_1\in [-1,1]$, the distance \[\ov{d}([T_1,g_1],[T_2,g_2])=\inf_{k\in K}\sqrt{d_1(x_1,x_2k)^2+d_2(y_1,y_2)^2+2(s_1-s_2)^2}\]just depends on $\ov{d_1}(x_1K,x_2K)$ and $d_2(y_1,y_2)$, for the induced metric $\ov{d_1}$ on $\B(G^{ad},\Q_p)/K$ defined in the same way as $\ov{d}$, and the metric $d_2$ on $\B(J^{ad}_b,\Q_p)$.
\begin{proposition}
There exists a constant $c>0$, which depends only on the locally finite cell decomposition of $\M_K$, such that for any $[T_1,g_1],[T_2,g_2]\in \I_K$, if $\ov{d}([T_1,g_1],[T_2,g_2])>c$ then we have
\[\D_{[T_1,g_1],K}\bigcap\D_{[T_2,g_2],K}=\emptyset.\]
\end{proposition}
\begin{proof}
We need to prove that, there exists a constant $c>0$, such that for any $[T,g]\in\I_K$, and any $[T',g']\in\{[T',g']\in\I_K|\D_{[T',g'],K}\bigcap\D_{[T,g],K}\neq\emptyset\}$, we have $\ov{d}([T,g],[T',g'])\leq c$.
This is in fact implicitly contained in last section. We just indicate some key points here. First, if $V_K\subset\B$ is any fixed fundamental domain for the action of $K$, then $\forall x,y\in V_K$ by definition we have $\ov{d}(xK,yK)\leq d(x,y)$. Next, in the proof of proposition 10.2 (and also in remark 10.7), we see the $T\in G(\Z_p)\setminus G(\Q_p)/G(\Z_p)$ such that $\D\bigcap T.\D\neq \emptyset$ corresponds to elements $(a_1,\dots,a_n)\in X_\ast(A)_+\subset\Z^n_+$ such that $a_i\leq C$ for $i=1,\dots,n$ and $C$ is a constant independent of $T$. Also by results of Vollaard-Wedhorn which we reviewed in section 7, the vertices $\Lam'=g'\Lam\in\B(J^{der}_b,\Q_p)^0$ such that $\D\bigcap g'\D\neq \emptyset$ satisfy $\Lam'\bigcap\Lam\neq \emptyset$, i.e. they share some common neighborhood, therefore there exists some constant $C'$ independent of $\Lam,\Lam'$ such that $d_2(\Lam,\Lam')\leq C'$ for the metric $d_2$ on $\B(J^{der}_b,\Q_p)$. Then one can easily deduce the proposition for the case $K=G(\Z_p)$, and the general case will follow as soon as this case holds.
\end{proof}

From now on we assume $\gamma=(h,g)\in G(\Q_p)\times J_b(\Q_p)$ such that both $h$ and $g$ are regular elliptic semi-simple, $hKh^{-1}=K$ and $v_p(deth)+v_p(detg)=0$. Recall the $\gamma$-fixed vertices $(\B^0)^\gamma$ is non empty (cf. \cite{SS}), we fix a choice of $\gamma$-fixed vertex $\wh{o}$, and let $o$ be its image in $\B^0/K$. We can take a choice $\wh{o}\in\varphi^{-1}(0)^0$ so that $o\in\varphi^{-1}(0)^0/K$ (see the above $\varphi$). Then $\gamma(o)=o$ for the induced action $\gamma: \B^0/K\ra\B^0/K$.
For any $\rho>0$, we consider the subset in $\I_K^0$ defined by intersection of $\I_K^0$ with the closed balls of radius $\rho$ with center $o$ in $\B^0/K$:
\[A_\rho=\{x\in\B^0/K|\;\ov{d}(o,x)\leq\rho\}\bigcap\I_K^0,\]which is a finite set, and $\gamma(A_\rho)=A_\rho$ since $\gamma(o)=o$, $\ov{d}$ is $\gamma$-isometric, and $\gamma(\I_K^0)=\I_K^0$. For any finite set $A\subset \I_K^0$, we consider
\[\begin{split}&V_A=\bigcup_{[T,g']\in A}\D_{[T,g'],K},\\
&U_A=\M_K^0-\bigcup_{[T,g']\notin A}\D_{[T,g'],K}.\end{split}\]

\begin{proposition} With notations as above,
$U_{A}$ is an open subspace of $\M_K^0$, $V_{A}$ is a compact analytic domain of $\M_K^0$, and we have $U_{A}\subset V_{A}$.
\end{proposition}
\begin{proof}
By our assumption in this section, $\D_K$ is compact, thus it is clear that $V_{A}$ as a finite union of compact analytic domains, is still a compact analytic domain. Since the union of closed subsets $\bigcup_{[T,g']\notin A}\D_{[T,g],K}$ is locally finite, one can check easily it is closed in $\M_K^0$. Thus $U_{A}$ is open. Finally the inclusion $U_{A}\subset V_{A}$ holds since $\M_K^0=(\M_K^0-U_{A})\bigcup V_{A}$ by the cell decomposition of $\M_K^0$.
\end{proof}

For any $\rho>0$, denote \[U_\rho=U_{A_\rho},\; V_\rho=V_{A_\rho}.\] Then \[\gamma(U_\rho)=U_\rho,\gamma(V_\rho)=V_\rho\]since $\gamma(A_\rho)=A_\rho$. By the above proposition, each $U_\rho$ is a locally compact open subspace of $\M_K^0$. We have
\[H^j_c(\M_K^0\times \Cm_p, \ov{\Q}_l)=\varinjlim_{\rho}H^j_c(U_\rho\times\Cm_p,\ov{\Q}_l),\]with
\[dim_{\ov{\Q}_l}H_c^j(U_\rho\times\Cm_p,\ov{\Q}_l)<\infty.\] We have an induced action of $\gamma$ on cohomology of $U_\rho$
\[\gamma: H_c^j(U_\rho\times\Cm_p,\ov{\Q}_l)\ra H_c^j(U_\rho\times\Cm_p,\ov{\Q}_l).\]
Denote \[H_c^\ast(U_\rho\times\Cm_p,\ov{\Q}_l)=\sum_{j\geq0}(-1)^jH_c^j(U_\rho\times\Cm_p,\ov{\Q}_l),\]as an element in some suitable Grothendieck group. We consider the trace of $\gamma$ on this Euler-Poincar\'e characteristic
\[Tr(\gamma|H_c^\ast(U_\rho\times\Cm_p,\ov{\Q}_l))=\sum_{j\geq0}(-1)^jTr(\gamma|H_c^j(U_\rho\times\Cm_p,\ov{\Q}_l)).\]
We consider also
\[H_c^\ast(\M_K^0\times\Cm_p,\ov{\Q}_l)=\sum_{j\geq0}(-1)^jH_c^j(\M_K^0\times\Cm_p,\ov{\Q}_l),\]
as an element in the Grothendieck group of smooth representations of $(G(\Q_p)\times J_b(\Q_p))^1$.
Let $\textrm{Fix}(\gamma|\M_K^0\times\Cm_p)$ be the set of fixed points of $\gamma$ on $\M_K^0\times\Cm_p$, then each fixed point is simple since the $p$-adic period mapping $\M_K^0\ra\Fm^a$ is \'etale, and the fixed points of $g$ on $\Fm^a$ are all simple.

We will use our result of cell decomposition of $\M_K^0$, to verify that the action of $\gamma$ satisfies the conditions of Mieda's theorem 3.13 in \cite{Mi2}, thus deduce a Lefschetz trace formula in our case. In fact we will use a Berkovich version of loc. cit. Recall that, if $k$ is a complete non-archimedean field and $k^0$ is its ring of integers, then the category of Hausdorff strictly $k$-analytic spaces is equivalent to the category of adic spaces which are taut and locally of finite type over $spa(k,k^0)$, see \cite{Hu} chapter 8. If $X$ is a Hausdorff strictly $k$-analytic space, we denote by $X^{ad}$ the associated adic space, which is taut and locally of finite type over $spa(k,k^0)$.

\begin{theorem}
For the fixed $\gamma=(h,g)\in G(\Q_p)\times J_b(\Q_p)$ such that both $h$ and $g$ are regular elliptic, $v_p(deth)+v_p(detg)=0$, there exist a sufficient small open compact subgroup $K'\subset G(\Z_p)$ and a sufficient large number $\rho_0>>0$, such that for all open compact subgroup $K\subset K'$ which is normalized by $h$ and all $\rho\geq \rho_0$,
we have the Lefschetz trace formula
\[Tr(\gamma|H_c^\ast(U_\rho\times\Cm_p,\ov{\Q}_l))=\#\textrm{Fix}(\gamma|\M_K^0\times\Cm_p),\]which is well defined and finite. Since the right hand side is independent of $\rho$, we can define
\[Tr(\gamma|H_c^\ast(\M_K^0\times\Cm_p,\ov{\Q}_l)):=Tr(\gamma|H_c^\ast(U_\rho\times\Cm_p,\ov{\Q}_l))\]for $\rho>>0$, and thus \[Tr(\gamma|H_c^\ast(\M_K^0\times\Cm_p,\ov{\Q}_l))=\#\textrm{Fix}(\gamma|\M_K^0\times\Cm_p).\]
\end{theorem}
\begin{proof}
Since $h$ is elliptic, for sufficiently small open compact subgroup $K\subset G(\Z_p)$ such that $hKh^{-1}=K$, we have the following claim
 \[\ov{d}(x, \gamma x)\ra \infty,\;\textrm{when}\;x\in\I_K^0, \ov{d}(o,x)\ra \infty.\]

In fact, since $o,x\in \varphi^{-1}(0)^0/K$, write $o=[o_1K,-s,s,o_2], x=[x_1K,-t,t,x_2]$ with $o_1,x_1\in \B(G^{ad},\Q_p), o_2,x_2\in \B(J^{ad}_b,\Q_p),s,t\in \{0\} (n\; odd)\; \textrm{or}\;\{0,1\} (n\;even)$, then \[\gamma(x)=[x_1hK,\frac{2}{n}v_p(deth)-t,\frac{2}{n}v_p(detg)+t,gx_2]=[x_1hK,-t',t',gx_2]\] with $t'\in \{0\} (n\; odd)\; \textrm{or}\;\{0,1\} (n\;even)$,
\[\begin{split}\ov{d}(x,\gamma x)&=\inf_{k\in K}\sqrt{d_1(x_1,x_1hk)^2+d_2(x_2,gx_2)^2+2(t-t')^2},\\
\ov{d}(o,x)&=\inf_{k\in K}\sqrt{d_1(x_1k,o_1)^2+d_2(x_2,o_2)^2+2(t-s)^2}
.\end{split}\]
To prove the above statement, we first work with $\B':=\B(G^{ad},\Q_p)\times \B(J^{ad}_b,\Q_p)$. Denote the metric on $\B'$ by $d'$. Since $h, g$ are elliptic, the fixed points set $(\B')^{\gamma}$ is nonempty and compact. Moreover, for $K$ sufficiently small, $(\B')^{(h,g)}=(\B')^{(h',g)}$ for any $h'\in hK$ (cf. the proof of lemma 12 in \cite{SS}). For $o'=(o_1,o_2)\in (\B')^{(h,g)}$ fixed, a simple triangle inequality shows that $d'(x',(\B')^{(h,g)})\ra\infty$ when $x'=(x_1,x_2)\in \B', d'(o',x')=\sqrt{d_1(x_1,o_1)^2+d_2(x_2,o_2)^2}\ra\infty$, since $(\B')^{(h,g)}$ is compact. On the other hand, for any automorphism $\sigma$ of $\B'$ with $(\B')^\sigma\neq\emptyset$, there exists a constant $0<\theta\leq \pi$ which just depends on $\B'$ and $\sigma$, such that \[d'(x',\sigma x')\geq 2d'(x',(\B')^\sigma)\sin(\frac{\theta}{2}),\]
see \cite{R} proposition 2.3. In particular, $d'(x', \gamma'x')\ra\infty$ when $d'(o',x')\ra \infty$ for any $h'\in hK, \gamma'=(h',g)$. As $K$ is compact this deduces the above statement.

We have
 \[\begin{split}\M_K^0-U_\rho&=\bigcup_{[T,g']\in \I_K^0-A_\rho}\D_{[T,g'],K}\\
 V_\rho-U_\rho&=\bigcup_{[T,g']\in A_\rho-A_{\rho-c}}F_{T,g'},\end{split}\]  where for any $[T,g']\in A_\rho$,
 \[F_{T,g'}=\D_{[T,g'],K}\bigcap(\M_K^0-U_\rho),\]which is nonempty only if $[T,g']\in A_\rho-A_{\rho-c}$ by proposition 11.1,
 in which case
 $F_{T,g'}$ is a compact analytic domain in $\D_{[T,g'],K}\subset V_\rho$. By the above claim, there exists a sufficiently large $\rho_0>>0$, such that for any $\rho\geq \rho_0,[T,g']\in \I_K^0-A_{\rho-c}$, we have $\ov{d}([T,g'],\gamma([T,g']))> c$, and thus by proposition 11.1
 \[\D_{[T,g'],K}\bigcap \gamma(\D_{[T,g'],K})=\emptyset,\;
 F_{T,g'}\bigcap \gamma(F_{T,g'})=\emptyset \;(\textrm{for}\;[T,g']\in A_\rho-A_{\rho-c}).\]

 To apply Mieda's theorem, we pass to adic spaces. We have the locally finite cell decomposition
 \[(\M_K^0)^{ad}=\bigcup_{[T,g']\in\I_K^0}\D_{[T,g'],K}^{ad},\]
 where each cell $\D_{[T,g'],K}^{ad}$ is a quasi-compact open subspace of $(\M_K^0)^{ad}$, $\D^{ad}_{[T_1,g_1],K}\bigcap\D^{ad}_{[T_2,g_2],K}\neq \emptyset \Leftrightarrow \D_{[T_1,g_1],K}\bigcap\D_{[T_2,g_2],K}\neq \emptyset$, and the action of $\gamma$ on $(\M_K^0)^{ad}$ induces the action of the cells in the same way as the case of Berkovich spaces. By \cite{Hu1} 8.2, $U_\rho^{ad}$ is an open subspace of $(\M_K^0)^{ad}$, which is separated, smooth, and partially proper. On the other hand $V_\rho^{ad}=\bigcup_{[T,g']\in A_\rho}\D^{ad}_{[T,g'],K}$ is a quasi-compact open subspace. Consider the closure $\ov{V_\rho^{ad}}=\bigcup_{[T,g']\in A_\rho}\ov{\D^{ad}_{[T,g'],K}}$ of $V_\rho^{ad}$ in $(\M_K^0)^{ad}$, which is a proper pseudo-adic space and contained in the quasi-compact space $V_\rho^{ad}$. We know that $\ov{V_\rho^{ad}}$ (resp. $\ov{\D^{ad}_{[T,g'],K}}$) is the set of all the specializations of the points in $V_\rho^{ad}$ (resp. $\D^{ad}_{[T,g'],K}$). Moreover, $\gamma$ induces automorphisms $\gamma: \ov{V_\rho^{ad}}\ra\ov{V_\rho^{ad}}, V_\rho^{ad}\ra V_\rho^{ad}$, $U_\rho^{ad}\ra U_\rho^{ad}$. Since $V_{\rho-c}^{ad}\subset U_\rho^{ad}\subset V_\rho^{ad}$, we have $\ov{V_\rho^{ad}}-V_\rho^{ad}=\bigcup_{[T,g']\in A_\rho-A_{\rho-c}}(\ov{\D^{ad}_{[T,g'],K}}-\D^{ad}_{[T,g'],K})$. Note \[ \D^{ad}_{[T_1,g_1],K}\bigcap \D^{ad}_{[T_2,g_2],K}\neq\emptyset
 \Leftrightarrow \ov{\D^{ad}_{[T_1,g_1],K}}\bigcap \ov{\D^{ad}_{[T_2,g_2],K}}\neq\emptyset.\](The direction $\Rightarrow$ is trivial; for the direction $\Leftarrow$, just note that the set of generalizations of a point is a totally ordered chain.)
 For $[T,g']\in A_\rho-A_{\rho-c}$, let $W_{T,g'}=\ov{\D^{ad}_{[T,g'],K}}-\D^{ad}_{[T,g'],K}$.
 By the paragraph above, for $\rho>>0$ we have $\gamma(W_{T,g'})\bigcap W_{T,g'}=\emptyset$. One sees the conditions of theorem 3.13 in \cite{Mi2} for $V_\rho^{ad}$ and its compactification $\ov{V_\rho^{ad}}$ hold. For the convenience of the audience, we recall these conditions in our setting: for all $x\in \ov{V_\rho^{ad}}-V_\rho^{ad}$, there exist closed constructible subsets $W_1, W_2$ of $\ov{V_\rho^{ad}}$ such that $x\in W_1, \gamma(x)\in W_2$ and $W_1\bigcap W_2=\emptyset$. Hence we get
  \[Tr(\gamma|H_c^\ast(V^{ad}_\rho\times\Cm_p,\ov{\Q}_l))=\#\textrm{Fix}(\gamma|V_\rho^{ad}\times\Cm_p)
  =\#\textrm{Fix}(\gamma|V_\rho\times\Cm_p).\]By \cite{Hu1} proposition 2.6 (i) and lemma 3.4, we have \[Tr(\gamma|H_c^\ast(V_\rho^{ad}\times\Cm_p,\ov{\Q}_l))=Tr(\gamma|H_c^\ast(U_\rho^{ad}\times\Cm_p,\ov{\Q}_l))+
  Tr(\gamma|H_c^\ast((V_\rho^{ad}-U_\rho^{ad})\times\Cm_p,\ov{\Q}_l)).\] Since $V^{ad}_\rho-U_\rho^{ad}=\bigcup_{[T,g']\in {A_\rho-A_{\rho-c}}}F_{[T,g']}'$, where $F_{[T,g']}'\subset F_{[T,g']}^{ad}$ is some locally closed subset, by the paragraph above, $F_{[T,g']}'\bigcap\gamma(F_{[T,g']}')=\emptyset$. As $A_\rho-A_{\rho-c}$ is a finite set, there are only finitely many orbits for the action of $\gamma$. For the union of the subspaces $F_{[T,g']}'$ over an orbit of $\gamma$, the trace of $\gamma$ on its cohomology is 0. Then one can repeat by the above argument to get that $Tr(\gamma|H_c^\ast((V_\rho^{ad}-U_\rho^{ad})\times\Cm_p,\ov{\Q}_l))=0$.
  By the comparison theorem on compactly support cohomology of Berkovich spaces and adic spaces (proposition 8.3.6 of \cite{Hu1}), we can conclude
 \[Tr(\gamma|H_c^\ast(U_\rho\times\Cm_p,\ov{\Q}_l))=Tr(\gamma|H_c^\ast(U^{ad}_\rho\times\Cm_p,\ov{\Q}_l))
 =Tr(\gamma|H_c^\ast(V^{ad}_\rho\times\Cm_p,\ov{\Q}_l))=
 \#\textrm{Fix}(\gamma|V_\rho\times\Cm_p).\]
As the reason above, for $\rho>>0$, $\gamma$ permutes the cells $\D_{[T,g'],K}$ for $[T,g']\notin A_\rho$, we have \[\textrm{Fix}(\gamma|V_\rho\times\Cm_p)=\textrm{Fix}(\gamma|\M_K^0\times\Cm_p).\]
 The theorem is thus proved.
 \end{proof}

\begin{remark}In fact we have
\[H_c^i(\M_K^0\times\Cm_p,\ov{\Q}_l)\simeq H_c^i((\M_K^0)^{ad}\times\Cm_p,\ov{\Q}_l)=\varinjlim_{\rho}H_c^i(V_\rho^{ad}\times\Cm_p,\ov{\Q}_l),\;\forall\,i\geq0,\]
here the second equality comes from proposition 2.1 (iv) in \cite{Hu1}.
We can work totally in the framework of adic spaces when considering cohomology. But here we have chosen to transfer back the result to Berkovich spaces, so we insist on working with the open subspaces $U_\rho$.
\end{remark}

We have a nice formula for the number of fixed points for the quotient space $\M_K/p^\Z$. Note if $g\in J_b(\Q_p)$ is a regular elliptic semi-simple element, for any $x\in \textrm{Fix}(g|\Fm^a(\Cm_p))$, there is an element $h_{g,x}\in G(\Q_p)$ which is conjugate to $g$ over $\ov{\Q}_p$ defined by the comparison isomorphism
\[V_p(H_y)\otimes_{\Q_p}B_{dR}\st{\sim}{\longrightarrow}V_L\otimes_LB_{dR},\]
where $y\in\pi^{-1}(x)$ is any point in the fiber of the $p$-adic period mapping $\pi:\M\ra\Fm^a$. $h_{g,x}$ is well defined, and it does not depend on the choice of $y\in\pi^{-1}(x)$. We remark that for different $x_1,x_2\in \textrm{Fix}(g|\Fm^a(\Cm_p))$, $h_{g,x_1}$ and $h_{g,x_2}$ may be in the same conjugacy class in $G(\Q_p)$.
 \begin{corollary}
Let the notations be as in the above theorem. If $n$ is even we assume that $\frac{2}{n}(v_p(deth)+v_p(detg))$ is even. Fix compatible Haar measures on $G(\Q_p)$ and the centralizer of $h_{g,x}$, $G_{h_{g,x}}:=\{h'\in G(\Q_p)|h'h_{g,x}h'^{-1}=h_{g,x}\}$. Denote the characteristic function of $h^{-1}K$ by $1_{h^{-1}K}$ and the volume of $K$ under the fixed Haar measure by $Vol(K)$. Then we have the following formula
\[Tr(\gamma|H_c^\ast((\M_K/p^{\Z})\times\Cm_p,\ov{\Q}_l))=\sum_{x\in \textrm{Fix}(g|\Fm^a(\Cm_p))}Vol(G_{h_{g,x}}/p^{\Z})O_{h_{g,x}}(\frac{1_{h^{-1}K}}{Vol(K)}),\]
where $Vol(G_{h_{g,x}}/p^\Z)$ is the volume of $G_{h_{g,x}}/p^\Z$ by the induced Haar measure on $G(\Q_p)/p^\Z$,
\[O_{h_{g,x}}(\frac{1_{h^{-1}K}}{Vol(K)})=\int_{G(\Q_p)/G_{h_{g,x}}}\frac{1_{h^{-1}K}}{Vol(K)}(z^{-1}h_{g,x}z)dz\] is the orbit integral of $\frac{1_{h^{-1}K}}{Vol(K)}$ over the conjugate class of $h_{g,x}$.
\end{corollary}
\begin{proof}
We just need count the number of the fixed geometric points set $\textrm{Fix}(\gamma|(\M_K/p^\Z)\times\Cm_p)$. This can be done by considering the map  \[\M_K/p^\Z\ra \Fm^a\subset\mathbf{P}^{n-1,an}\]induced the $p$-adic period mapping,
as in theorem 2.6.8 of \cite{St} and \cite{Mi4}. In particular we have
\[\#\textrm{Fix}(\gamma|(\M_K/p^\Z)\times\Cm_p)=\sum_{x\in \textrm{Fix}(g|\Fm^a(\Cm_p))}\#\{h'\in G(\Q_p)/p^\Z K|h^{'-1}h_{g,x}h'=h^{-1}\}.\]
One can then write this number easily in the form as in the corollary.
\end{proof}

For a supercuspidal representation $\pi$ of $G(\Q_p)$, we consider
\[H(\pi)=\sum_{j\geq 0}(-1)^j\textrm{Hom}_{G(\Q_p)}(\varinjlim_{K}H_c^j(\M_K\times\Cm_p,\ov{\Q}_l),\pi).\]
Assume that $\textrm{Hom}_{G(\Q_p)}(\varinjlim_{K}H_c^j(\M_K\times\Cm_p,\ov{\Q}_l),\pi)$ is of finite length for each $j\geq 0$, which should be always the case, then $H(\pi)$ is a well defined element in $\textrm{Groth}_{\ov{\Q}_l}(J_b(\Q_p))$.
\begin{corollary}
 Let $\pi$ be a supercuspidal representation of $G(\Q_p)$, $g\in J_b(\Q_p)$ be a regular elliptic semi-simple element. Assume that $\pi$ is of the form $\pi=c-\textrm{Ind}_{K_\pi}^{G(\Q_p)}\lambda$, for some open compact modulo center subgroup $K_{\pi}\subset G(\Q_p)$ and some finite dimensional representation $\lambda$ of $K_\pi$.
  Then we have
\[tr_{H(\pi)}(g)= \sum_{x\in \textrm{Fix}(g|\Fm^a(\Cm_p))} tr_{\pi}(h_{g,x}).\]
\end{corollary}
\begin{remark}Conjecturally, all the supercuspidal representations of $G(\Q_p)$ are of the above form $c-\textrm{Ind}_{K_\pi}^{G(\Q_p)}\lambda$, see for example \cite{Ste}.
\end{remark}
\begin{proof}
One computes exactly as in \cite{Mi4} or the proof of theorem 4.1.3 in \cite{St}, using theorem 11.3 and corollary 11.5. Here we just indicate a point when using the method of \cite{Mi4}. As the notations there, let $T=p^\Z$ considered as a subgroup of $G(\Q_p)$, and $K_0\vartriangleleft K_\pi$ be an open compact normal subgroup such that $\lambda$ is trivial on $TK_0$. Let $\Xi_\pi=K_\pi/TK_0$ and $\Xi^e_\pi$ be the subset of classes which contain an elliptic representative. We have a natural bijection $H_c^\ast((\M_{K_0}/T)\times\Cm_p,\ov{\Q}_l)\simeq H_c^\ast(\M_{K_0}^0\times\Cm_p,\ov{\Q}_l)$ if $n$ is odd, and $H_c^\ast((\M_{K_0}/T)\times\Cm_p,\ov{\Q}_l)\simeq \oplus_{i=0,1} H_c^\ast(\M_{K_0}^i\times\Cm_p,\ov{\Q}_l)$ if $n$ is even. Then (cf. \cite{K} and \cite{St})
\[\begin{split}tr_{H(\pi)}(g)&=tr(g|\textrm{Hom}_{K_\pi}(H_c^\ast((\M_{K_0}/T)\times\Cm_p,\ov{\Q}_l),\lambda))\\
&=\frac{1}{\#\Xi_\pi}\sum_{h\in \Xi_\pi^e}tr((h,g)|H_c^\ast(U_{\rho_h}\times\Cm_p,\ov{\Q}_l))tr(h^{-1}|\lambda)
\end{split}\]if $n$ is odd, and
\[=\frac{1}{\#\Xi_\pi}\sum_{h\in \Xi_\pi^e}tr((h,g)|\oplus_{i=0,1}H_c^\ast(U^i_{\rho_h}\times\Cm_p,\ov{\Q}_l))tr(h^{-1}|\lambda)\] if $n$ is even. The remaining computations are just using theorem 11.3 and corollary 11.5.
\end{proof}

 As \cite{St} and \cite{Mi4}, we will hope to use the above corollary to prove the realization of the local Jacquet-Langlands correspondence for smooth representations of $G(\Q_p)$ and $J_b(\Q_p)$ (for $n$ even) in the cohomology of our Rapoport-Zink spaces. By the method of Strauch in \cite{St} (and \cite{Mi4}), we are just reduced to problems of classification of $L$-packets of $G(\Q_p)$ and $J_b(\Q_p)$, and characterization of the local Jacquet-Langlands correspondence between smooth representations of them. Considering the recent progress on classification of $L$-packets (global and local) for unitary groups, for example see \cite{Mo} for the local case which concerns us, it seems that these can be achieved soon.

\end{document}